\documentclass[reqno]{amsart}
\usepackage{amsmath,amssymb}
\usepackage{graphicx,amssymb,xcolor}
\usepackage[mathscr]{euscript}
\usepackage{hyperref}
\usepackage[left=3cm, right=3cm, top=3cm, bottom=3cm]{geometry}
\newtheorem{theorem}{Theorem}[section]
\newtheorem{lemma}[theorem]{Lemma}
\newtheorem{proposition}[theorem]{Proposition}

\newtheorem{remark}[theorem]{Remark}

\newtheorem{consequence}[theorem]{Consequences}

\numberwithin{equation}{section}

\newcommand{\mc}[1]{{\mathcal #1}}

\newcommand{\mf}[1]{{\mathfrak #1}}
\newcommand{\mb}[1]{{\mathbf #1}}
\newcommand{\bb}[1]{{\mathbb #1}}

\newcommand{\R}{\mathbb R}

\newcommand{\N}{\mathbb N}
\newcommand{\Z}{\mathbb Z}

\let\L=\Lambda

\newcommand{\<}{\langle}
\renewcommand{\>}{\rangle}

\newcommand{\1}{\,\rlap{\small 1}\kern.13em 1}

\newcommand{\sqr}[2]{{\vcenter{\hrule height.#2pt%
                      \hbox{\vrule width.#2pt height#1pt\kern#1pt%
                            \vrule width.#2pt}%
                      \hrule height.#2pt}}}

\renewcommand{\limsup}{\mathop{\overline{\hbox{\rm lim}}}}

\newcommand{\wh}{\widehat}

\newcommand{\ws}{\wh\Sigma}

\newcommand{\mbe}{\mathbf{E}}

\renewcommand\footnotemark{}

\usepackage{mathtools}

\begin{document}
\title[boundary driven generalised contact process] 
{A boundary driven generalised contact process with exchange of particles:
Hydrodynamics in infinite volume}

\author[K. Kuoch]{Kevin Kuoch}  
\address{Kevin Kuoch
  \hfill\break\indent  Johann Bernoulli Institute for Mathematics and Computer Science,
  \hfill\break\indent University of Groningen, PO Box 407, 9700AK Groningen, The Netherlands
    }
\email{kevin.kuoch@gmail.com}

\author[M. Mourragui]{Mustapha Mourragui}  
\address{Mustapha Mourragui
  \hfill\break\indent LMRS, UMR 6085, Universit\'e de Rouen,
  \hfill\break\indent Avenue de l'Universit\'e, BP.12, Technop\^ole du
  Madril\-let, \hfill\break\indent
F76801 Saint-\'Etienne-du-Rouvray, France.}
\email{Mustapha.Mourragui@univ-rouen.fr}

\author[E. Saada]{Ellen Saada}  \address{Ellen Saada
  \hfill\break\indent CNRS, UMR 8145, Laboratoire MAP5, Universit\'e Paris Descartes,
  Sorbonne Paris Cit\'e,
  \hfill\break\indent 45 rue des Saints-P\`eres, 75270 Paris Cedex 06, France}
\email{Ellen.Saada@mi.parisdescartes.fr}

\date{\today}




\maketitle
\thispagestyle{empty}

\begin{abstract}
We consider a two species process which evolves
in a  finite or infinite  domain in contact with  particle  reservoirs 
at different densities,
according to the superposition of
a  generalised contact process and a rapid-stirring dynamics in the bulk of the domain,
and a creation/annihilation mechanism at its boundaries.
For this process, we study the law of large numbers for densities and 
current. The limiting equations are 
 given by a system of non-linear reaction-diffusion equations with Dirichlet boundary conditions.
\end{abstract}

\bigskip

\section{Introduction}\label{sec:intro}
In this paper, we consider the evolution on a lattice
of  two types of populations,  according to a
boundary driven generalised contact process with exchange of particles. This process is
the superposition of a  \textit{contact process with random slowdowns} (or CPRS) 
and a rapid-stirring dynamics in the bulk of the domain,
and a creation/annihilation mechanism at its boundaries, due to stochastic reservoirs.
\smallskip

\noindent
The CPRS was introduced in \cite{k} to model the \textit{sterile insect technique},
developed by E. Knipling and R. Bushland (see \cite{Kni55,DHB05})
in the fifties to control the New World screw worm, a serious threat to warm-blooded animals. 
This pest has been eradicated from the USA and Mexico only
in recent decades. The technique works as follows: Screw worms are reared in captivity 
and exposed to gamma rays. The male screw worms become sterile. If a sufficient number of sterile 
males are released in the wild then enough female screw worms are mated by sterile 
males so that the number of offspring will decrease generation after generation.
This technique is well suited for screw worms, because female apparently
mate only once in their lifetime; but it is also being tried for a large variety of pests,
including current projects to fight dengue in South America (Brazil, Panama).
\smallskip

\noindent
The particle system $(\eta_t)_{t\ge 0}$ we look at has state space
 $\{0,1,2,3\}^S$, for $S\subset\mathbb Z^d$ 
  (we refer to \cite{L} for interacting particle systems). 
Each site of $S$ is either empty (we say it is in state 0),
occupied by wild screw worms only (state 1), by sterile screw worms only (state 2),
or by wild and sterile screw worms together (state 3). On each site,
 we only keep track of the presence or not
of the type of the male screw worms (and not of their number), 
and we assume that enough female screw worms are around as not to limit mating.
\smallskip

\noindent
For the CPRS dynamics, 
we introduce a release rate $r$  and  growth rates $\lambda_1$, $\lambda_2$. 
A site gets sterile males at rate $r$ independently of everything else
(this corresponds to an artificial introduction of sterile males).
The rate at which wild males give birth (to wild males) on neighbouring sites
is $\lambda_1$ at sites in state 1, and $\lambda_2$ at sites in state 3.
Sterile males do not give birth.
 We  assume that $\lambda_2<\lambda_1$ to
reflect the fact that at sites in state 3 the fertility is decreased. 
Deaths for each type of  male screw worms 
 occur at all sites at rate 1, 
they are mutually independent.
\smallskip

\noindent
For a configuration $\eta$, the transitions of the CPRS at a site $x\in S$
are summarized as follows:
\begin{equation}\label{def:contact}
\begin{array}{l@{\extracolsep{1cm}}l}
0 \rightarrow 1 \ \mbox{ at rate } \lambda_1 n_1(x,\eta) + \lambda_2 n_3(x,\eta)
 & 1 \rightarrow 0 \ \mbox{ at rate } 1\\ 
0 \rightarrow 2 \ \mbox{ at rate } r & 2 \rightarrow 0 \ \mbox{ at rate } 1\\ 
1 \rightarrow 3 \ \mbox{ at rate } r & 3 \rightarrow 1 \ \mbox{ at rate } 1\\ 
2 \rightarrow 3 \ \mbox{ at rate } \lambda_1 n_1(x,\eta) + \lambda_2 n_3(x,\eta)
 & 3 \rightarrow 2 \ \mbox{ at rate } 1 
\end{array}
\end{equation}
where $n_i(x,\eta)$ is the number of nearest neighbours of $x$ in state $i$ for $i=1,3$.\\ 
In \cite{k}, a phase transition in $r$ is exhibited for the CPRS  in  $S=\mathbb Z^d$:
Assuming that $\lambda_2 \leq  \lambda_c < \lambda_1$, where $\lambda_c$ denotes the critical value of 
the $d$-dimensional basic contact process (see \cite{L} on the basic contact process), 
there exists a critical value $r_c$ such that 
 wild male screw worms (that is, states 1 and 3) 
survive for $r<r_c$, and die out for $r\ge r_c$.
\smallskip

\noindent
Our goal in the present paper is, for a given  infinite volume $S$ with boundaries
 (this expression seems puzzling, but
we typically think of a piece of land whose width is much smaller than its length; 
the latter is thought as infinite, and the former as having boundaries), 
to add to the above dynamics 
 displacements  within $S$,  as well 
as departures from $S$ and immigrations to $S$. 
We model them respectively by an exchange dynamics in the bulk, and by a creation/annihilation mechanism
at the boundaries of $S$ due to the presence of stochastic reservoirs. For the superposition
of the CPRS with these two dynamics,
we are interested in the evolution of the empirical densities and currents 
 corresponding to wild and sterile screw worms,  
for which we establish hydrodynamic limits.  
The limiting equations are given by systems of non-linear reaction-diffusion 
equations, with Dirichlet boundary conditions.  We also obtain hydrodynamic limits when
$S$ is either a finite volume 
in contact with reservoirs, or the infinite volume $\mathbb Z^d$. 
\smallskip

\noindent
Hydrodynamic limits investigate the macroscopic properties of interacting particle systems
(we refer to  the books \cite{sp1, kl} for a  	
thorough presentation). 
From a probabilistic point of view, it corresponds to a law of large numbers 
for the evolution of the spatial
density of particles in a given system. 
After the results of \cite{gpv, kov}, where
the intensive use of large deviation techniques led to a robust proof of the
hydrodynamic behaviour of a large class of finite volume gradient equilibrium systems, 
the method has been extended to
nonequilibrium systems in a bounded domain in \cite{fpv, els1, els2}, as well as
in an infinite volume without boundaries for conservative dynamics in \cite{fr1a,fr1b, fr1, y1, lm}.

\noindent
In the last years, many papers have been devoted to systems in contact with reservoirs 
in a bounded domain; we just quote a few of them, 
\cite{BDGJL06, bd, blm09, BL12} and references therein.
The nonequilibrium systems considered there were provided by lattice gas models
 also submitted to an external mechanism of creation and annihilation of particles,
modelling exchange reservoirs at the boundaries. 
Even though the stochastic dynamics describing the evolution in the bulk was conservative 
and in equilibrium,
the action of the reservoirs made the full process non reversible. 

\noindent
The hydrodynamic limit of a class of jump, birth and death processes 
has been studied in \cite{dfl1,dfl2}: the 
combination of  the Symmetric Simple Exclusion Process and of a Glauber
dynamics in dimension 1 to model the annihilation and creation of particles
led to a reaction-diffusion equation. 
The density and current large deviations have then been proved 
for the dynamics evolving on a torus and on a one-dimensional 
bounded interval in contact with reservoirs respectively  in \cite{jlv, BL12}; 
there, the lower bound of the large deviations principle was 
obtained only for smooth trajectories, and for birth and the death
rates that were monotone, concave functions.
\smallskip

\noindent
To our knowledge, the present paper  is the first work about 
hydrodynamics of an interacting particle system evolving in an 
 \textit{infinite} volume with boundary stochastic reservoirs
 and leading to \textit{a system}  of reaction-diffusion equations.
The natural questions that emerge after the hydrodynamic 
 limit are fluctuations and large deviations with respect to the
expected limit. Nonequilibrium fluctuations of interacting particle systems have only been
derived for few one-dimensional dynamics. It is one of the main open problems in the field. 
  We now plan to study  phase transitions, 
hydrostatics, and large deviations for our model.  
\smallskip

\noindent
Our set-up is the following. 
 The non-conservative dynamics that we consider 
 evolves in the infinite cylinder 
  \begin{equation}\label{def:LaN}
 \Lambda_N=\{-N,\cdots,N\}\times \mathbb Z^{d-1}.
 \end{equation} 
  In the bulk of $\Lambda_N$, particles evolve according to 
 the superposition of the CPRS  (which is non-conservative) and of an exchange dynamics
 (which is conservative). 
 The latter satisfies a detailed balance condition with respect to a family of Gibbs measures.
The reservoirs defining the movements of populations at the 
infinite boundary 
 \begin{equation}\label{def:GammaN}
\Gamma_N:=\{-N,N\}\times \mathbb Z^{d-1} 
\end{equation} 
   of $\Lambda_N$ are modelled by a reversible birth and death process.
 The full dynamics 
keeps fixed the value of the density at the boundary.
We have therefore to face the difficulty of non reversibility in the bulk 
combined with the fact that the stationary measures
are not explicitly known. 
\smallskip

\noindent
Our key tools to establish hydrodynamic limits will be first the analysis
of the specific entropy and the specific Dirichlet form \textit{in infinite volume}
 (to our knowledge, such an analysis is carried out
in infinite volume with reservoirs for the first time), 
then the use of couplings to derive hydrodynamics
by going from systems evolving 
in a large but finite volume to systems in infinite volume. 
 Here by a large finite volume, we mean the cylinder with a length $M_N$ 
(to be precised later on) large with respect to $N$. 
 The definition of  our dynamics 
will enable us to use basic coupling, which is not always possible
(we refer to \cite{gs, bo} for dynamics requiring more intricate couplings). 
 Finally we prove uniqueness of weak solutions to the limiting system of 
reaction-diffusion equations. 
\\ \\ 
 The paper is organized as follows.
In Section \ref{sec:2}, we detail our model.  In Section \ref{sec:results}, we 
state our results: 
the hydrodynamic limit of the boundary driven  process in infinite volume 
is stated in Theorem \ref{res:hydrodynamics}; 
we also state two related results, hydrodynamics in 
the full space $\mathbb Z^d$ in Theorem \ref{infvol:hydrodynamics}, and 
 in a  bounded space with boundary conditions in Theorem \ref{bdd:hydrodynamics}. 
Lastly the law of large numbers for the conservative and non-conservative currents 
is stated in Proposition \ref{res:lln_currents}.  For the proofs of 
Theorems \ref{infvol:hydrodynamics} and \ref{bdd:hydrodynamics} we refer to \cite{phd}. 
Sections \ref{sec:hydrodynamics} to \ref{sec:uniqueness} are devoted to the 
proof of Theorem \ref{res:hydrodynamics}: it is outlined 
in Section \ref{sec:hydrodynamics};  
Section \ref{sec:specific} deals with specific entropy and Dirichlet forms,
Section \ref{sec:hydro_fini} with hydrodynamics in large finite volume,
Section \ref{sec:hydro_infini}  with couplings to
derive the boundary conditions in infinite volume; 
uniqueness of solutions is proved in Section \ref{sec:uniqueness}. 
In Section \ref{sec:currents} we prove 
Proposition \ref{res:lln_currents}. 
\bigskip
\section{Description of the model}\label{sec:2}
\subsection{Equivalent formulations for configurations}\label{subsec:model}
Rather than studying directly the process $(\eta_t)_{t\ge 0}$ describing
the  evolution of states  $1,2,3$, we introduce another 
interpretation for the model.  
The  corresponding  configuration space is 
\begin{equation}\label{eq:state-space}
\widehat \Sigma_N:=\big(\{0,1\}\times\{0,1\} \big)^{\Lambda_N}.
\end{equation}
  (In the sequel, we shall denote with 
a ``hat'' everything related to  product spaces and to
vectors).
Elements of $\ws_N$ are denoted by $(\xi,\omega)$, where 
  $\xi$-particles represent the wild screw worms, while
$\omega$-particles represent the sterile ones.
The correspondence with $(\eta_t)_{t\ge 0}$ is given by the following relations: 
For $x\in\Lambda_N$,
\begin{equation}\label{eq:2.0}
\begin{array}{l@{\extracolsep{1cm}}l}
\eta(x)=0 \quad \Longleftrightarrow & (1-\xi(x))(1-\omega(x))=1\, , \\ 
\eta(x)=1 \quad \Longleftrightarrow & \xi(x)(1-\omega(x))=1\, ,\\
\eta(x)=2 \quad \Longleftrightarrow & (1-\xi(x))\omega(x)=1\, ,\\
\eta(x)=3 \quad \Longleftrightarrow & \xi(x)\omega(x)=1\, .\\
\end{array}
\end{equation}
In other words,
$\xi(x)=1$ (resp. $\omega(x)=1$) if wild (resp. sterile) screw worms  
are present on $x$. Both can be present, giving the state 3 for $\eta(x)$,
or only one of them, giving the states 1 or 2 for $\eta(x)$. 
 \noindent 
We may also express the correspondence \eqref{eq:2.0} by an application
from $\widehat \Sigma_N$ to $\{1,2,3\}^{\Lambda_N}$, that is, we write
\begin{equation}\label{eq:omxieta} 
\eta=\eta((\xi,\omega)),\quad\mbox{ where, for any  } x\in\Lambda_N,\quad \eta(x) = 2\omega(x)+\xi(x)\, . 
\end{equation}
 Moreover, in order to describe the evolution of densities for states 
1, 2, 3, we also define, for $x\in \Lambda_N$,
\begin{equation}\label{omega}
\begin{cases}
        \eta_1(x) = \xi(x)(1-\omega(x)) \equiv {\bf 1}_{\{\eta(x) =1\}} \, ,\\
        \eta_2(x) = (1-\xi(x))\omega(x) \equiv {\bf 1}_{\{\eta(x) =2\}}\, ,\\
        \eta_3(x) = \xi(x)\omega(x) \equiv {\bf 1}_{\{\eta(x) =3\}}\, .
  \end{cases}
\end{equation}
  By abuse of language, when $\eta_i(x) = 1$ for $i=1,2,3$,
we say that there is a particle of type $i$ at $x$. 
It is convenient to complete \eqref{omega} by defining 
 empty occupation of site $x$ by 
\begin{equation}\label{omega-compl}
 \eta_0(x) = (1-\xi(x))(1-\omega(x)) = {\bf 1}_{\{\eta(x) =0\}}
 = 1-\eta_1(x)-\eta_2(x)- \eta_3(x) \, .
\end{equation}
In view of going from finite to infinite volume, 
for each positive integer $n$, denote by 
\begin{equation}\label{La-Nn}
\L_{N,n} =\{-N,\cdots,N\}\times \{-n,\cdots,n\}^{d-1}
\end{equation}
the sub-lattice of size 
$(2N+1)\times (2n + 1)^{d-1}$ of $\Lambda_N$, and by
\begin{equation}\label{Si-Nn}
\widehat \Sigma_{N,n}=\big(\{0,1\}\times\{0,1\} \big)^{\L_{N,n}}
\end{equation}
the corresponding state space. 
\subsection{The infinitesimal generator}\label{subsec:generator} 
The boundary driven generalised contact process with exchange of particles 
 in $\Lambda_N$   is the Markov process on $\widehat\Sigma_N$ whose generator 
$\mf L_{N}\, :=\, \mf L_{\lambda_1,\lambda_2,r,\widehat b,N}$ 
can be decomposed as  
\begin{equation}\label{eq:gen}
\mathfrak L_N \, :=\,  
N^2\mathcal L_{N} + \mathbb L_{N}\;+N^2\; L_{\widehat b,N} \, ,
\end{equation} 
with the generators $\mathcal L_N$ for the exchanges of particles, 
$\mathbb L_{N}$ for the CPRS, and $ L_{\widehat b,N}$ for the boundary dynamics.
We now detail those dynamics and their properties.  For the existence of the Markov process 
with generator $\mathfrak L_N$ in infinite volume, we refer to \cite[Chapter 1]{L},
since we consider a compact state space and bounded rates (defined below). Cylinder functions are
a core for the generator  $\mathfrak L_N$. \\ \\
$\bullet$ For the exchange dynamics, the action of $\mathcal L_N$ on cylinder functions 
$f:\widehat \Sigma_N\to \mathbb R$ is 
\begin{eqnarray}\label{def:genExch}
\mathcal{L}_{N} f(\xi,\omega) &=& \sum\limits_{k=1}^d \sum_{x,x+e_k\in\Lambda_N} 
\mathcal{L}^{x,x+e_k} f(\xi,\omega)\,\mbox{ with }\\\label{def:genExch-x}
\mathcal{L}^{x,x+e_k} f(\xi,\omega)
&=&
 f(\xi^{x,x+e_k},\omega^{x,x+e_k}) - f(\xi,\omega) \, ,
\end{eqnarray}
  where $(e_1,\ldots,e_d)$ denotes the canonical basis of $\mathbb R^d$,
  and  for any $\zeta \in \Sigma_N:=\{0,1\}^{\Lambda_N}$, $\zeta^{x,y}$ 
is the configuration obtained from $\zeta$  by
exchanging the occupation variables $\zeta (x)$ and $\zeta (y)$, i.e.\ 
\begin{equation*}
(\zeta^{x,y}) (z) := 
  \begin{cases}
        \zeta (y) & \textrm{ if \ } z=x\, ,\\
        \zeta (x) & \textrm{ if \ } z=y\, ,\\
        \zeta (z) & \textrm{ if \ } z\neq x,y\, .
  \end{cases}
\end{equation*}
Since $(\xi,\omega)\in \ws_N$, these exchanges correspond to
 jumps between sites $x$ and $y$ for $\xi$-particles and $\omega$-particles,
which do not influence each other. \\ \\
$\bullet$ We now define a family of \textit{invariant probability measures}   for $\mathcal L_N$, 
 which are  product, and parametrized  by three chemical potentials, 
since the exchange dynamics conserves, in each transition, the number of
particles of each type. 
  We denote by $\Lambda$ the macroscopic open cylinder 
 \begin{equation}\label{def:La} 
 \Lambda=(-1,1) \times \mathbb R^{d-1} \, .
 \end{equation} 
For a vector-valued function 
$\widehat m = (m_1, m_2, m_3) : \Lambda \to \mathbb R^3$, 
let $\bar \nu_{\widehat m(\cdot)}^N$ be the product measure on $\Lambda_N$ 
with varying chemical potential $\widehat m$
such that, for all positive integers $n$, the restriction 
$\bar \nu_{\widehat m(\cdot),n}^N$ of $\bar \nu_{\widehat m(\cdot)}^N$
to   $\widehat \Sigma_{N,n}$ 
is given by
\begin{equation}\label{eq:rest-bar-nu}
\bar \nu_{\widehat m(\cdot),n}^N (\xi,\omega) = \widehat Z_{\widehat m,n}^{-1} 
\exp \Big\{ \sum_{i=1}^3 \sum_{x\in \L_{N,n}} m_i(x/N) \eta_i(x) \Big\},
\end{equation}
where $\widehat Z_{\widehat m,n}$ is the normalization constant: 
\begin{equation}\label{eq:Z_mn}
\widehat Z_{\widehat m,n} = \prod_{x\in \L_{N,n}}\Big\{1+ \sum_{i=1}^3 \exp (m_i(x/N))\Big\}.
\end{equation} 
Notice that the family of measures 
$\big\{{\bar \nu}^N_{\widehat m}\, ,\ \ \widehat m\in\mathbb R^3\big\}$ 
with constant parameters is reversible with respect to the generator $\mathcal L_{N}$.
For $\widehat m\in \mathbb R^3$ and $1\le i\le 3$, let $\psi_i(\widehat m)$ 
be the expectation of $\eta_i(0)$ under ${\bar \nu}^N_{\widehat m}$:
$$
\psi_i(\widehat m)\, =\, \mbe^{\bar \nu_{\widehat m}^N} \big[\eta_i(0) \big]\, .
$$
Observe that the function $\Psi$ defined on $(0,+\infty)^3$ by 
$\Psi(\widehat m)=(\psi_1(\widehat m),\psi_2(\widehat m),\psi_3(\widehat m))$ is a bijection
from $(0,+\infty)^3$ to $(0,1)^3$. 
We shall therefore do a change of parameter:
For every $\widehat \rho =(\rho_1,\rho_2,\rho_3)\in (0,1)^3$,
we denote by $\nu_{\widehat \rho}^N$ the product measure 
$\overline \nu_{\widehat m}^N$ such that $\Psi(\widehat m) = \widehat \rho$, hence
\begin{equation}\label{eq:rho_i}
\rho_i\, =\, \mbe^{\nu_{\widehat \rho}^N} \big[\eta_i(0) \big]\, , \quad i=1,2,3\, .
\end{equation}
{}From now on, we work with the representation $\nu_{\widehat \rho(\cdot),n}^N$ of the measure
 $\bar \nu_{\widehat m(\cdot),n}^N$.\\
 
 \noindent
 Finally, note that in view of the diffusive scaling limit, the generator $\mathcal L_{N}$  
has been speeded up by $N^2$ in \eqref{eq:gen}.  \\ \\
$\bullet$ According to \eqref{def:contact},
 the generator $\mathbb L_{N}:=\mathbb L_{N,\lambda_1, \lambda_2,r} $ 
 of the CPRS  is given by 
\begin{eqnarray}\label{def:genCP-DRE}
\mathbb L_{N} f(\xi,\omega) & =&  \sum\limits_{x \in \Lambda_N} \mathbb L_{\Lambda_N}^x f(\xi,\omega)
\quad\mbox{ where, for any }{\mathcal A}\subset \Z^d,\,x \in \mathcal A\\\label{def:genCP-DRE-x}
\mathbb L_{\mathcal A}^x f(\xi,\omega)& =&
\Big( r(1-\omega (x)) + \omega (x) \Big) \Big[ f(\xi,\sigma^x \omega) - f(\xi,\omega) \Big] \\\nonumber
 && \quad  + \Big( \beta_{\mathcal A}(x,\xi,\omega) 
 \big( 1 - \xi(x) \big)  + \xi(x) \Big) 
\Big[ f(\sigma^x \xi,\omega) - f(\xi,\omega) \Big]\, \quad\mbox{ with }\\
\label{def:ratesCP-DRE}
\beta_{\mathcal A}(x,\xi,\omega)&=& \lambda_1 \sum\limits_{y \in {\mathcal A}\atop\|y-x\|=1} 
\xi(y)(1-\omega(y)) + \lambda_2 \sum\limits_{y \in {\mathcal A}\atop\|y-x\|=1} \xi(y)\omega(y)\, ,
\end{eqnarray}
  where $\|\cdot\|$ denotes the norm in $\mathbb R^d$, $\|u\|=\sqrt{\sum_{i=1}^d|u_i|^2}$
 for $u\in\mathbb R^d$, 
  and for $\zeta \in \Sigma_N$, $\sigma^x \zeta $ is the configuration obtained from
$\zeta$ by flipping the configuration at $x$, i.e.
\begin{equation*}
(\sigma^{x} \zeta) (z) := 
  \begin{cases}
        1-\zeta (x) & \textrm{ if \ } z=x\, ,\\
        \zeta (z) & \textrm{ if \ } z\neq x\, .
  \end{cases}
\end{equation*} 
The representation \eqref{eq:2.0} sheds light on the fact that \eqref{def:genCP-DRE}
corresponds to a contact process (the $\xi$-particles) in a dynamic random environment, namely
the $\omega$-particles. Indeed,  the $\omega$-particles
move on their own and are not influenced by $\xi$-particles, while 
$\xi$-particles have birth rates whose value depends on the presence or not of
$\omega$-particles. In \cite{k}  a variant of
the CPRS dynamics in a quenched 
random environment is also considered, with the $(\xi,\omega)$-formalism.
On the other hand, we noticed previously
that $\omega$-particles can also be considered as an
environment for the exchange dynamics.\\ \\
$\bullet$ We now turn to the dynamics at the boundaries of the domain. 
We denote respectively the closure and the boundary
of $\Lambda$  (see \eqref{def:La})  by 
\begin{equation}\label{def:Gamma}
\overline \Lambda = [-1,1]\times \mathbb R^{d-1},\qquad\Gamma
= \{(u_1,\dots , u_d)\in \overline \Lambda  : u_1 = \pm 1\}.
\end{equation} 
For a metric space $E$ and any integer $1\leq m \leq +\infty$, 
  denote by  $\mathcal C^m (\overline \Lambda; E)$
(resp. $\mathcal C^m_c (\Lambda; E)$) the space of
$m$-continuously differentiable functions on $\overline \Lambda$ 
(resp. with compact support in $\Lambda$) with values in $E$,
and by $\mathcal C (\overline \Lambda; E)$
(resp. $\mathcal C (\Lambda; E)$, $\mathcal C_c (\Lambda; E)$) the space of
continuous functions on $\overline \Lambda$ 
(resp. on $\Lambda$, with compact support in $\Lambda$) with values in $E$.\\ \\
Fix a positive function  $\widehat b = (b_1,b_2, b_3): \Gamma \to \bb R_+^3$. 
 Assume that there
exists a neighbourhood $V$ of $\overline\Lambda$  in $\bb R^d$  and a smooth function 
$\widehat\theta=(\theta_1,\theta_2,\theta_3) : V
\to (0,1)^3$ in   $\mathcal C^2(V;\mathbb R^3)$ such that  
\begin{equation}\label{eq:cC}
0<c^*\leq\min_{1\le i\le 3}|\theta_i|\leq\max_{1\le i\le 3}|\theta_i|\leq C^*< 1
\end{equation}
for two positive constants $c^*,C^*$, and such that the restriction 
of $\widehat\theta$ to $\Gamma$ is equal to $\widehat b$:
\begin{equation}\label{b-for-rev}
\widehat\theta (\cdot){\big\vert_{\Gamma}} =\; \widehat b (\cdot)\, .
\end{equation}
The boundary dynamics acts as a birth and death process on the boundary
 $\Gamma_N$ of $\Lambda_N$ (see \eqref{def:GammaN}) 
described by the generator $L_{\widehat b,N}$ defined by
\begin{eqnarray}\label{def:LbN}
L_{\widehat b,N} f(\xi,\omega)&  = & \sum_{x\in \Gamma_N}L_{\widehat b,N}^x f(\xi,\omega)\, ,
\quad\mbox{ where }\\\nonumber
L_{\widehat b,N}^x f(\xi,\omega)&  = & 
c_x\big(\widehat b(x/N), \xi,\sigma^x \omega\big)
 \Big[ f(\xi, \sigma^x \omega) - f(\xi, \omega) \Big]\\ \nonumber
& &  +  c_x\big(\widehat b(x/N), \sigma^x\xi, \omega\big)
\big[ f(\sigma^x\xi, \omega) - f(\xi, \omega) \big]\\ \label{def:LbN-x}
& &  +  c_x\big(\widehat b(x/N), \sigma^x\xi, \sigma^x\omega\big) 
\big[ f(\sigma^x\xi, \sigma^x\omega) - f(\xi, \omega) \big]\, ,
\end{eqnarray}
where the rates $c_x\big(\widehat b(x/N), \xi, \omega\big)$ are given 
for $x\in \Gamma_N$ and $(\xi,\omega)\in \widehat \Sigma_N$ by
\begin{eqnarray}\label{b_rates}
c_x\big(\widehat b(x/N), \xi, \omega \big) &=&
\sum_{i=0}^3   b_i(x/N)\eta_i(x)\, ,\\
 \text{ with }\quad b_0(x/N) &=& 1 - \sum_{i=1}^3b_i(x/N)\label{b0}
\end{eqnarray} 
where  
$\eta_i(x)$, $i\in\{0,1,2,3\}$ are defined in \eqref{omega}--\eqref{omega-compl}. 
In other words, according to $\widehat b(.)$, a site $x \in \Gamma_N$ 
goes from state $i\in\{0,1,2,3\}$ to state $j\in\{0,1,2,3\} \ (j\neq i)$ at rate $b_j(x/N)$ 
(see \eqref{Lb-alternative} below,  where this interpretation is used). \\ 
We have that by \eqref{b-for-rev},
the boundary dynamics is such that the measure 
$\nu_{\widehat \theta}^N$ is reversible with respect to the operator 
$L_{\widehat b ,N}$ (see Consequences \ref{csq:chg-var}\textit{(ii)}). 
 Note also that  the generator $L_{\widehat b,N}$  
has been speeded up by $N^2$ in \eqref{eq:gen}. 
%
\section{The results}\label{sec:results}
\subsection{Notation}\label{subsec:notation}
We fix $T>0$. 
We denote by $(\xi_t,\omega_t)_{t\in [0,T]}$ the Markov process on $\widehat \Sigma_N$ with generator
$\mf L_{N}$  defined in \eqref{eq:gen}.  Given a probability measure $\mu$ on $\widehat \Sigma_N$, 
the probability measure $\bb P^{N,\widehat b}_\mu$ on the path space 
$D([0,T], \widehat \Sigma_N)$, 
endowed with the Skorohod topology and the corresponding Borel $\sigma$-algebra, 
is the law of $(\xi_t,\omega_t)_{t\in [0,T]}$ with initial 
distribution  $\mu$. The  associated expectation
is  denoted by  $\bb E^{N,\widehat b}_\mu$.\\
 We denote by $(S_N^{\widehat b}(t))_{t\in [0,T]}$ the semigroup associated to the generator
$\mf L_N$. 
We shall denote by $\mu (t)$ the time evolution 
of the measure $\mu$ under the semigroup $S_N^{\widehat b}$:
$\mu (t) = \mu S_N^{\widehat b}(t) $.\\ 
We denote by $\mathcal M$ the space of finite signed measures on $\Lambda$,
endowed with the weak topology. For $m\in\mathcal M$ and a function 
$F\in\mathcal C (\Lambda; \mathbb R)$,
we let $\langle m, F\rangle$ be the integral of $F$ with respect to
$m$.  For each configuration $(\xi,\omega)\in\widehat \Sigma_N$, let
$\widehat\pi^N = \widehat \pi^N(\xi,\omega)=(\pi^{N,1},\pi^{N,2},\pi^{N,3})\in \mathcal M^3$, where for
$i\in\{1,2,3\}$, the positive measure $\pi^{N,i}$ is obtained by assigning mass $N^{-d}$ 
to each particle of type $i$  (cf. \eqref{omega}): 
\begin{equation*} 
\pi^{N,i} \, =\, N^{-d}\sum_{x \in \Lambda_N}\eta_i(x)\, \delta_{x/N}\; ,
\end{equation*}
where $\delta_{u}$ is the Dirac measure concentrated on $u$. 
For any  function 
$\widehat G = (G_1,G_2, G_3)\in \mathcal C(\Lambda;\mathbb R^3)$,
the integral of $\widehat G$ 
with respect to $\widehat \pi^N$, denoted by $\langle\widehat \pi^N, \widehat G\rangle $, is given by
\begin{equation*}
\langle\widehat \pi^N, \widehat G\rangle  =\, \sum_{i=1}^3 \langle\pi^{N,i}\, ,\, G_i \rangle\, .
\end{equation*}
We  also  denote by $\widehat \pi^N$   the map from $D([0,T],\widehat \Sigma_N)$  to 
 $D([0,T],\mathcal{M}^3)$  defined by 
$\widehat \pi^N (\xi_{\cdot},\omega_{\cdot})_t=  \widehat \pi^N (\xi_t,\omega_t)$ and 
we denote by  
$Q_{\mu}^{N,\widehat b}  =  \mathbb P^{N,\widehat b}_{\mu} \circ (\widehat \pi^N)^{-1} $ 
 the law of the process
$\big(\widehat \pi^N (\xi_t,\omega_t) \big)_{t\in [0,T]}$.\\
Let $\mathcal{M}_+^1$ be the subset of $\mathcal{M}$ of all positive measures
absolutely continuous  with respect to the Lebesgue measure
with positive density bounded by $1$:
\begin{equation*}
\mathcal{M}_+^1=\big\{\pi\in \mathcal {M}:\pi(du)=\rho(u)du \;\; \hbox{ and } \;\;
0\leq \rho(u)\leq 1\;  \hbox{ a.e.} \big\}\, .
\end{equation*}
Let $D([0,T],(\mathcal{M}_+^1)^3)$ be the set of right continuous  trajectories with left
limits with values in $(\mathcal{M}_+^1)^3$,  endowed with the Skorohod
topology and equipped with its Borel $\sigma-$ algebra. \\    
 For a metric space $E$, and  integers $1\leq m,k \leq +\infty$, we denote by  
$\mathcal C^{k,m}_c ([0,T]\times \overline \Lambda; E)$
(resp. $\mathcal C^{k,m}_{c,0} ([0,T]\times \overline \Lambda; E)$)
the space  of functions  from $[0,T]\times \overline\Lambda$ to $E$  
that are $k$-continuously differentiable in time and
$m$-continuously differentiable in space, with compact support in $[0,T]\times\overline \L$ (resp. 
and vanishing at the boundary $\Gamma$ of $\L$). Similarly, we define 
$\mathcal C^{k,m}_c ([0,T]\times \Lambda; E)$ to be the subspace of 
$\mathcal C^{k,m}_c ([0,T]\times \overline \Lambda; E)$
of functions with compact support in $[0,T]\times \L$. \\
For  $\widehat G = (G_1,G_2,G_3), \widehat H= (H_1,H_2,H_3) \in \big(L^2(\Lambda)\big)^3$,
$\langle \widehat G(\cdot),\widehat H(\cdot) \rangle$ is the scalar product:
$$
\langle \widehat G(\cdot), \widehat H(\cdot) \rangle 
= \sum\limits_{i=1}^3 \langle G_i (\cdot),H_i(\cdot) \rangle = \sum_{i=1}^3 \int_
\Lambda G_i(u)H_i(u) du \, .
$$ 
For a smooth function $G: [0,T]\times \L \to \R$, $\partial_s G(s,u)$ 
represents the partial derivative with respect to the time variable $s$ and
for $1\le j\le d,k\ge 1$, $\partial_{e_j}^k G(s,u)$ stands for the $k$-th partial derivative 
in the direction $e_j$ with respect to the space variable $u$. 
 The discrete gradient $\partial_{e_1}^N$ in the direction $e_1$ is defined, for
 $x,x+e_1 \in \Lambda_N$, $G:\Lambda \to \mathbb R$, by
$$
\partial_{e_1}^N G(x/N) = N\Big(G\Big(\frac{x+e_1}{N}\Big) - G\Big(\frac{x}{N}\Big)\Big)\, .
$$ 
The discrete Laplacian $\Delta_N$ and the Laplacian $\Delta$ are respectively defined for 
$G \in \mathcal C^2 (\Lambda ; \mathbb R)$, if $x,x\pm e_j\in\Lambda_N $ for $1\le j\le d$ and
 $u\in \Lambda\setminus \Gamma$, by 
\begin{align*}
& \Delta_N G(x/N) = N^2 \sum\limits_{j=1}^d \Big[ G\Big(\frac{x+e_j}{N}\Big)
 + G\Big(\frac{x-e_j}{N}\Big) - 2 G\Big(\frac{x}{N}\Big) \Big] \ \ \mbox{ and } \ \ 
 \Delta G (u)= \sum\limits_{j=1}^d \partial_{e_j}^2 G(u) .
\end{align*}
We have now all the material to state our results. 
\subsection{Hydrodynamics in infinite volume with reservoirs.}\label{subsec:hydro}
We first describe \textit{the hydrodynamic equations}. 
Let $\widehat \gamma = (\gamma_1,\gamma_2,\gamma_3) : \overline \Lambda \to [0,1]^3$ 
be a smooth initial profile, and denote by 
$\widehat \rho = (\rho_1,\rho_2,\rho_3) : [0,T] \times \overline \Lambda \to [0,1]^3$ 
a typical macroscopic trajectory.  We shall prove in Theorem 
 \ref{res:hydrodynamics} below that the macroscopic evolution of  the  local 
particle  density $\widehat \pi^N$  is described by  the following system of 
non-linear reaction-diffusion equations
\begin{equation}\label{f00}
\begin{cases}
\partial_t\widehat \rho  & = \;\;  \Delta \widehat\rho 
\,+\, \widehat F(\widehat \rho) \;\;\; \hbox {in } \Lambda \times (0,T) ,\\ 
\widehat \rho (0,\cdot)    & =  \;\; \widehat \gamma (\cdot)  \;\;\;  \hbox {in }  \Lambda ,\\
\widehat \rho(t,\cdot) |_\Gamma  & = \;\; \widehat b(\cdot) \;\;\;  \hbox{ for }\; 0\leq t\leq T \, ,
\end{cases}
\end{equation}
where $\widehat F=(F_1,F_2,F_3):[0,1]\to \mathbb R^3$ is given by
 \begin{equation}\label{eq:F}
\begin{cases}
F_1 (\rho_1,\rho_2,\rho_3) & = \;\;  2d (\lambda_1 \rho_1 + \lambda_2 \rho_3) \rho_0 
+ \rho_3 - \rho_1 (r+1)\, ,\\ 
F_2 (\rho_1,\rho_2,\rho_3)   & =  \;\; r \rho_0 + \rho_3 -2d (\lambda_1 \rho_1 
+ \lambda_2 \rho_3) \rho_2 - \rho_2  \, ,\\
F_3 (\rho_1,\rho_2,\rho_3)  & = \;\; 2d (\lambda_1 \rho_1 + \lambda_2 \rho_3) \rho_2 
+ r \rho_1 - 2 \rho_3   \, ,
\end{cases}
\end{equation}
where $\rho_0 = 1-\rho_1-\rho_2-\rho_3$. \\
\smallskip

\noindent
A weak solution of \eqref{f00} is a function 
$\widehat \rho(\cdot, \cdot ) =(\rho_1,\rho_2,\rho_3):[0,T]\times \Lambda\to [0,1]^3 $ 
satisfying (IB1), (IB2) and (IB3) below:
\begin{itemize} 
 \item[(IB1)] 
For any $i\in\{1,2,3\}$, $\rho_i \in L^\infty \left(( 0,T) \times \Lambda \right)$.
\item[(IB2)]   For any function $\widehat G(t,u)=\widehat G_t(u)=(G_{1,t}(u),G_{2,t}(u),G_{3,t}(u))$ in 
${\mathcal C}_{c,0}^{1,2}\big([0,T]\times \overline \Lambda ;\mathbb R^3\big)$,  
writing similarly the density  as $\widehat \rho_t(u) = \widehat \rho(t,u) $, we have
\begin{equation}
 \begin{aligned}
& \langle \widehat \rho_T(\cdot), \widehat G_T(\cdot)\rangle
-\langle \widehat \rho_0(\cdot), \widehat G_0(\cdot)\rangle -
\int_0^T ds \,  \langle \widehat \rho_s(\cdot),\partial_s \widehat G_s(\cdot)\rangle 
=\, \int_0^T ds \,  \langle \widehat \rho_s(\cdot),\Delta \widehat G_s(\cdot)\rangle \\
\, 
&\qquad \quad +\, \int_0^T ds \,  \langle \widehat F(\rho_s)(\cdot),\widehat G_s(\cdot)\rangle 
 -\; \sum_{i=1}^3 \int_0^T ds
\int_{\Gamma} \,  \text{\bf n}_1(r)\, b_i(r) (\partial_{e_1}G_{i,s})(r) \, dS(r) \, ,
 \end{aligned}
\end{equation}
where {\bf n}=$(\text{\bf n}_1,\ldots ,\text{\bf n}_d)$ stands for the
outward unit normal vector to the boundary surface $\Gamma$ and 
$\text{d} \text{S}$ for an element of surface on $\Gamma$. 
\item[(IB3)]  $\widehat \rho (0,u)= \widehat \gamma (u)$ a.e.
\end{itemize}
 This system of equations \eqref{f00} has a unique weak solution  (see Section \ref{sec:uniqueness}).
\medskip

\noindent
Our main result is:
\begin{theorem}\label{res:hydrodynamics}
 For each $N\ge 1$, let $\mu_N$ be a probability measure 
on $\widehat \Sigma_N$. 
The sequence of probability measures $(Q_{\mu_N}^{N,\widehat b})_{N\geq 1}$ 
is weakly relatively compact and all its converging
subsequences converge to some limit $Q^{\widehat b,*}$ which is concentrated 
on absolutely continuous paths  in $\mathcal C ([0,T],(\mathcal M_+^1)^3)$ whose density
$\widehat \rho$ satisfies (IB1) and (IB2).

Moreover, if for any $\delta >0$ and for  any  function 
 $\widehat G\in  \mathcal C_c(\overline \Lambda ;\mathbb R^3)$, 
\begin{equation}\label{cts_profile}
\lim_{N\to \infty}\mu_N\Big\{ \Big| \langle \widehat \pi_N(\xi,\omega), \widehat G(\cdot)\rangle\, -\, 
\langle \widehat \gamma (\cdot),\widehat G(\cdot)\rangle   \Big| \geq \delta \Big\}=0\, ,
\end{equation}
for an initial continuous profile $\widehat \gamma:\Lambda\to [0,1]^3$,  then
the sequence $(Q_{\mu_N}^{N,\widehat b})_{N\geq 1}$ 
converges to the Dirac measure  concentrated on the unique   weak  solution 
$\widehat \rho(\cdot,\cdot)$ of the boundary value 
problem \eqref{f00}. Accordingly, for any $t\in [0,T]$, 
any $\delta >0$ and any function 
$\widehat G \in {\mathcal C}_c^{1,2}\big([0,T]\times \overline \Lambda ;\mathbb R^3\big)$,
$$
\lim_{N\to \infty}\mathbb P_{\mu_N}^{N,\widehat b}\Big\{ \Big| 
\langle  \widehat \pi_N(\xi_t,\omega_t), \widehat G(\cdot)\rangle \, -\, 
\langle \widehat \rho_t(\cdot),\widehat G(\cdot)\rangle  
 \Big| \ge \delta \Big\}=0\, .
$$
\end{theorem}
\noindent
The proof of Theorem \ref{res:hydrodynamics} will be outlined  in Section \ref{sec:hydrodynamics}, 
and done in the following Sections \ref{sec:specific}, \ref{sec:hydro_fini}, \ref{sec:hydro_infini},
\ref{sec:uniqueness}.
%
  %
 \subsection{Hydrodynamics in $\mathbb Z^d$.}\label{subsec:hydro-Zd} 
Our approach enables us  to derive as well the hydrodynamic limit in the  full 
volume $\mathbb Z^d$.
There, the reaction-diffusion process $(\xi_t,\omega_t)_{t \in[0,T]}$ 
on $\big(\{0,1\}\times \{0,1\}\big)^{\mathbb Z^d}$ has generator
 $$N^2 \mathcal L + \mathbb L=  N^2 
 \sum\limits_{k=1}^d \sum_{x\in\mathbb Z^d} 
\mathcal{L}^{x,x+e_k}
+ \sum\limits_{x \in\mathbb Z^d} \mathbb L_{\mathbb Z^d}^x $$
(see \eqref{def:genExch-x}, \eqref{def:genCP-DRE-x}--\eqref{def:ratesCP-DRE}) 
with law ${\mathbb P}_{\mu_N}^N$ when the initial
distribution is $\mu_N$. It satisfies  the following theorem, proved  in \cite{phd}. 
\begin{theorem}\label{infvol:hydrodynamics}
Consider a sequence of probability measures $(\mu_N)_{N \geq 1}$ on  
$\big(\{0,1\}\times \{0,1\}\big)^{\mathbb Z^d}$ associated to a continuous profile
 $\widehat \gamma : \mathbb R^d \rightarrow [0,1]^3$, that is, for  any
 function $\widehat G= (G_1,G_2,G_3) \in \mathcal C_c (\mathbb R^d ; \mathbb R^3)$,
 %
$$\lim_{N\to \infty}\mu_N\Big\{ \Big| \langle \widehat \pi_N(\xi,\omega), \widehat G(\cdot)\rangle\, -\, 
\langle \widehat \gamma (\cdot),\widehat G(\cdot)\rangle   \Big| \geq \delta \Big\}=0\, ,$$
 %
for all $\delta >0$. 
Then for any $t\in [0,T]$, for  any function 
$\widehat G= (G_1,G_2,G_3) \in \mathcal C_c ( \mathbb R^d ; \mathbb R^3)$,
$$\lim\limits_{N \rightarrow \infty}
 {\mathbb P}_{\mu_N}^N \Big\{ \Big| 
\langle \widehat \pi_N(\xi_t,\omega_t), \widehat G(\cdot)\rangle\, -\, 
\langle \widehat \rho_t(\cdot),\widehat G(\cdot)\rangle  
 \Big| \ge \delta \Big\}=0$$
for all  $\delta >0$,  where $\widehat \rho(t,u)$ is the unique weak solution of the system 
$$
\left\{ 
\begin{array}{ll}
\partial_t \widehat \rho = \Delta \widehat \rho 
+ \widehat F (\widehat \rho) \ \mbox{ in } \ \mathbb Z^d \times (0,T), \\
\widehat \rho(0,\cdot) = \widehat \gamma (\cdot) \ \mbox{ in }  \ \mathbb Z^d,
\end{array}
\right. 
$$
that is,  $\widehat \rho(\cdot, \cdot )$ satisfies the following assertions:
\begin{itemize}
\item[(I1)] For any $i\in\{1,2,3\}$, $\rho_i \in L^\infty \big([0,T] \times \mathbb R^d \big)$,
\item[(I2)] For any function $\widehat G_t(u) = \widehat G(t,u)$ 
in  ${\mathcal C}_{c}^{1,2}\big([0,T]\times \mathbb R^d ;\mathbb R^3\big)$, we have 
$$
 \begin{multlined}
 \langle \widehat \rho_T(\cdot), \widehat G_T(\cdot)\rangle
-\langle \widehat \rho_0(\cdot), \widehat G_0(\cdot)\rangle -
\int_0^T ds \,  \langle \widehat \rho_s(\cdot),\partial_s \widehat G_s(\cdot)\rangle \\
=\, \int_0^T ds \,  \langle \widehat \rho_s(\cdot),\Delta \widehat G_s(\cdot)\rangle 
\, +\, \int_0^T ds \,  \langle \widehat F(\rho_s)(\cdot),\widehat G_s(\cdot)\rangle \, ,
 \end{multlined}
$$
\item[(I3)]  $\widehat \rho (0,u)= \widehat \gamma (u)$ a.e.
\end{itemize}
\end{theorem}
\smallskip
\subsection{Hydrodynamic limit in finite volume with reservoirs.}
As Theorem \ref{res:hydrodynamics} deals with an infinite bulk, 
we are consequently able to derive the limit in a finite volume 
in contact with reservoirs. We denote
$B_N = \{-N,\cdots,N\}\times\mathbb T^{d-1}_N$, 
$B=(-1,1)\times\mathbb T^{d-1}$  and its closure 
$\overline B=[-1,1]\times\mathbb T^{d-1}$, 
 where $\mathbb T^{d}_N$ is the 
$d$-dimensional microscopic torus of length $N$  and $\mathbb T^d$ 
is the $d$-dimensional torus $[0,1)^d$.  
In finite volume with stochastic reservoirs, the reaction-diffusion 
process $(\xi_t,\omega_t)_{t \in[0,T]}$ on 
 $\widetilde{\Sigma}_N=\big(\{0,1\}\times \{0,1\}\big)^{B_N}$, 
with generator  
$N^2 \widetilde{\mathcal L}_N + \widetilde{\mathbb L}_N + N^2 \widetilde{L}_{\widehat b,N}$,
given by formulas \eqref{def:genExch},  \eqref{def:genCP-DRE}, \eqref{def:genCP-DRE-x},
\eqref{def:ratesCP-DRE}, \eqref{def:LbN}, where we replace 
${\Lambda}_N$ by  $B_N$,  and $\Gamma_N$ by 
$\widetilde{\Gamma}_N=\{-N,N\}\times\mathbb T^{d-1}_N$,
satisfies  the following theorem, proved in  \cite{phd}. 
\begin{theorem}\label{bdd:hydrodynamics} 
If for any $\delta >0$ and for  any  function 
 $\widehat G\in \mathcal C(\overline B ;\mathbb R^3)$,
$$
\lim_{N\to \infty}\mu_N\Big\{ \Big| \langle \widehat \pi_N(\xi,\omega), \widehat G(\cdot)\rangle\, -\, 
\langle \widehat \gamma (\cdot),\widehat G(\cdot)\rangle   \Big| \geq \delta \Big\}=0\, ,
$$
for an initial continuous profile $\widehat \gamma: B \to [0,1]^3$,  then
the sequence of probability measures $(Q_{\mu_N}^{N,\widehat b})_{N\geq 1}$ 
converges to the Dirac measure  concentrated on the unique  weak  solution 
$\widehat \rho(\cdot,\cdot)$ of the boundary value problem 
$$
\begin{cases}
\partial_t\widehat \rho  & = \;\;  \Delta \widehat\rho \,+\, \widehat F(\widehat \rho)\, 
\qquad \hbox {in } \qquad  B \times (0,T) ,\\ 
\widehat \rho(0,\cdot)    & =  \;\; \widehat \gamma (\cdot)  \,  \qquad \hbox {in } \qquad  B ,\\
\widehat \rho(t,\cdot)|_\Gamma  & = \;\; \widehat b(\cdot) \;\;\; 
\hbox{ for }\; 0\leq t\leq T \, ,
\end{cases}
$$
that is, $\widehat \rho(\cdot, \cdot ):[0,T]\times B \to [0,1]^3 $ satisfies
\begin{itemize}
 \item[(B1)] For any $i\in\{1,2,3\}$, $\rho_i \in L^2 \left(( 0,T) ; H^1(B)\right)$,
  that is, $\rho_i$ admits a generalised derivative
 such that,  if  $\nabla \rho_i =(\partial_{e_1}\rho_i,\cdots ,\partial_{e_d}\rho_i)$ 
stands for the gradient of $\rho_i$, 
$$\int_0^T d s\Big(  \int_B  {\parallel\nabla \rho_i(s,u)\parallel}^2 
du \Big)<\infty \, . $$
\item[(B2)]  For any function $\widehat G_t(u) = \widehat G(t,u) \in {\mathcal C}_{0}^{1,2}\big([0,T]
\times \overline B ;\mathbb R^3\big)$, we have
$$
 \begin{aligned}
& \langle \widehat \rho_T(\cdot), \widehat G_T(\cdot)\rangle
-\langle \widehat \rho_0(\cdot), \widehat G_0(\cdot)\rangle -
\int_0^T ds \,  \langle \widehat \rho_s(\cdot),\partial_s \widehat G_s(\cdot)\rangle
=\, \int_0^T ds \,  \langle \widehat \rho_s(\cdot),\Delta \widehat G_s(\cdot)\rangle 
 \\
&\qquad \quad
\, +\, \int_0^T ds \,  \langle \widehat F(\rho_s)(\cdot),\widehat G_s(\cdot)\rangle
 -\; \sum_{i=1}^3 \int_0^T ds
 \int_{\widetilde\Gamma} \,  \text{\bf n}_1(r)\, b_i(r) (\partial_{e_1}G_{i,s})(r) \, dS(r) \, ,
 \end{aligned}
$$
where  $\widetilde\Gamma$  denotes here  the boundary of $B$.
\item[(B3)]  $\widehat \rho (0,u)= \widehat \gamma (u)$ a.e.
\end{itemize}
\end{theorem}
{}From now on, we are back to the set-up of Theorem \ref{res:hydrodynamics}. 
\subsection{Currents.}\label{subsec:pres-currents}
In this subsection,  we study the evolution of the empirical currents, 
namely the conservative one (cf. \cite{BDGJL06}) and the non-conservative one (cf. \cite{BL12}).
For $t\ge 0$, $1\le i\le 3$ and any $y,z \in \Lambda_N$ such that $\|y-z\|=1$, 
 denote by $J^{y,z}_t(\eta_i)$ the total number of particles of type $i$ that
jumped from $y$ to $z$ before time $t$ and by 
$$W^{x,x+e_j}_t(\eta_i) = J^{x,x+e_j}_t(\eta_i)  - J^{x+e_j,x}_t(\eta_i), \ 1\le j\le d $$
the conservative current of particles of type $i$ across the bond $\{x,x+e_j\}$ before time $t$.
The corresponding conservative empirical measure $\mathbb W_t^N$ is the product finite signed measure on 
$\Lambda_N$ defined as  $\mathbb W_t^N(\eta_i)=(W^{N}_{1,t}(\eta_i), \dots, 
W^{N}_{d,t}(\eta_i)) \in \mathcal M^d$, where for $1\le j\le d$, $1\le i\le 3$,
$$
W^{N}_{j,t}(\eta_i) \;=\; N^{-(d+1)} \sum_{x, x+e_j\in\Lambda_N} 
W^{x,x+e_j}_t(\eta_i) \delta_{x/N}\, .
$$
For a continuous vector field   $\mathbf G = (G_1, \dots, G_d) \in \mathcal C_c (\Lambda;\bb
R^d)$ the integral of $\mathbf G$ with respect to $\mathbb W^N_t(\eta_i)$, also
denoted by $\langle \mathbb W^N_t(\eta_i), \mathbf G\rangle $, is given by
\begin{equation}\label{empcurr}
\langle \mathbb W^N_t (\eta_i), \mathbf G\rangle 
\;=\; \sum_{j=1}^d \langle W^N_{j,t}(\eta_i) , G_j\rangle\, .
\end{equation}
  For $x\in\Lambda_N$, $1\le i\le 3$, we denote by $Q_t^x(\eta_i)$  
  the total number of particles of type $i$ created 
minus the total number of particles of type $i$ annihilated at site $x$ before time $t$. 
The corresponding non-conservative empirical measure is
$$
Q^N_t(\eta_i) =\dfrac{1}{N^d}  \sum_{x\in \Lambda_N}Q_t^x(\eta_i)\delta_{x/N}.
$$
  We can now state the law of large numbers for the current: 
\begin{proposition}\label{res:lln_currents}
Fix a smooth initial profile $\widehat \gamma :\Lambda \to [0,1]^3$.
Let $(\mu_N)_{N \geq 1}$ be a sequence of probability measures on $\widehat \Sigma_N$
satisfying \eqref{cts_profile} and $\widehat\rho$ be the weak solution of the system of equations 
\eqref{f00}. Then, for each $T>0$, $\delta>0$,   
 $\widehat {\mathbf G}= ({\mathbf G}_1, {\mathbf G}_2, {\mathbf G}_3)
  \in \mathcal C^1_c(\Lambda;(\mathbb R^d)^3)$ 
and $\widehat H= (H_1, H_2, H_3) \in \mathcal C^1_c(\Lambda;\mathbb R^3)$, 
\begin{align}
&\lim_{N\to\infty} \mathbb P_{\mu_N}^{N,\widehat b} \Big[\, \Big| 
\sum_{i=1}^3 \langle \mathbb W^N_t (\eta_i), \mb G_i\rangle
\;-\; \int_0^T dt\, \big\langle \big\{-\nabla\widehat {\rho}_t \big\}
\, , \,\widehat{\mb G}\big\rangle \Big\vert > \delta \Big ] \;=\; 0\,  \label{res:lln_currents_1} ,\\
&\lim_{N\to\infty} \mathbb P_{\mu_N}^{N,\widehat b} 
\Big[\, \Big| 
\sum_{i=1}^3 \langle Q^N_t (\eta_i), H_i\rangle
\;-\; \int_0^T dt\, \big\langle   \widehat F(\widehat {\rho}_t)
\, , \,\widehat H\big\rangle \Big| > \delta \Big ] \;=\; 0\, . \label{res:lln_currents_2} 
\end{align}
\end{proposition} 
  We shall prove Proposition \ref{res:lln_currents}
in Section \ref{sec:currents}.
\section{Outline of the proof of Theorem \ref{res:hydrodynamics}}\label{sec:hydrodynamics}
The proof is divided essentially in two parts. In the
first one, we prove the hydrodynamic limit for the system evolving in a large finite volume.
There, by large we just mean  a volume of size $M_N$ such that $\lim_{N\to+\infty}M_N/N=+\infty$. 
In the second part, from the first one and coupling arguments, 
we shall derive the result in infinite volume. 
The coupling will allow us to prove that  by taking $M_N$ appropriately large enough,
particles outside the cylinder of length $M_N$ do not affect enough 
the number of particles in a box with length of order $KN$ for any 
$K>0$ (cf. Propositions \ref{bddinf2} and \ref{bddinf3} below). 
 For all these requirements on $M_N$ to be fulfilled we take 
\begin{equation}\label{def:MN}
M_N= N^{1+\frac{1}{d}}.
\end{equation}
\noindent 
$\bullet$ For the first part
of the proof of Theorem \ref{res:hydrodynamics}, we  consider a Markov process 
with  state space $\widehat\Sigma_{N,M_N}$ (cf. \eqref{Si-Nn})
and generator $\mf L_{N,M_N}$, where for any positive integer $n>1$,
$\mf L_{N,n}$ denotes the restriction of the generator $\mf L_{N}$
to the box $\Lambda_{N,n}$:  
\begin{equation}\label{def:gen-Nn}
\mf L_{N,n}=N^2\mathcal L_{N,n}\, +\, \mathbb L_{N,n}\, +\, N^2 L_{\widehat b,N,n}  \, ,
 \end{equation} 
 with,  for $\mathcal L^{x,x+e_k}$ defined in \eqref{def:genExch-x}, 
 $\mathbb L^{x}_{\Lambda_{N,n}}$ in
  \eqref{def:genCP-DRE-x},  and 
$L^{x}_{\widehat b,N}$ in \eqref{def:LbN-x}, 
\begin{equation}\label{def:gen-x}
\mathcal L_{N,n} =\sum\limits_{k=1}^d \sum_{x,x+e_k\in\Lambda_{N,n}}
\mathcal L^{x,x+e_k}\, ,\quad
\mathbb L_{N,n} =\sum_{x\in \Lambda_{N,n}}\mathbb L^{x}_{\Lambda_{N,n}}\, ,\quad
L_{\widehat b,N,n} = \sum_{x\in \Lambda_{N,n}\cap\Gamma_N} L^{x}_{\widehat b,N}\, .
 \end{equation} 
Observe that this finite volume dynamics 
can be seen as a dynamics $(\zeta_t,\chi_t)_{t\in [0,T]}$ evolving in the infinite volume $\Lambda_N$,
where transitions taking place outside $\Lambda_{N,M_N}$, or involving particles 
outside $\Lambda_{N,M_N}$, are suppressed. For the exchange part of the dynamics, it means that
particles outside $\Lambda_{N,M_N}$ do not move
and particles inside $\Lambda_{N,M_N}$ jump as in the original infinite volume
process in $\Lambda_N$, with the restriction that jumps off $\Lambda_{N,M_N}$ are suppressed.
For the CPRS part, it means that 
transitions outside $\Lambda_{N,M_N}$, as well as births in $\Lambda_{N,M_N}$
induced by particles outside $\Lambda_{N,M_N}$, are suppressed. For the boundary dynamics part,
it means that transitions outside $\Lambda_{N,M_N}$ are suppressed.
By abuse of language, we still denote the generator of 
$(\zeta_t,\chi_t)_{t\in [0,T]}$ by $\mf L_{N,M_N}$.
Given a probability measure $\mu_N$ on $\widehat \Sigma_N$,
we denote by ${\widetilde {\mathbb P}}_{\mu_N}^{M_N,\widehat b}$ 
the law of the process $(\zeta_t,\chi_t)_{t\in [0,T]}$ with initial distribution $\mu_N$, by
${\widetilde {\mathbb E}}_{\mu_N}^{M_N,\widehat b}$ the corresponding  expectation, and by
$\widetilde {Q}_{\mu_N}^{M_N,\widehat b}  =  
{\widetilde {\mathbb P}}^{M_N,\widehat b}_{\mu_N} \circ (\widehat \pi^N)^{-1} $ 
  the law of the process
$\big(\widehat \pi^N (\zeta_t,\chi_t) \big)_{t\in [0,T]}$. 
 We define $(\widetilde \eta_1,\widetilde \eta_2,\widetilde \eta_3)$ then
$\widetilde \eta_0$ associated to $(\zeta,\chi)$ as we did for
 $(\eta_1,\eta_2,\eta_3)$ and $\eta_0$ w.r.t. $(\xi,\omega)$ 
in \eqref{omega}--\eqref{omega-compl}.
\smallskip

\noindent
Following the Guo, Papanicolaou and Varadhan method \cite{gpv},
to derive the hydrodynamic behaviour of our system in large finite volume, 
we divide the proof into several steps.
 We first prove through the
 martingales associated to the process $(\zeta_t,\chi_t)_{t\in [0,T]}$
that the limit is the Dirac mass concentrated on the set of weak solutions to the system
of equations \eqref{f00}:
\begin{itemize}
 \item[{(1)}] tightness of the measures 
$(\widetilde {Q}_{\mu_N}^{M_N,\widehat b})_{N \geq 1}$ in $D([0,T],(\mathcal{M}_+^1)^3)$,
\item[{(2)}] identification of the limit points of $(\widetilde {Q}_{\mu_N}^{M_N,\widehat b})_{N \geq 1}$
 as  absolutely continuous paths whose density satisfies conditions (IB1) and (IB2).
\end{itemize}
Then
condition \eqref{cts_profile} implies (IB3), and we have to prove
\begin{itemize}
 \item[{(3)}] uniqueness of a weak solution to the hydrodynamic equations \eqref{f00},
\end{itemize} 
which allows to conclude that
the limit is the Dirac measure associated to the unique solution of  \eqref{f00}. 
\smallskip

\noindent
The proof of {(1)}  and of part of {(2)} is by now standard and left to the reader
(we refer to \cite{kl} for details).
We postpone the proof of {(3)} 
 to Section \ref{sec:uniqueness}. As $M_N\gg N$, the hydrodynamic limit we obtain
is the system of equations \eqref{f00} in infinite volume with 
 reservoirs.  
 To understand how \eqref{f00} appears as limit point in  {(2)}, 
let us consider, for any function 
$\widehat G=(G_1,G_2,G_3) \in \mathcal  C_{c,0}^{1,2} ([0,T] \times \overline \Lambda ; \mathbb R^3)$, 
the mean-zero martingale 
  \begin{eqnarray}
\widehat M_T^{N} (\widehat G)& =& \sum_{i=1}^3 M_T^{N,i} (G_i) 
\qquad\mbox{where, for } 1\le i\le 3,\nonumber\\
M_T^{N,i} (G_i)&=&
\langle \pi_T^{N,i}, G_{i,T} \rangle - \langle \pi_0^N, G_{i,0} \rangle
  - \int_0^T \langle \pi_s^{N,i}, \partial_s G_{i,s} \rangle ds
 - \int_0^T \mathfrak L_{N,M_N} \langle \pi_s^{N,i}, G_{i,s} \rangle ds \, .
 \label{martingale}
  \end{eqnarray}
 To establish the convergence of $\widehat M_T^{N} (\widehat G)$
and exhibit a limit point,  we  compute, for $1\le i\le 3$: 
\begin{eqnarray}
N^2 \mathcal L_{N,M_N} \langle \pi_s^{N,i}, G_{i,s} \rangle 
&=& \langle \pi_s^{N,i}, \Delta_N G_{i,s} \rangle 
- \dfrac{1}{N^{d-1}}\sum\limits_{x \in 
 \Gamma_N^+}
  \partial_{e_1}^N G_{i,s}((x-e_1)/N) \widetilde\eta_{i,s}(x)\nonumber\\
&& + \dfrac{1}{N^{d-1}}\sum\limits_{x \in 
\Gamma_N^-}
 \partial_{e_1}^N G_{i,s}(x/N) \widetilde\eta_{i,s}(x)\, , \label{martingale_diff}
\end{eqnarray}
where $\Gamma_N^\pm 
= \{(u_1, \dots , u_d)\in \overline \Lambda_N  : u_1 = \pm N \}$.
Indeed, since $M_N\gg N$ and $G$ has compact support, for $N$ large enough $M_N$
does not appear on the r.h.s. of \eqref{martingale_diff}.\\ 
We also compute $\widehat f = (f_1,f_2,f_3) : \ws_N \rightarrow \mathbb R^3$:
\begin{equation}\label{eq:fmicro}
\begin{cases}
f_1 (\zeta,\chi) & = \;\; \mathbb L_{N,M_N} \widetilde\eta_1(0) 
= \beta_{\Lambda_{N,M_N}}(0,\zeta,\chi)  \widetilde\eta_0(0)
 + \widetilde\eta_3(0) - (r+1)\widetilde\eta_1(0), \\ 
f_2 (\zeta,\chi)   & =  \;\; \mathbb L_{N,M_N} \widetilde\eta_2(0) 
= r \widetilde\eta_0(0) + \widetilde\eta_3(0)
 -\beta_{\Lambda_{N,M_N}}(0,\zeta,\chi) \widetilde\eta_2(0) - \widetilde\eta_2(0), \\ 
 f_3 (\zeta,\chi)  & = \;\; \mathbb L_{N,M_N} \widetilde\eta_3(0) 
 = \beta_{\Lambda_{N,M_N}}(0,\zeta,\chi) \widetilde\eta_2(0)
  + r \widetilde\eta_1(0) - 2 \widetilde\eta_3 (0)\, ,
\end{cases}
\end{equation}
so that
\begin{equation}\label{martingale_diff-2}
 \mathbb L_{N,M_N} \langle \pi_s^{N,i}, G_{i,s} \rangle 
=  \dfrac{1}{N^{d}}\sum\limits_{x \in \Lambda_{N,M_N}} G_{i,s}(x/N) \tau_x f_i (\zeta_s,\chi_s), 
\end{equation} 
 where $\tau_\cdot$ denotes the shift operator, that is, 
for $x \in \Lambda_{N,M_N}$, 
$\tau_x \widetilde \eta_1 (0) = \mathbf 1_{\{\widetilde \eta (x) = 1\}} =\zeta(x)(1-\chi(x))$.\\
Again, for $N$ large enough,  $M_N$ does not play any role 
on the r.h.s. of \eqref{martingale_diff-2}.\\
Since $\widehat G$ vanishes at the boundaries on $\overline \Lambda$,
the generator $L_{\widehat b,N,M_N}$ is not needed.\\
Therefore, to close the equations in $\widehat M_t^{N} (\widehat G)$, we need to do two replacements,
stated in Lemma \ref{replacement} and Proposition \ref{bddinf2} below:
we have to replace local functions in the bulk by a function of the empirical density 
in \eqref{martingale_diff-2}, and  
to replace the density at the boundary by a value of the function 
$\widehat b$ in \eqref{martingale_diff}.
Lemma \ref{replacement} and Proposition \ref{bddinf2} are the main steps to show that
 any limit point of the sequence 
 $(\widetilde {Q}_{\mu_N}^{M_N,\widehat b})_{N \geq 1}$
is concentrated on trajectories that are weak solutions of the system of equations \eqref{f00}.
The proof of Lemma \ref{replacement} relies on uniform upper bounds 
on the entropy production and the Dirichlet form stated in 
Subsection \ref{subsec:specific} (Theorem \ref{dirichlet}) 
and proved in Subsection \ref{sec:entropy}.
The proof of Proposition \ref{bddinf2} relies on the properties of the boundary dynamics.
\begin{remark}\label{rem_bdd} 
Except for the replacement lemma at the boundary (that is, Proposition \ref{bddinf2}),
all the results needed in steps (1), (2) and (3) are valid both in infinite volume and in a
large finite volume with length
$M_N=N^{1+\frac1d}$. Therefore, we shall state and prove all these results in infinite volume.
\end{remark}
For any smooth profile $\widehat{\theta}=(\theta_1,\theta_2,\theta_3) : \overline \L \to (0,1)^3$ 
satisfying \eqref{eq:cC} and \eqref{b-for-rev}, and for
any cylinder function $\phi(\xi,\omega)$, denote by $\widetilde \phi (\widehat \theta) $ 
the expectation of $\phi$ with respect to $\nu_{\widehat \theta }^N$. 
For any $\ell \in \mathbb N$, define the empirical mean densities in a box of size 
$(2\ell+1)^d$ centered at $x$ by 
$\widehat \eta^\ell (x) = (\eta_1^\ell (x),\eta_2^\ell (x),\eta_3^\ell (x))$:
$$\eta_i^\ell (x) = \dfrac{1}{(2\ell+1)^d} 
\sum_{\|y-x\| \leq \ell\atop y\in \Lambda_N} \eta_i(y), 
\mbox{ for  } 1\le i \le 3.$$
so that we can define for any $\epsilon >0$ small enough (as usual we omit 
  to write integer parts in bounds of intervals: 
   $\epsilon N$  replaces $\lfloor\epsilon N\rfloor$),
  \begin{equation}\label{block-V}
V_{\epsilon N} (\xi,\omega) = \Big| \dfrac{1}{(2\epsilon N+1)^d} 
 \sum\limits_{\|y\| \leq \epsilon N}\tau_y \phi(\xi,\omega) 
 - \widetilde \phi(\widehat \eta^{\epsilon N}(0)) \Big|.
\end{equation}
\begin{lemma}[replacement in the bulk]\label{replacement}
For any $G \in \mathcal C_c^\infty ([0,T] \times  \Lambda, \mathbb R)$,
\begin{equation*}\label{replacement_1}
\limsup\limits_{\epsilon \rightarrow 0} \limsup\limits_{N \rightarrow \infty} 
 \mathbb E_{\mu_N}^{N,\widehat b} \Big( \dfrac{1}{N^d} \sum\limits_{x \in \Lambda}
   \int_0^T |G_s(x/N)| \tau_x V_{\epsilon N}(\xi_s,\omega_s) ds \Big) = 0.
\end{equation*}
\end{lemma}
The proof of Lemma \ref{replacement} is postponed to Subsection 
\ref{subsec:rep-lem}.\\
We now state that the limiting trajectories for the system in large finite volume 
satisfy the Dirichlet boundary 
conditions with value $\widehat b(\cdot)$. 
The proof of Proposition \ref{bddinf2} is postponed to  Section \ref{sec:hydro_fini}.
\begin{proposition}[replacement at the boundary]\label{bddinf2} 
For any bounded function $H : [0,T] \times \Gamma\to \R$ with compact support in $\Gamma$, for any 
$\delta >0$, for all $i\in\{1,2,3\}$,
$$
\limsup\limits_{N \rightarrow \infty} 
{\widetilde {\mathbb P}}_{\mu_N}^{M_N,\widehat b}
\Big( \Big| \int_0^T 
\dfrac{1}{N^{d-1}} \sum\limits_{x \in \Lambda_{N,M_N}\cap\Gamma_N}  
H_t(x/N) \Big( {\widetilde \eta}_{i,t}(x) - b_i(x/N) \Big)
 dt \Big|>\delta \Big) =0.
$$
\end{proposition}
\noindent 
$\bullet$ For the second part 
of the proof of Theorem \ref{res:hydrodynamics},
we couple the original process $(\xi_t,\omega_t)_{t\in [0,T]}$ in infinite volume
with $(\zeta_t,\chi_t)_{t\in [0,T]}$. 
Let ${\overline \mu}_N$ be the measure on ${\widehat\Sigma}_N\times {\widehat\Sigma}_N$
concentrated on its diagonal and with
marginals equal to ${\mu}_N$.
Denote by ${\overline {\mathbb P}}_{{\overline \mu}_N}^{M_N,\widehat b}$
the law of the coupled process $((\xi_t,\omega_t), (\zeta_t,\chi_t))_{t\in [0,T]}$
with initial distribution  ${\overline \mu}_N$, 
and by ${\overline {\mathbb E}}_{{\overline \mu}_N}^{M_N,\widehat b}$
the corresponding expectation. 
By Tchebycheff inequality, for all $\delta >0$ and $t\ge 0$,
\begin{eqnarray}\nonumber
&& \mathbb P_{\mu_N}^{N,\widehat b} \Big( \Big|  \dfrac{1}{N^{d}} 
\sum\limits_{x \in \Lambda_N}  G_{i,t}(x/N) \Big( \eta_{i,t}(x) - \rho_i(t,x/N) \Big)
 \Big|>\delta \Big)\\
&& \quad  \le
{\widetilde {\mathbb P}}_{\mu_N}^{M_N,\widehat b} \Big( \Big| \dfrac{1}{N^{d}} 
\sum\limits_{x \in \Lambda_N}  G_{i,t}(x/N) \Big( \widetilde\eta_{i,t}(x) - \rho_i(t,x/N)  \Big)
 \Big|>\frac{\delta}{2} \Big)\label{large-finite}\\
&& \qquad  +\frac2\delta
{\overline {\mathbb E}}_{\overline \mu_N}^{M_N,\widehat b} \Big( \Big| \dfrac{1}{N^{d}} 
\sum\limits_{x \in \Lambda_N}  G_{i,t}(x/N) \Big( \eta_{i,t}(x) -\widetilde\eta_{i,t}(x) \Big)
 \Big|\Big). \label{infinite-finite}
\end{eqnarray}
 The hydrodynamic result in large finite volume 
enables to deal with \eqref{large-finite}. For \eqref{infinite-finite},
 we have to prove the
following coupling result, which will conclude the proof of Theorem \ref{res:hydrodynamics}.
\begin{proposition}\label{bddinf3} 
For any bounded function $\widehat G  = (G_1,G_2,G_3): [0,T] \times \Lambda\to \R^3$ 
with compact support in $\Lambda$, for all $i\in\{0,1,2,3\}$,
\begin{equation*}
\limsup\limits_{N \rightarrow \infty} 
{\overline {\mathbb E}}_{\overline \mu_N}^{M_N,\widehat b} 
\Big( \Big|\dfrac{1}{N^{d}} \sum\limits_{x \in \Lambda_N}  
G_{i,t}(x/N) \Big( {\widetilde \eta}_{i,t}(x) -\eta_{i,t}(x) \Big)\Big|\Big) =0\, .
\end{equation*}
\end{proposition}
In Section \ref{sec:hydro_infini} we shall define the appropriate coupling
between $(\xi_t,\omega_t)_{t\in [0,T]}$ and
 $(\zeta_t,\chi_t)_{t\in [0,T]}$,  which turns out to be basic coupling, 
 and prove Proposition \ref{bddinf3}.
%
%
\section{Specific entropy, Dirichlet forms and proof of Lemma \ref{replacement}}\label{sec:specific}
\subsection{Specific entropy: Definitions and results}\label{subsec:specific}
We start by defining the two main ingredients needed in the proof 
of the hydrodynamic limit: the specific entropy and the specific Dirichlet form 
of a measure on  $\widehat \Sigma_N$ with respect to some reference product 
measure.
For each positive integer $n$ and a measure $\mu$ on ${\widehat \Sigma_N}$, 
we denote by $\mu_n$ the marginal of $\mu$ on ${\widehat \Sigma}_{N,n}$: 
For each $(\zeta,\chi)\in \widehat \Sigma_{N,n}$,
\begin{equation}\label{eq:marginal}
\mu_n(\zeta,\chi)=\mu\big\{ (\xi,\omega)\; :\; (\xi(x),\omega(x))=(\zeta(x),\chi(x)) \; \; 
 {\rm for}\ x \in \Lambda_{N,n} \big\} \, .
\end{equation}
We fix as \textit{reference measure} a  product measure 
$\nu^N_{\widehat\theta}:=\nu^N_{\widehat\theta(\cdot)} $, 
where $\widehat\theta=(\theta_1,\theta_2,\theta_3):\Lambda \to (0,1)^3$ 
is  a smooth function satisfying \eqref{eq:cC} and \eqref{b-for-rev}.
  In other words (recall \eqref{eq:rest-bar-nu},
 \eqref{eq:rho_i}), introducing the function 
 $\theta_0(.) = 1-\theta_1(.)-\theta_2(.)-\theta_3(.)$,
we have
\begin{eqnarray}\label{def:ref-measure}
\nu_{\widehat\theta(\cdot),n}^N (\xi,\omega) &=& \widehat Z_{\widehat \theta,n}^{-1} 
\exp \Bigg\{ \sum_{i=1}^3 \sum_{x\in \Lambda_{N,n}} 
\Bigg(\log \frac{\theta_i(x/N)}{\theta_0(x/N)}\Bigg) \eta_i(x) \Bigg\}\\
\hbox{  with}\qquad\widehat Z_{\widehat \theta,n}^{-1} &=&\prod_{x\in \Lambda_{N,n}} \theta_0(x/N) .\nonumber
\end{eqnarray}
We denote by $s_n(\mu_n | \nu_{\widehat\theta, n}^N)$ 
the relative entropy of $\mu_n$ with respect to $\nu_{\widehat\theta, n}^N$ defined by
\begin{equation}\label{def:rel-ent}
s_n(\mu_n | \nu_{\widehat\theta, n}^N) = \sup_{U\in C_b( \widehat \Sigma_{N,n})} \Big\{
\int U(\xi,\omega) d\mu_n(\xi,\omega) - \log \int e^{U(\xi,\omega)}d\nu_{\widehat\theta, n}^N (\eta,\xi)
\Big\} .
\end{equation}
In this formula $C_b( \widehat \Sigma_{N,n})$ stands for the space of all bounded continuous
functions on $\widehat \Sigma_{N,n}$. Since the measure $\nu_{\widehat\theta, n}^N$ gives a positive probability to each
configuration, all the measures on $\widehat \Sigma_{N,n}$ are absolutely continuous with respect to
$\nu_{\widehat\theta, n}^N$ and we have an explicit formula for the entropy :
\begin{equation}\label{def:rel-ent_exp}
s_n(\mu_n | \nu_{\widehat\theta, n}^N)= \int 
\log \left( f_n(\xi,\omega) \right)  d\mu_n (\xi,\omega),
\end{equation}
where $f_n$ is the probability density of $\mu_n$ with respect to $\nu_{\widehat\theta, n}^N$.

Define the  Dirichlet forms
\begin{equation}\label{def:Dn}
 D_n(\mu_n | \nu_{\widehat\theta, n}^N ) 
=\mathcal D_n^0 (\mu_n | \nu_{\widehat \theta ,n}^N)\,
+\, D^{\widehat b}_n  (\mu_n | \nu_{\widehat \theta ,n}^N) \, ,
\end{equation}
with
\begin{eqnarray}\label{def:Dn0}
\mathcal D_n^0 (\mu_n | \nu_{\widehat \theta ,n}^N)
 &=&\sum\limits_{k=1}^d\sum_{x:(x,x+e_k)\in \Lambda_{N,n}\times\Lambda_{N,n} }
 (\mathcal D_n^0)^{x,x+e_k} (\mu_n | \nu_{\widehat \theta ,n}^N)\, ,\\\label{def:Dnb}
 D^{\widehat b}_n  (\mu_n | \nu_{\widehat \theta ,n}^N)
 & =& \sum_{x\in \Lambda_{N,n}\cap\Gamma_N}(D^{\widehat b}_n)^{x}  (\mu_n | \nu_{\widehat \theta ,n}^N)\, ,
\end{eqnarray}
where, writing $y=x+e_k$,
\begin{align}\label{def:Dn0xy}
& (\mathcal D_n^0)^{x,y} (\mu_n | \nu_{\widehat \theta ,n}^N) =  
\int \Big( \sqrt{f_n(\xi^{x,y},\omega^{x,y})}
 - \sqrt{f_n(\xi,\omega)} \Big)^2d\nu_{\widehat\theta, n}^N(\xi,\omega), \\
& \begin{multlined}\label{def:Dnbx}
(D^{\widehat b}_n)^{x} (\mu_n | \nu_{\widehat \theta ,n}^N)= 
\int c_x\big(\widehat b(x/N), \xi,\sigma^x \omega\big) 
\Big( \sqrt{f_n(\xi, \sigma^x \omega)} - \sqrt{f_n(\xi, \omega)} \Big)^2
d\nu_{\widehat\theta, n}^N(\xi,\omega)\\ 
+ 
\int  c_x\big(\widehat b(x/N), \sigma^x\xi, \omega\big)
\Big( \sqrt{f_n(\sigma^x\xi, \omega)} - \sqrt{f_n(\xi, \omega)} \Big)^2
d\nu_{\widehat\theta, n}^N(\xi,\omega)\\ 
+  
\int  c_x\big(\widehat b(x/N), \sigma^x\xi, \sigma^x\omega\big) 
\Big( \sqrt{ f_n(\sigma^x\xi, \sigma^x\omega)}
 -\sqrt{ f_n(\xi, \omega)} \Big)^2
 d\nu_{\widehat\theta, n}^N(\xi,\omega) \, .
 \end{multlined}
\end{align} 
  We shall also need
\begin{equation}\label{def:dir-form-reaction}
\mathbb D_n (\mu_n | \nu_{\widehat \theta ,n}^N)
=\sum_{x\in \Lambda_{N,n}} (\mathbb D_n)^{x} (\mu_n | \nu_{\widehat \theta ,n}^N)\, ,
\end{equation} 
where
\begin{align}
& \begin{multlined} \label{def:dir-form-reaction-x}
(\mathbb D_n)^{x} (\mu_n | \nu_{\widehat \theta ,n}^N) =  
\int \Big(r(1-\omega(x))+\omega(x)\Big) \Big(\sqrt{f_n(\xi,\sigma^x \omega)} 
- \sqrt{f_n(\xi,\omega)} \Big)^2 d\nu_{\widehat\theta, n}^N(\xi,\omega)\\ 
+ 
\int \Big(\beta_{\Lambda_{N,n}} (x,\xi,\omega)(1-\xi(x))+\xi(x)\Big)
\Big( \sqrt{f_n(\sigma^x \xi,\omega)} - \sqrt{f_n(\xi,\omega)} \Big)^2
d\nu_{\widehat\theta, n}^N(\xi,\omega)\, .
\end{multlined}
\end{align}\\ 
Define  the \textit{specific entropy} $\mathcal S (\mu|\nu_{\widehat\theta}^N)$ 
and the \textit{Dirichlet form} $\mathfrak D(\mu|\nu_{\widehat\theta}^N)$ of a
measure $\mu$ on $\widehat \Sigma_N$ with respect to $\nu_{\widehat\theta}^N$ as
\begin{align}\label{def:spec-ent}
\mathcal S (\mu|\nu_{\widehat\theta}^N) 
& =N^{-1} \sum_{n\geq 1} s_n(\mu_n | \nu_{\widehat\theta,n}^N)e^{- n/N},\\
\mf D(\mu|\nu_{\widehat\theta}^N)  & = N^{-1}\sum_{n\geq 1} D_n(\mu_n | \nu_{\widehat\theta,n}^N)e^{-n/N}.
\end{align}
Notice that by the entropy convexity and since 
$\sup_{x\in \Lambda_N}\{\xi(x)+\omega(x) \}$ is finite, for any positive
measure $\mu$ on $\widehat \Sigma_N$ and any integer $n$, we have
\begin{equation}\label{entropy_0}
s_n(\mu_n | \nu_{\widehat\theta,n}^N)\, \le\, C_0 N n^{d-1}\, ,
\end{equation}
for some constant $C_0$ that depends on $\widehat \theta$ 
(cf. comments following Remark V.5.6 in \cite{kl}).
Moreover there exists a positive constant
$C_0' \equiv C(\widehat \theta)$ such that for any positive measure $\mu$ on $\widehat \Sigma_N$,
\begin{equation}\label{entropy_Nd}
\mathcal S (\mu|\nu_{\widehat\theta}^N) \, \le\, C_0' N^{d}\, .
\end{equation}
Indeed, by \eqref{def:spec-ent} and \eqref{entropy_0} we have
\begin{equation}\label{entropy_Nd-detail}
\mathcal S (\mu|\nu_{\widehat\theta}^N) \, \le\, 
N^d C_0 \frac{1}{N} \sum_{n\geq 1}\, e^{- n/N}  {\Big(\frac{n}{N}\Big)}^{d-1}
\end{equation}
that we bound comparing the Riemann sum with an integral.
\smallskip

\smallskip
\noindent
 The goal of this section is to  prove the following 
 appropriate bounds on the entropy production and the Dirichlet form. 
\begin{theorem}\label{dirichlet}
Let $\widehat\theta=(\theta_1,\theta_2,\theta_3):\overline{\Lambda} \to (0,1)^3$ 
be a smooth function
satisfying \eqref{eq:cC} and \eqref{b-for-rev}.
For any time $t\ge 0$, there exists a positive finite constant 
$C_1\equiv C(t,\widehat \theta, \lambda_1, \lambda_2, r)$, so  that 
$$
\int_0^t \mf D(\mu (s)|\nu_{\widehat\theta}^N)\, ds\, \le\, C_1 N^{d-2}\, .
$$
\end{theorem}
To get this result, we need to 
 bound the entropy production in terms of the Dirichlet form. This is given by the following lemma.
\begin{lemma}\label{ent_bound}
There exist positive constants $A_0, A_1$ such that for any $t> 0$,
\begin{equation}\label{ent_bound:eq}
\partial_t \mathcal S (\mu(t)|\nu_{\widehat\theta}^N) \leq -A_0 N^2 \mathfrak D(\mu (t)|\nu_{\widehat\theta}^N) + A_1 N^d\, .
\end{equation}
\end{lemma}
 Note that since we do not know invariant measures for our dynamics, we control 
the entropy production with respect to a product measure which is not stationary,
which leads to the correction term $A_1 N^d$ in \eqref{ent_bound:eq}. 
\medskip
\subsection{Specific entropy: Proof of Theorem \ref{dirichlet}}\label{sec:entropy}
We now prove Theorem \ref{dirichlet} and Lemma \ref{ent_bound}.
\begin{proof}[Proof of Theorem \ref{dirichlet}]
Integrate the expression \eqref{ent_bound:eq} from 0 to $t$ and use \eqref{entropy_Nd}.
\end{proof}
\begin{proof}[Proof of Lemma \ref{ent_bound}]
 For a measure $\mu_n$ on  $\widehat \Sigma_{N,n}$, 
 denote by $f_n^t$ the density of $\mu_n(t)$ with respect to 
 $\nu_{\widehat \theta, n}^N$. 
  By the definition \eqref{def:spec-ent} of specific entropy, we need to bound 
 $\partial_t s_n(\mu_n (t)| \nu_{\widehat\theta , n}^N)$. 
 For any subset ${\mathcal A} \subset \Lambda$ 
 and any function $f \in L^1(\nu_{\widehat \theta}^N)$, denote by 
 $\langle f \rangle_{\mathcal A}$ the function on 
 $(\{0,1\}\times \{0,1\} )^{\Lambda \setminus {\mathcal A}}$ obtained by integrating 
 $f$ with respect to $\nu_{\widehat \theta}^N$ over the coordinates 
 $\{(\xi(x),\omega(x)), x \in {\mathcal A}\}$.  We denote
\begin{equation}\label{def:LNn-ext}
\Lambda_{N,n}^c=\Lambda_{N,n+1}\setminus\Lambda_{N,n}.
\end{equation} 
In the case where 
 ${\mathcal A} = \Lambda_{N,n}^c$, we simplify the notation $\langle f \rangle_{\mathcal A}$  by 
 $\langle f \rangle_{n+1}$. Note that
\begin{equation} \label{4.21bis}
\langle f_{n+1}^t \rangle_{n+1}=f_n^t\, .
\end{equation} 
Following the Kolmogorov forward equation, one has
\begin{equation}\label{kolmogorov}
\partial_t f_n^t = \langle \mathfrak L_{N,n+1}^* f_{n+1}^t \rangle _{n+1},
\end{equation}
where $ \mathfrak L_{N,n}^*$ stands for the adjoint operator of 
$\mathfrak L_{N,n}$ in $L^2(\nu_{\widehat \theta,n}^N)$. 
By \eqref{kolmogorov},
\begin{align}\nonumber
& \partial_t s_n(\mu_n (t)| \nu_{\widehat\theta , n}^N) =  \partial_t \int f_n^t \big( \log f_n^t  \big)
d\nu_{\widehat \theta ,n}^N  
 = \int \big( \log f_n^t \big) \mathfrak L_{N,n+1}^* f_{n+1}^t d\nu_{\widehat \theta ,n+1}^N  \\
 & =  N^2 \int f_{n+1}^t  \mathcal L_{N,n+1} \big(\log f_n^t \big)  d\nu_{\widehat \theta ,n+1}^N 
+ \int f_{n+1}^t  \mathbb L_{N,n+1} \big(\log f_n^t\big)  d\nu_{\widehat \theta ,n+1}^N \nonumber\\
&\qquad + N^2 \int f_{n+1}^t L_{\widehat b,N,n+1} \big(\log f_n^t\big)  d\nu_{\widehat \theta ,n+1}^N 
\label{thelast3}\\
&\  =: \Omega_1\, +\,\Omega_2\, +\, \Omega_3\, .
\end{align}
We   shall first derive useful tools (in Step 1 below), then
obtain (in Steps 2,3,4) the following bounds on the three integrals $\Omega_1$, $\Omega_2$ 
and $\Omega_3$ in terms of the entropy and Dirichlet forms. 
There exists positive constants $C,C''_1,K_2$ such that
\begin{eqnarray}\nonumber
\Omega_1 &\leq&
- \frac{N^2}{2}  
 \mathcal D_n^0 (\mu_n(t) | \nu_{\widehat \theta ,n}^N) +  (Cn+C''_1A N)Nn^{d-2} \\
&& + \frac{N^2}{A}\sum_{\ell=1}^N\sum\limits_{k=2}^d 
\sum\limits_{(x,x+e_k) \in (\Lambda_{N,n}\times\Lambda_{N,n}^c)
\cup (\Lambda_{N,n}^c\times\Lambda_{N,n})}  
({\mathcal D}^0_{n+\ell})^{x,x+e_k}
(\mu_{n+\ell}(t) | \nu_{\widehat \theta ,n+\ell}^N)
\label{entbound_int1}\, ,\\
\label{entbound_int2}
\Omega_2 &\leq& - \mathbb D_n (\mu_n(t) | \nu_{\widehat \theta  ,n}^N)  + K_2 Nn^{d-1}\, ,\\
\label{entbound_int3}
\Omega_3 &\leq& - N^2 D_n^{\widehat b}(\mu_n (t)| \nu_{\widehat \theta  ,n}^N)\, .
\end{eqnarray}
The constant $A$ comes from \eqref{ineq2} in Step 1 below.
Gathering \eqref{entbound_int1}, \eqref{entbound_int2} and \eqref{entbound_int3} gives 
\begin{eqnarray}\nonumber
\partial_t s_n (\mu_n(t) | \nu_{\widehat \theta ,n}^N) 
&\leq& -\frac{N^2}{2} \mathcal D^0_n (\mu_n(t)|\nu_{\widehat \theta(\cdot),n}^N) \\ \nonumber
&&+ \frac{N^2}{A}\sum_{\ell=1}^N\sum\limits_{k=2}^d 
\sum\limits_{(x,x+e_k) \in (\Lambda_{N,n}\times\Lambda_{N,n}^c)
\cup (\Lambda_{N,n}^c\times\Lambda_{N,n})}  
({\mathcal D}^0_{n+\ell})^{x,x+e_k}
(\mu_{n+\ell}(t) | \nu_{\widehat \theta ,n+\ell}^N)\\ 
&&- \mathbb D_n (\mu_n(t) | \nu_{\widehat \theta  ,n}^N)
 -N^2 D_n^{\widehat b}  (\mu_n(t) | \nu_{\widehat \theta ,n}^N)
 + \big( (C + K_2)  n+C''_1A N\big)Nn^{d-2}\, . 
\label{eq:last-for-s}\end{eqnarray}
 Note that for any $M>\!\!> N$ large enough, for some positive constant  $K_1'$,
\begin{equation}\label{dirformalter}
\begin{aligned}
&\sum_{n= 1}^{M}\, \frac{1}{N} e^{- n/N} 
 \sum_{\ell=1}^N\sum\limits_{k=2}^d 
\sum\limits_{(x,y) \in \Lambda_{N,n}\times\Lambda_{N,n}^c\atop
y\in\{x+e_k,x-e_k\}}  
({{\mathcal D}}^0_{n+\ell})^{x,y} 
({\mu}_{n+\ell}(t) | {\nu}_{\widehat \theta ,n+\ell}^N)
\\
&\qquad\qquad\qquad\qquad
\le  K_1' \sum_{n= 1}^{M +N}\, \frac{1}{N} e^{- n/N}
{{\mathcal D}}^0_n ({\mu}_n(t)|{\nu}_{\widehat \theta,n}^N)
\; \le  K_1'\mathfrak D(\mu (t)|\nu_{\widehat\theta}^N)\, .
\end{aligned}
\end{equation}
To conclude the proof of the Lemma it remains to multiply \eqref{eq:last-for-s} 
by $N^{-1}\exp(-n/N)$ and sum over $n \in \mathbb N$ to write an upper bound for
$\partial_t \mathcal S (\mu(t)|\nu_{\widehat\theta}^N)$. 
Using \eqref{dirformalter} and choosing (for instance) $A=4K_1'$ on one hand, 
and dealing with the last term on the r.h.s. of 
\eqref{eq:last-for-s} as in \eqref{entropy_Nd-detail} on the other hand,
we get \eqref{ent_bound:eq}.  \\ \\
\noindent\textit{Step 1: Tools.}
To do changes of variables,
it is convenient to write \eqref{def:ref-measure} as follows:
\begin{eqnarray}\label{def:ref-measure-bis}
\nu_{\widehat \theta(\cdot),n}^N (\xi,\omega) = 
\exp \Big\{ \sum_{j=0}^3 \sum_{x\in \Lambda_{N,n}} \vartheta_j(x/N) \eta_j(x) \Big\}\qquad
\hbox{  with}\qquad\vartheta_j(x/N) =\log \theta_j(x/N)\,   \label{def:var-et-theta}\, .
\end{eqnarray}
\noindent$\bullet$\textit{ Changes of variables formulas}: For a cylinder function $f$
 on  $\widehat \Sigma_{N}$,  and $x,y\in\Lambda_N$, we have
\begin{itemize} 
 \item[(i)]
for $(i,j) \in \{0,1,2,3\}^2$ such that $ i\not=j$,
\begin{equation}\label{change_stir}
\int \eta_i(x) \eta_j(y) f(\xi^{x,y}, \omega^{x,y}) d\nu_{\widehat \theta }^N (\xi,\omega) 
 = \int \eta_j(x) \eta_i(y) 
 (R_{i,j}^{x,y} (\widehat \theta)+1) f(\xi, \omega) d\nu_{\widehat \theta }^N (\xi,\omega)
\end{equation}
\begin{eqnarray}\label{eq:Rij}
\mbox{where}&& R_{i,j}^{x,y} (\widehat \theta)  = 
\exp \Big((\vartheta_j(y/N) - \vartheta_j(x/N))- (\vartheta_i(y/N)-\vartheta_i(x/N)) \Big) - 1\\
\mbox{satisfies}&& R_{i,j}^{x,y} (\widehat \theta)  = O(N^{-1}). \label{Rij-Taylor}
\end{eqnarray} 
 \item[(ii)] for  $(i,j) \in \{(1,2),(2,1),(3,0),(0,3)\}$,
 \begin{equation}\label{change_jump1}
  \int \eta_i(x) f(\sigma^x\xi, \sigma^x\omega) d\nu_{\widehat \theta }^N (\xi,\omega)  
 = \int \eta_j(x) e^{\vartheta_i(x/N) 
 - \vartheta_j(x/N)} f(\xi, \omega) d\nu_{\widehat \theta }^N (\xi,\omega), 
 \end{equation}
\item[(iii)] for  $(i,j) \in \{(1,0),(0,1),(3,2),(2,3)\}$,
\begin{equation}\label{change_jump2}
 \int \eta_i(x) f(\sigma^x\xi, \omega) d\nu_{\widehat \theta }^N (\xi,\omega) 
 = \int \eta_j(x) e^{\vartheta_i(x/N)
  - \vartheta_j(x/N)} f(\xi, \omega) d\nu_{\widehat \theta }^N (\xi,\omega),
 \end{equation}
 \item[(iv)] for  $(i,j) \in \{(1,3),(3,1),(2,0),(0,2)\}$,
\begin{equation}\label{change_jump3}
 \int \eta_i(x) f(\xi, \sigma^x\omega) d\nu_{\widehat \theta }^N (\xi,\omega)  
= \int \eta_j(x) e^{\vartheta_i(x/N) 
- \vartheta_j(x/N)} f(\xi, \omega) d\nu_{\widehat \theta }^N (\xi,\omega) \, .
 \end{equation}
  \end{itemize}

    \noindent$\bullet$\textit{ Inequalities}:  For any positive $a,b,A$, 
\begin{eqnarray}\label{ineq1}
a (\log b - \log a) &\leq& - \big( \sqrt{b} - \sqrt{a} \big)^2 + (b-a),\\
\label{ineq:encore}
\log a &\leq& 2 (\sqrt{a} - 1),\\
2ab &\leq& \dfrac{N}{A} a^2 + \dfrac{A}{N} b^2.\label{ineq2}
\end{eqnarray}
\begin{consequence}\label{csq:chg-var}
  \begin{itemize} 
 \item[\textit{(i)}] 
For all $x$, $1\le k\le d$ such that $x,x+e_k\in \Lambda_{N,n}$
we have, 
\begin{align*}
& 
\int  \mathcal L^{x,x+e_k} f(\xi,\omega) d\nu_{\widehat \theta,n}^N (\xi,\omega) 
= \dfrac{1}{2}\int  \mathcal L^{x,x+e_k} f(\xi,\omega) d\nu_{\widehat \theta,n}^N (\xi,\omega)
\\
&\quad+\dfrac{1}{2}\sum\limits_{0\le i\not=j\le 3}
\int \eta_j(x) \eta_i(x+e_k)(R_{i,j}^{x,x+e_k}(\widehat \theta)+1)
\Big(f(\xi,\omega)-f(\xi^{x,x+e_k},\omega^{x,x+e_k}) \Big)  
 d\nu_{\widehat \theta,n}^N (\xi,\omega)\\
& =  -\dfrac{1}{2}  \sum\limits_{0\le i\not=j\le 3} \int \eta_j(x) 
\eta_i(x+e_k)   R_{i,j}^{x,x+e_k}(\widehat \theta) 
  \Big(f(\xi^{x,x+e_k},\omega^{x,x+e_k})- f(\xi,\omega) \Big)  
 d\nu_{\widehat \theta,n}^N (\xi,\omega)\\
&\le \dfrac{1}{2} (\mathcal D_{n}^0 )^{x,x+e_k}(f)\\
&\quad+\dfrac{1}{8}\sum\limits_{0\le i\not=j\le 3} \int \eta_j(x) 
\eta_i(x+e_k)   \big(R_{i,j}^{x,x+e_k}(\widehat \theta)\big)^2 
\Big(\sqrt{f(\xi^{x,x+e_k},\omega^{x,x+e_k})}+\sqrt{ f(\xi,\omega)} \Big)^2
 d\nu_{\widehat \theta,n}^N (\xi,\omega)\\  
 &\le  \dfrac{1}{2} (\mathcal D_{n}^0 )^{x,x+e_k}(f)\\  
&\quad+ \dfrac{1}{8}\sum\limits_{0\le i\not=j\le 3} \int \eta_j(x) 
\eta_i(x+e_k)   \big(R_{i,j}^{x,x+e_k}(\widehat \theta)\big)^2 
 \Big(3 f(\xi^{x,x+e_k},\omega^{x,x+e_k})
 +\dfrac{3}{2} f(\xi,\omega) \Big) 
 d\nu_{\widehat \theta,n}^N (\xi,\omega) \\ 
  &\le  \dfrac{1}{2} (\mathcal D_{n}^0 )^{x,x+e_k}(f)\\  
 &\quad+ \dfrac{3}{8}\sum\limits_{0\le i\not=j\le 3} \int \eta_j(x+e_k) 
\eta_i(x)   \big(R_{i,j}^{x,x+e_k}(\widehat \theta)\big)^2 
 (R_{j,i}^{x+e_k,x}(\widehat \theta)+1)   f(\xi,\omega)
 d\nu_{\widehat \theta,n}^N (\xi,\omega)\\
  &\quad+ \dfrac{3}{16}\sum\limits_{0\le i\not=j\le 3} 
  \int \eta_j(x) \eta_i(x+e_k)   \big(R_{i,j}^{x,x+e_k}(\widehat \theta)\big)^2
   f(\xi,\omega) d\nu_{\widehat \theta,n}^N (\xi,\omega)\\
&  \leq 
\dfrac{1}{2}(\mathcal D_{n}^0 )^{x,x+e_k}(f)
 + 
 C  N^{-2} \| \sqrt f\|_{L^2(\nu_{\widehat \theta,n}^N)}^2 
\end{align*} 
for some positive constant $C$, where we  have used the change of variables formula
\eqref{change_stir} for the first equality and  last but one  inequality, 
\eqref{ineq2} with $A=N/2$
for the first and second inequalities (there we have first expanded the square),  
and \eqref{Rij-Taylor} 
twice with \eqref{eq:cC} for the last inequality.
 \item[\textit{(ii)}] Because the restriction of $\widehat\theta$ to $\Gamma$ is equal to $\widehat b$
 (see \eqref{b-for-rev}),
 \eqref{change_jump1}--\eqref{change_jump3} yield that
 the measure $\nu_{\widehat \theta}^N$ is reversible with respect 
  to the operator $L_{\widehat b ,N}$.
  \end{itemize}
  \end{consequence}
\noindent\textit{Step 2: Bound on $\Omega_1$}. We decompose the generator 
$\mathcal L_{N,n+1}$ into a part associated to exchanges within $\Lambda_{N,n}$ 
and a part associated to exchanges at the boundaries, that is,  
\begin{eqnarray}\nonumber
\Omega_1 
& = & N^2 \sum\limits_{k=1}^d\sum\limits_{(x,x+e_k) \in \Lambda_{N,n} \times \Lambda_{N,n}}  
 \int f_{n+1}^t \mathcal L^{x,x+e_k} (\log f_n^t) d\nu_{\widehat \theta ,n+1}^N \\ \nonumber
&&   + N^2 \sum\limits_{k=2}^d
\sum\limits_{(x,x+e_k) \in (\Lambda_{N,n}\times\Lambda_{N,n}^c)
\cup (\Lambda_{N,n}^c\times\Lambda_{N,n})}  
\int f_{n+1}^t \mathcal L^{x,x+e_k} (\log f_n^t) d\nu_{\widehat \theta ,n+1}^N\\
& = & N^2 \sum\limits_{k=1}^d\sum\limits_{(x,x+e_k) \in \Lambda_{N,n} \times \Lambda_{N,n}} 
 \Omega_1^{(1)} (x,x+e_k)\label{two-parts-1}\\  
& & + N^2 \sum\limits_{k=2}^d
\sum\limits_{(x,x+e_k) \in (\Lambda_{N,n}\times\Lambda_{N,n}^c)
\cup (\Lambda_{N,n}^c\times\Lambda_{N,n})}  
\Omega_1^{(2)}(x,x+e_k)\label{two-parts-2}\, .
\end{eqnarray}
 Successively,  for the term \eqref{two-parts-1}, writing $y=x+e_k$,
\begin{align}
& \Omega_1^{(1)} (x,y) = 
\int f_{n+1}^t (\xi,\omega) \Big( \log f_n^t(\xi^{x,y},\omega^{x,y}) 
- \log f_n^t (\xi,\omega) \Big) d\nu_{\widehat \theta  ,n+1}^N (\xi,\omega) \nonumber\\
 & \qquad =    \int \langle  f_{n+1}^t (\xi,\omega)\rangle_{n+1} \log 
\dfrac{ f_n^t(\xi^{x,y},\omega^{x,y})}{ f_n^t (\xi,\omega)} 
d\nu_{\widehat \theta  ,n}^N (\xi,\omega) \nonumber\\ 
& \qquad \leq - (\mathcal D_n^0)^{x,y} (\mu_n(t) | \nu_{\widehat \theta ,n}^N) 
+ \int \mathcal L^{x,y}f_n^t(\xi,\omega) d\nu_{\widehat \theta ,n}^N(\xi,\omega)\nonumber \\
&\qquad \leq  - \frac 12 (\mathcal D_n^0)^{x,y} (\mu_n(t) | \nu_{\widehat \theta ,n}^N) 
+ \frac{C}{N^2}\,\, ,
\label{successively}
\end{align}
 where we used \eqref{4.21bis} and \eqref{ineq1} for the first inequality
and Consequences \ref{csq:chg-var}\textit{(i)} combined with the fact that
$f_n^t$ is a probability density for the second one. \\ 

For the part \eqref{two-parts-2} associated to the boundaries, we shall write 
for each pair $(x,y)= (x,x+e_k)\in (\Lambda_{N,n} \times\Lambda_{N,n}^c)\cup
(\Lambda_{N,n}^c \times\Lambda_{N,n})$,
\begin{equation}\label{gen_ij}
\mathcal L^{x,y} = \sum\limits_{0\le i\not= j\le 3} \mathcal L^{x,y}_{i \leftrightarrow j} 
\end{equation}
where $\mathcal L^{x,y}_{i \leftrightarrow j}$ stands for the exchange of values $i$ and $j$
between sites $x$ and $y$.
 \begin{equation}\label{eq:L_ij}
\begin{aligned}
\mathcal L^{x,y}_{i \leftrightarrow j}  f(\xi,\omega)
 &= \eta_i(x) \eta_j(y)\Big( f(\xi^{x,y},\omega^{x,y}) 
 - f(\xi,\omega) \Big) +\eta_j(x)\eta_i(y) \Big( f(\xi^{x,y},\omega^{x,y}) - f(\xi,\omega) \Big)\, .
 \end{aligned}
\end{equation}
So that,  
\begin{eqnarray}\nonumber
\Omega_1^{(2)}  (x,y)
 & = & \sum\limits_{0\le i\not= j\le 3}  
 \int \eta_i(x)\eta_j(y) f_{n+1}^t (\xi,\omega) \log \dfrac{f_n^t(\xi^{x,y},\omega^{x,y})}{f_n^t(\xi,\omega)}  
 d\nu_{\widehat \theta  ,n+1}^N (\xi,\omega)\\ 
&& + \sum\limits_{0\le i\not= j\le 3}  \int \eta_j(x)\eta_i(y) f_{n+1}^t (\xi,\omega) 
\log \dfrac{f_n^t(\xi^{x,y},\omega^{x,y})}{f_n^t(\xi,\omega)} d\nu_{\widehat \theta  ,n+1}^N (\xi,\omega)\, .
\label{so-that-last-rhs}
\end{eqnarray}
Let us detail the computation for $(x,y)\in \Lambda_{N,n} \times\Lambda_{N,n}^c$ and $i=1$, $j=3$, the other values would be deduced in a similar way.
By a change of variables $(\xi',\omega') = (\xi^{x,y},\omega^{x,y})$ 
in the integral corresponding to $i=1,j=3$ in the second term of the 
r.h.s. \eqref{so-that-last-rhs},  
using \eqref{change_stir},\eqref{eq:Rij}, and noticing 
that if $\eta_1(x)\eta_3(y)=1$ then $\xi^{x,y}=\xi$, 
and that since $f_n^t$ does not depend on $y$, 
$f_n^t (\xi,\omega^{x,y})=f_n^t (\xi,\sigma^x \omega)$, we have
\begin{align}\nonumber 
& \int f_{n+1}^t(\xi,\omega) \mathcal L^{x,y}_{1 \leftrightarrow 3}(\log f_n^t (\xi,\omega))
d\nu_{\widehat \theta  ,n+1}^N (\xi,\omega)
 =  \int \eta_1(x)  \eta_3(y) f_{n+1}^t (\xi,\omega)
\log \dfrac{f_n^t(\xi^{x,y},\omega^{x,y})}{f_n^t(\xi,\omega)}  d\nu_{\widehat \theta  ,n+1}^N (\xi,\omega)\\ \nonumber 
&  \qquad+ \int  \eta_1(x) \eta_3(y) (R_{1,3}^{x,y} (\widehat \theta)+1)
 f_{n+1}^t (\xi^{x,y},\omega^{x,y})  \log \dfrac{f_n^t(\xi,\omega)}{f_n^t(\xi^{x,y},\omega^{x,y})} 
d\nu_{\widehat \theta  ,n+1}^N (\xi,\omega)\\\nonumber 
& \quad =  \int \eta_1(x) \langle \eta_3(y) f_{n+1}^t (\xi,\omega) \rangle_{n+1} 
\log \dfrac{f_n^t(\xi,\sigma^x \omega)}{f_n^t(\xi,\omega)}  d\nu_{\widehat \theta  ,n}^N (\xi,\omega)+O(N^{-1})\\ \nonumber 
&\qquad  + \int  \eta_1(x) \langle \eta_3(y) f_{n+1}^t (\xi,\omega^{x,y}) \rangle_{n+1}
\log \dfrac{f_n^t(\xi,\omega)}{f_n^t(\xi,\sigma^x \omega)} d\nu_{\widehat \theta  ,n}^N (\xi,\omega)\\
& \quad = \int \eta_1(x) \Big( \langle F_{1,3}^{(1)}(\xi,\omega) \rangle_{n+1} 
- \langle F_{1,3}^{(2)}(\xi,\omega) \rangle_{n+1} \Big) 
\log \dfrac{f_n^t(\xi,\sigma^x \omega)}{f_n^t(\xi,\omega)}  
d\nu_{\widehat \theta  ,n}^N (\xi,\omega) + O(N^{-1})\, ,\label{ineq:last-int}
\end{align}
where  we used \eqref{Rij-Taylor} in the second equality, and where
\begin{equation}\label{Fij}
F_{i,j}^{(1)} (\xi,\omega)= \eta_j(y) f_{n+1}^t (\xi,\omega) \, ,
\quad F_{i,j}^{(2)}(\xi,\omega) = \eta_j(y) f_{n+1}^t (\xi^{x,y},\omega^{x,y})\,.
\end{equation}
If we now define
\begin{eqnarray}
E_1(i,j) &:=& \{(\xi,\omega) :\langle F_{i,j}^{(1)}(\xi,\omega) \rangle_{n+1}
\geq \langle F_{i,j}^{(2)}(\xi,\omega) \rangle_{n+1} ,  
{f_n^t(\xi,\sigma^x \omega)} \geq {f_n^t(\xi,\omega)} \}\label{def:E1} \\
E_2(i,j) &:=&\{(\xi,\omega) : \langle F_{i,j}^{(1)}(\xi,\omega) \rangle_{n+1} 
\leq \langle F_{i,j}^{(2)}(\xi,\omega) \rangle_{n+1} , 
{f_n^t(\xi,\sigma^x \omega)} \leq {f_n^t(\xi,\omega)} \} \label{def:E2}
\end{eqnarray}
the integral in the r.h.s. of \eqref{ineq:last-int}
is non-negative on $E_1(1,3) \cup E_2(1,3)$. Then, thanks to the inequalities
\eqref{ineq:encore}--\eqref{ineq2},  for $A$ to be chosen later,  
the integral in the r.h.s. of \eqref{ineq:last-int} is bounded by
\begin{align}\nonumber
& \int_{E_1(1,3) \cup E_2(1,3)} \eta_1(x) \Big( \langle F_{1,3}^{(1)}(\xi,\omega) \rangle_{n+1} 
- \langle F_{1,3}^{(2)}(\xi,\omega) \rangle_{n+1} \Big) 
\log \dfrac{f_n^t(\xi,\sigma^x \omega)}{f_n^t(\xi,\omega)} 
d\nu_{\widehat \theta  ,n}^N (\xi,\omega) \\\nonumber
& \begin{multlined} \quad \leq 2 \int_{E_1(1,3) \cup E_2(1,3)} 
 \eta_1(x) \Big( \langle F_{1,3}^{(1)}(\xi,\omega) \rangle_{n+1} 
  - \langle F_{1,3}^{(2)}(\xi,\omega)  \rangle_{n+1} \Big) 
\times \Big(\sqrt{\dfrac{f_n^t(\xi,\sigma^x \omega)}{f_n^t(\xi,\omega)}}-1\Big)  
d\nu_{\widehat \theta  ,n}^N (\xi,\omega) \end{multlined} \\
& \begin{multlined} \quad \leq \dfrac{N}{A} \int_{E_1(1,3) \cup E_2(1,3)} 
 \eta_1(x) 
\Big( \sqrt{\langle F_{1,3}^{(1)}(\xi,\omega) \rangle_{n+1}} -
\sqrt{\langle F_{1,3}^{(2)}(\xi,\omega)  \rangle_{n+1}} \Big)^2 d\nu_{\widehat \theta  ,n}^N (\xi,\omega) \\
 + \dfrac{A}{N} \int_{E_1(1,3) \cup E_2(1,3)}  
 \Big( \sqrt{\langle F_{1,3}^{(1)}(\xi,\omega) \rangle_{n+1}} 
 + \sqrt{\langle F_{1,3}^{(2)}(\xi,\omega)  \rangle_{n+1}} \Big)^2 
  \times\Big(\sqrt{\dfrac{f_n^t(\xi,\sigma^x \omega)}{f_n^t(\xi,\omega)}}-1\Big)^2  
d\nu_{\widehat \theta  ,n}^N (\xi,\omega)\end{multlined}\nonumber\\ 
&\quad=: I_1+I_2\label{last-rhs} \, .
 \end{align} 
 In order to get rid of $N$ in $I_1$,
we use \eqref{4.21bis} to introduce a new sum in $m$, and to rewrite 
$I_1$ as
\begin{equation}\label{new-sum}
 \begin{aligned}
 I_1  =  \dfrac{N}{A} \dfrac{1}{N} \sum\limits_{m=n+1}^{n+N} 
\int_{E_1(1,3) \cup E_2(1,3)} \eta_1(x) \Big( \sqrt{\langle \eta_3(y) 
f_m^t(\xi^{x,y},\omega^{x,y}) \rangle_{\Lambda_{N,m}\setminus\Lambda_{N,n}}} \\ 
- \sqrt{\langle \eta_3(y)f_m^t(\xi,\omega)  \rangle_{\Lambda_{N,m}\setminus\Lambda_{N,n}}} \Big)^2
d\nu_{\widehat \theta  ,m}^N (\xi,\omega).
 \end{aligned}
\end{equation}
We now apply Cauchy-Schwarz inequality to bound $I_1$ by a piece of the specific Dirichlet form, 
\begin{align}
&\begin{multlined} I_1\leq  \dfrac{1}{A} \sum\limits_{m=n+1}^{n+N} 
\int_{E_1(1,3) \cup E_2(1,3)}  \eta_1(x) \Big\langle\eta_3(y)
\Big( \sqrt{ f_m^t(\xi^{x,y},\omega^{x,y})} 
- \sqrt{ f_m^t(\xi,\omega)}  \Big)^2 \Big\rangle_{\Lambda_{N,m}\setminus\Lambda_{N,n}}
 d\nu_{\widehat \theta  ,m}^N (\xi,\omega) \end{multlined} \nonumber\\
&  =  \dfrac{1}{A} \sum\limits_{m=n+1}^{n+N} 
\int_{E_1(1,3) \cup E_2(1,3)} \eta_1(x)\eta_3(y) 
\Big( \sqrt{ f_m^t(\xi^{x,y},\omega^{x,y})}  - \sqrt{ f_m^t(\xi,\omega)}  \Big)^2  
 d\nu_{\widehat \theta  ,m}^N (\xi,\omega)
 \nonumber \\ 
& \begin{multlined}  \leq  \dfrac{1}{A} \sum\limits_{m=n+1}^{n+N} 
\int
\eta_1(x)\eta_3(y)\Big( \sqrt{ f_m^t(\xi^{x,y},\omega^{x,y})} 
 - \sqrt{ f_m^t(\xi,\omega)}  \Big)^2  d\nu_{\widehat \theta  ,m}^N (\xi,\omega) 
 \end{multlined} \, .
\label{eq:bd1}\end{align}
  Now, to bound $I_2$, we separate the integrations on
 $E_1(1,3)$  and on $E_2(1,3)$. We first look at the integral on $E_1(1,3)$, to get 
\begin{align}
& \begin{multlined}\dfrac{A}{N} \int_{E_1(1,3)} \eta_1(x)
\Big( \sqrt{\langle F_{1,3}^{(1)}(\xi,\omega) 
\rangle_{n+1}} + \sqrt{\langle F_{1,3}^{(2)}(\xi,\omega)  \rangle_{n+1}} \Big)^2
 \Big(\sqrt{\dfrac{f_n^t(\xi,\sigma^x \omega)}{f_n^t(\xi,\omega)}}-1\Big)^2  
d\nu_{\widehat \theta  ,n}^N (\xi,\omega) \end{multlined}\nonumber\\
& \quad \leq \dfrac{4A}{N} \int_{E_1(1,3)} \eta_1(x) 
\dfrac{\langle F_{1,3}^{(1)}(\xi,\omega) \rangle_{n+1}}
{f_n^t(\xi,\omega)} \Big(\sqrt{f_n^t(\xi,\sigma^x \omega)} 
- \sqrt{f_n^t(\xi,\omega)}\Big)^2  d\nu_{\widehat \theta  ,n}^N (\xi,\omega)\nonumber \\
& \quad \leq \dfrac{4A}{N} \int_{E_1(1,3)} \eta_1(x)\Big(f_n^t(\xi,\sigma^x \omega) 
- 2\sqrt{f_n^t(\xi,\sigma^x \omega)}\sqrt{f_n^t(\xi,\omega)}+f_n^t(\xi,\omega)\Big)  d\nu_{\widehat \theta  ,n}^N (\xi,\omega) \nonumber\\
& \quad \leq \dfrac{4A}{N} \int_{E_1(1,3)} \eta_1(x)\Big(f_n^t(\xi,\sigma^x \omega) 
-f_n^t(\xi,\omega)\Big)  d\nu_{\widehat \theta  ,n}^N (\xi,\omega) \nonumber\\
& \quad \leq \dfrac{4A}{N} \int \eta_1(x)f_n^t(\xi,\sigma^x \omega) 
  d\nu_{\widehat \theta  ,n}^N (\xi,\omega)   = \dfrac{4A}{N}\int \eta_3(x) e^{(\vartheta_1(x/N) - \vartheta_3(x/N))}f_n^t(\xi, \omega)
   d\nu_{\widehat \theta  ,n}^N (\xi,\omega) \nonumber\\
& \quad \leq \dfrac{AC_1}{N} \label{cl:E1}
\end{align} 
for some positive constant $C_1$.
We have used the definition \eqref{def:E1} of $E_1(1,3)$ for the first and third inequalities,
 the definition \eqref{Fij} of $F_{1,3}^{(1)}(\xi,\omega)$ with the bound 
 $\langle F_{1,3}^{(1)}(\xi,\omega)  \rangle_{n+1}
 \leq\langle f_{n+1}^t (\xi,\omega)\rangle_{n+1}=f_n^t(\xi,\omega)$ for the second inequality,
 \eqref{change_jump3} for the equality, \eqref{def:var-et-theta}, 
  \eqref{eq:cC}
 and that $f_n^t$ is a probability density to conclude.\\
 
We now look at the integral on $E_2(1,3)$, to get
\begin{align}\nonumber
& \begin{multlined}\dfrac{A}{N} \int_{E_2(1,3)} \eta_1(x)
\Big( \sqrt{\langle F_{1,3}^{(1)}(\xi,\omega) 
\rangle_{n+1}} + \sqrt{\langle F_{1,3}^{(2)}(\xi,\omega)  \rangle_{n+1}} \Big)^2 
\Big(\sqrt{\dfrac{f_n^t(\xi,\sigma^x \omega)}{f_n^t(\xi,\omega)}}-1\Big)^2  
d\nu_{\widehat \theta  ,n}^N (\xi,\omega) \end{multlined}\\ \nonumber
& \quad \leq \dfrac{  4 A}{N} \int_{E_2(1,3)} \eta_1(x) 
\dfrac{\langle F_{1,3}^{(2)}(\xi,\omega) \rangle_{n+1}}
{f_n^t(\xi,\omega)} \Big(\sqrt{f_n^t(\xi,\sigma^x \omega)} 
- \sqrt{f_n^t(\xi,\omega)}\Big)^2  d\nu_{\widehat \theta  ,n}^N (\xi,\omega) \\ \nonumber
& \quad  \leq \dfrac{8 A}{N} \int_{E_2(1,3)} \eta_1(x) 
\dfrac{\langle F_{1,3}^{(2)}(\xi,\omega) \rangle_{n+1}}
{f_n^t(\xi,\omega)}  
f_n^t(\xi,\omega)  d\nu_{\widehat \theta  ,n}^N (\xi,\omega) \\ \nonumber
& \quad \leq \dfrac{  8 A}{N} \int
\eta_3(x) \eta_1(y) (R_{1,3}^{x,y} (\widehat \theta)+1)
 f_{n+1}^t(\xi, \omega)
  d\nu_{\widehat \theta  ,n+1}^N (\xi,\omega) \\
& \quad \leq \dfrac{AC_1'}{N} \label{cl:E2}
\end{align}
 for some positive constant $C'_1$.
We have used the definition \eqref{def:E2} of $E_2(1,3)$ for the first and second inequalities,
 the definition \eqref{Fij} of $F_{1,3}^{(2)}(\xi,\omega)$ with \eqref{change_stir} 
 for the third inequality, and \eqref{def:var-et-theta}, 
  \eqref{eq:cC}
 and finally that $f_n^t$ is a probability density.\\
 %
 Combining \eqref{successively} with \eqref{eq:bd1}, \eqref{cl:E1}, \eqref{cl:E2}, we get 
the upper bound \eqref{entbound_int1} of $\Omega_1$. 
\mbox{}\\ \\
\noindent\textit{Step 3: Bound on $\Omega_2$.} We decompose 
  the generator of the reaction part into a part involving only sites
within $\Lambda_{N,n}$ and a part 
involving sites in $\Lambda_{N,n}^c$. 
Recalling  \eqref{def:gen-Nn}, \eqref{def:gen-x}, 
we have
\begin{eqnarray}\nonumber
\Omega_2 & = & \int f_{n+1}^t \mathbb L_{N,n+1} (\log f_n^t )d\nu_{\widehat \theta , n+1}^N 
=  \int f_{n+1}^t \mathbb L_{N,n} (\log f_n^t) d\nu_{\widehat \theta , n+1}^N 
 + \Omega_2^{(1)} \, .\label{eq:(A)}
\end{eqnarray}
Proceeding as for \eqref{successively}, we get
\begin{eqnarray}
\int f_{n+1}^t \mathbb L_{N,n} \log f_n^t d\nu_{\widehat \theta , n+1}^N
 \leq  - \mathbb D_n (\mu_n(t) | \nu_{\widehat \theta ,n}) + \int \mathbb L_{N,n} f_n^t 
d\nu_{\widehat \theta , n}^N\, .\label{eq:(B)}
\end{eqnarray}
The second term on the r.h.s. is of order $O(Nn^{d-1})$ since the rates $\beta_{\Lambda_{N,n}}(.,.)$
are bounded.
Moreover, denoting $\partial \Lambda_{N,n}=\{x\in\Lambda_{N,n}:\exists y\in\Lambda_{N,n}^c, \|y-x\|=1\}$, 
\begin{eqnarray*}
\Omega_2^{(1)} & = & \sum\limits_{x \in \partial \Lambda_{N,n}}  
\int f_{n+1}^t (\xi,\omega) \Big( \lambda_1 \sum\limits_{y \in\Lambda_{N,n}^c\atop\|y-x\|=1} 
\xi(y)(1-\omega(y)) \\ & & 
+ \lambda_2 \sum\limits_{y \in\Lambda_{N,n}^c\atop\|y-x\|=1} \xi(y)\omega(y) \Big)
 (1-\xi(x))\log \dfrac{f_n^t(\sigma^x\xi, \omega)}{f_n^t (\xi,\omega)} 
 d\nu_{\widehat \theta , n+1}^N(\xi,\omega)  
\end{eqnarray*} 
which can be proved to be of order $O(Nn^{d-2})$  in an analogous way to the computation done for 
$\Omega_1^{(2)}$, and using that the rates  $\beta_{\Lambda_{N,n}}(.,.)$ 
are bounded,  inequalities
  \eqref{ineq:encore}--\eqref{ineq2}.  Combining this with \eqref{eq:(A)} 
yields the upper bound \eqref{entbound_int2} for $\Omega_2$. \\ \\ 
\noindent\textit{Step 4: Bound on $\Omega_3$.} It is in this step that 
the reversibility of the measure $\nu_{\widehat \theta  ,n}^N$ with respect to
the generator $L_{\widehat b ,N,n}$ plays a crucial role. It implies that,
for any $x \in \Lambda_{N,n} \cap \Gamma_N$, 
\begin{equation}\label{int-rev0}
\int L_{\widehat b, N}^x f_n^t d\nu_{\widehat \theta, n}^N=0.
\end{equation}
Since $L_{\widehat b, N,n+1} = \sum\limits_{x \in \Lambda_{N,n+1} \cap \Gamma_N} L_{\widehat b, N}^x$,  
using \eqref{4.21bis} and  inequality \eqref{ineq1}  we derive 
\eqref{entbound_int3} as follows. 
\begin{eqnarray*} \nonumber
\Omega_3 & = & N^2 \sum\limits_{x \in \Lambda_{N,n+1} \cap \Gamma_N} 
\int f_{n+1}^t L_{\widehat b, N}^x (\log f_n^t) d\nu_{\widehat \theta, n+1}^N \\\nonumber
& = &  N^2 \sum\limits_{x \in \Lambda_{N,n} \cap \Gamma_N} 
\int \langle f_{n+1}^t(\xi,\omega)\rangle_{n+1} L_{\widehat b, N}^x (\log f_n^t) d\nu_{\widehat \theta, n}^N \\\nonumber 
& \leq & - N^2 D_n^{\widehat b}(\mu_n (t)| \nu_{\widehat \theta  ,n}^N) 
 + N^2 \sum\limits_{x \in \Lambda_{N,n} \cap \Gamma_N} \int L_{\widehat b, N}^x f_n^t d\nu_{\widehat \theta, n}^N \\ 
& = &- N^2 D_n^{\widehat b}(\mu_n (t)| \nu_{\widehat \theta  ,n}^N)\, , 
\end{eqnarray*} 
where we used that $f_n^t$ does not depend on $\Lambda_{N,n}^c$ for the second equality. 
Thanks to \eqref{int-rev0}, in the last equality we got rid of a term with an order too large in $N$.
\end{proof}
\subsection{Replacement lemma in the bulk
(proof of Lemma \ref{replacement})}\label{subsec:rep-lem}
Fix $G \in \mathcal C_c^\infty ([0,T] \times \Lambda, \mathbb R)$, let $K>0,\delta>0$ 
be such that the (compact) support of $G$ is contained in the box 
$\Lambda(1-\delta,K) := [-1+\delta,1-\delta] \times [-K,K]^{d-1}$.
Let $0<a<\delta/2$, 
and let $\widehat \theta_a = (\theta_{a,1},\theta_{a,2},\theta_{a,3}) :\Lambda \rightarrow (0,1)^3$ 
be a smooth function,  equal in  
$\Lambda(1-a,K)$ 
to some constant, say $\widehat \alpha$, and to $\widehat b$ at the boundaries. 
Therefore
\begin{eqnarray*}
\mathbb E_{\mu_N}^{N,\widehat b} \Big( \dfrac{1}{N^d} \sum\limits_{x \in \Lambda_N} 
 \int_0^T |G_s(x/N)|  \tau_x V_{\epsilon N}(\xi_s,\omega_s) ds \Big)
\leq \|G\|_\infty 
\mathbb E_{\mu_N}^{N,\widehat b} \Bigg( \dfrac{1}{N^d} 
\sum\limits_{x \in \Lambda_{\lfloor N(1-\delta)\rfloor,NK}} 
  \int_0^T \tau_x V_{\epsilon N}(\xi_s,\omega_s) ds \Bigg).
  \label{replacement_eq1}
\end{eqnarray*} 
Denote  $\bar f^T = T^{-1} \displaystyle\int_0^T f_{N(K+2)}^s ds$. 
By Theorem \ref{dirichlet}, there exists some positive constant $C_1$ 
such that the expectation on the above r.h.s. 
 is bounded by 
\begin{equation*}
\dfrac{T }{N^d} \int \sum\limits_{x \in \Lambda_{\lfloor N(1-\delta)\rfloor,NK}} 
\tau_x V_{\epsilon N}(\xi,\omega) 
\bar f^T(\xi,\omega) d\nu_{\widehat \theta_a,N(K+2)}^N (\xi,\omega) 
 - \gamma T N^{2-d}  D_{N(K+2)} (\bar f^T)  + \gamma C_1,
 \end{equation*}
for all positive $\gamma$. 
To prove the Lemma, it thus remains to show that for all positive $\gamma, a$,
\begin{multline*}
\limsup\limits_{ \epsilon \rightarrow 0} \limsup\limits_{N \rightarrow \infty} 
\sup\limits_{f} \Bigg( \dfrac{1 }{N^d} \int \sum\limits_{x \in \Lambda_{\lfloor N(1-\delta)\rfloor,NK}} 
\tau_x V_{\epsilon N}(\xi,\omega) f(\xi,\omega) d\nu_{\widehat \alpha, N(K+2)}^N (\xi,\omega) 
- \gamma  N^{2-d}  D_{N(K+2)} (f) \Bigg)  = 0, 
\end{multline*}
where the supremum is carried over all densities $f$ with respect to  
$\nu_{\widehat \alpha,N(K+2)}^N$ such that $D_{N(K+2)} (f)\le C N^{d-2}$.
This is a consequence of the one and two 
block estimates stated below 
 (see \cite{gpv, kl} for the now standard proofs).
The one block estimate ensures the average of local functions in some large 
microscopic boxes can be replaced by its mean with respect to the grand-canonical measure 
parametrized by the particle density in these boxes. While the two block estimate ensures 
the particle density over large microscopic boxes is close to the one over small macroscopic boxes:
\begin{lemma}[One block estimate]\label{1block}
 Given  a constant profile $\widehat \rho = (\rho_1,\rho_2,\rho_3) \in (0,1)^3$,
\begin{equation*}
\limsup\limits_{k \rightarrow \infty} \limsup\limits_{N \rightarrow \infty} 
\sup\limits_{f : \mathcal D^0_{N(K+2)}(f) \leq CN^{d-2}} \int \dfrac{1}{N^d} 
\sum\limits_{x \in \Lambda_{N,NK}} \tau_x V_k(\xi,\omega) f(\xi,\omega) 
d\nu_{\widehat \rho,{N(K+2)}}^N(\xi,\omega)= 0\, ,
\end{equation*}
where for $k \in \mathbb N$, $V_k (\xi,\omega)$ was defined in \eqref{block-V}.
\end{lemma}
\begin{lemma}[Two block estimate]\label{2block}
Given a constant profile $\widehat \rho = (\rho_1,\rho_2,\rho_3) \in (0,1)^3$, for $i\in\{1,2,3\}$,
\begin{multline*}
\limsup\limits_{k \rightarrow \infty} \limsup\limits_{\epsilon \rightarrow 0} 
\limsup\limits_{N \rightarrow \infty} \sup\limits_{f : \mathcal D^0_{N(K+2)} (f) \leq CN^{d-2}} \\ 
\sup\limits_{|h| \leq \epsilon N} \dfrac{1}{N^d} \int \sum\limits_{x \in \Lambda_{N,NK}} 
\big|  \eta_i^k (x+h)-  \eta_i^{\epsilon N}(x) \big| f(\xi,\omega) 
d\nu_{\widehat \rho,{N(K+2)}}^N (\xi,\omega) = 0.
\end{multline*}
\end{lemma}
\section{Hydrodynamic limit in large finite volume: 
 Proof of Proposition \ref{bddinf2}}\label{sec:hydro_fini}
In this section, we prove the last result to derive
the hydrodynamic limit in large finite volume (that is of size 
$M_N= N^{1+\frac{1}{d}}$), Proposition \ref{bddinf2}. 
As mentioned in Remark \ref{rem_bdd}, Proposition \ref{bddinf2} is
 the only difference in our proof of hydrodynamics with the case of 
 infinite volume dynamics. 
\subsection{Estimates for finite volume}\label{sec:est_fini}
Next estimates will be useful to prove Proposition \ref{bddinf2}.
\begin{lemma}\label{estimates_bdd}
For a smooth profile 
$\widehat \theta =(\theta_1,\theta_2,\theta_3) : \overline \Lambda \to (0,1)^3$  
satisfying \eqref{eq:cC} and \eqref{b-for-rev},
there exist positive constants $A_0$, $A_0'$ and $A_1$ depending only 
on $\widehat \theta$ such that for any $c >0$, for any cylinder function
$f \in L^2 (\nu_{\widehat \theta,M_N}^N)$,
\begin{align}
&\langle L_{\widehat b,N,M_N} \sqrt f, \sqrt f\rangle 
= - D_{M_N}^{\widehat b} (f), \label{estimates_bdd_1} \\
& \langle \mathcal L_{N,M_N}\sqrt  f,\sqrt f \rangle 
 \le  -A_0 \mathcal D_{M_N}^0 (f) + A_0' |\L_{N,M_N}| N^{-2} 
\|\sqrt f\|_{L^2(\nu_{\widehat \theta,M_N}^N)}^2 , \label{estimates_bdd_2} \\
& \langle \mathbb L_{N,M_N} \sqrt f,\sqrt f\rangle 
 \le 
 A_1 |\L_{N,M_N}| \|\sqrt f\|_{L^2(\nu_{\widehat \theta,M_N}^N)}^2, \label{estimates_bdd_3}
\end{align}
 where $|\L_{N,M_N}|=(2N+1)(2{M_N}+1)^{d-1}$ stands for the cardinality of the set $\L_{N,M_N}$. 
\end{lemma}
\begin{proof}
The reversibility of $\nu_{\widehat \theta,M_N}^N$ with respect to the generator 
$L_{\widehat b,N,M_N}$ yields \eqref{estimates_bdd_1}.\par 
To prove \eqref{estimates_bdd_2}, 
observe that for all $A,B>0$, $\displaystyle A(B-A) = -\frac12(B-A)^2 +\frac12 (B^2-A^2)$, so that
\begin{align}\nonumber
& \langle \mathcal L_{N,M_N} \sqrt f, \sqrt f\rangle
 = \sum\limits_{k=1}^d\sum\limits_{x,x+e_k \in \Lambda_{N,M_N}} 
 \int \sqrt { f(\xi,\omega) } \Big(\sqrt {f(\xi^{x,x+e_k},\omega^{x,x+e_k})} 
 - \sqrt{ f(\xi,\omega) }\Big) d\nu_{\widehat \theta,M_N}^N (\xi,\omega)\\
&  = -\dfrac{1}{2}\mathcal D_{M_N}^0 (f) 
+ \dfrac{1}{2} \sum\limits_{k=1}^d\sum\limits_{x,x+e_k \in \Lambda_{N,M_N}}
\int  \mathcal L^{x,x+e_k} f(\xi,\omega) d\nu_{\widehat \theta,M_N}^N (\xi,\omega)\, .\label{forbd2-1}
\end{align}
By  Consequences \ref{csq:chg-var}\textit{(i)}, we have
\begin{equation}\label{forbd2-2}
\dfrac{1}{2} 
\int  \mathcal L_{N,M_N} f(\xi,\omega) d\nu_{\widehat \theta,M_N}^N (\xi,\omega)
\le \dfrac{1}{4}\mathcal D_{M_N}^0 (f) +
   C |\L_{N,M_N}| N^{-2} \| \sqrt f\|_{L^2(\nu_{\widehat \theta,M_N}^N)}^2 
\end{equation}
for some positive constant $C$.  Putting together \eqref{forbd2-1} and \eqref{forbd2-2}
we obtain \eqref{estimates_bdd_2}. \par 
To prove  \eqref{estimates_bdd_3}, we have
\begin{align*}
& \langle \mathbb L_{N,M_N} \sqrt f, \sqrt f\rangle = I_1 + I_2\\
& \quad := \sum\limits_{x \in \Lambda_{N,M_N}} 
\int \Big(\beta_{\Lambda_{N,M_N}}(x,\xi,\omega) (1-\xi(x))+\xi(x)\Big) \sqrt {f(\xi,\omega)} 
\Big( \sqrt {f(\sigma^x \xi,\omega)} - \sqrt {f(\xi,\omega)} \Big) d\nu_{\widehat \theta,M_N}^N(\xi,\omega)\\
& \qquad + \sum\limits_{x \in \Lambda_{N,M_N}} \int \Big(r (1-\omega(x))+\omega(x)\Big) 
\sqrt {f(\xi,\omega)} \Big(\sqrt {f(\xi,\sigma^x \omega)} - \sqrt {f(\xi,\omega)} \Big) 
d\nu_{\widehat \theta,M_N}^N(\xi,\omega)\, .\\
\end{align*}
For $I_1$, using first that all the rates are bounded,
then \eqref{ineq2} with $A=2N$, we have
\begin{align*}
&  I_1 \leq C(\lambda_1,\lambda_2,r) \sum\limits_{x \in \Lambda_{N,M_N}} 
\int \Big(\sqrt {f(\xi,\omega)} \sqrt {f(\sigma^x \xi,\omega)}
 +  {f(\xi,\omega)} \Big) d\nu_{\widehat \theta,M_N}^N(\xi,\omega)\\
& \quad \leq C(\lambda_1,\lambda_2,r) \sum\limits_{x \in \Lambda_{N,M_N}} 
 \int\Big(  f(\xi,\omega) + \dfrac{1}{4}  f(\sigma^x \xi,\omega) 
 +  f(\xi,\omega) \Big) d\nu_{\widehat \theta,M_N}^N(\xi,\omega)
\end{align*}
for some  constant $C(\lambda_1,\lambda_2,r)$. 
We conclude using the change of variables \eqref{change_jump2} and \eqref{eq:cC}.
We proceed similarly for $I_2$, using the change of variables \eqref{change_jump3}.
\end{proof}

\medskip
\subsection{Boundary conditions in large finite volume} 
Recall that $\mu_N$ stands for the initial measure and 
$\mu_{N,M_N}$ for the marginal of $\mu_N$ on $\widehat \Sigma_{N,M_N}$ (cf. \eqref{eq:marginal}).
Recall from Section \ref{sec:hydrodynamics} that we denoted by 
${\widetilde {\mathbb P}}_{\mu_N}^{M_N,\widehat b}$ the law of the 
finite volume process $(\zeta_t,\chi_t)_{t\in [0,T]}$, by
${\widetilde {\mathbb E}}_{\mu_N}^{M_N,\widehat b}$ the corresponding  expectation and by
$(\widetilde \eta_1,\widetilde \eta_2,\widetilde \eta_3)$ then
$\widetilde \eta_0$ the associated conserved quantities for the exchange dynamics as in 
\eqref{omega}--\eqref{omega-compl}.
\begin{proof}[Proof of Proposition \ref{bddinf2}]
Let $\widehat \theta =(\theta_1,\theta_2,\theta_3):\L\to (0,1)^3$ be a smooth profile
satisfying \eqref{eq:cC} and \eqref{b-for-rev}.
Denote by $\nu_{\widehat \theta,\L_N\setminus \L_{N,M_N}}^N$ the marginal of $\nu^N_{\widehat \theta}$ on 
$\big(\{0,1\}\times \{0,1\}\big)^{\L_N\setminus \L_{N,M_N}}$, 
and by $\widetilde \mu^{N,M_N}_{\widehat \theta}$ the measure on  $\widehat \Sigma_{N}$
given by 
\begin{equation}\label{def:mu-comp}
\widetilde \mu^{N,M_N}_{\widehat \theta}
=\mu_{N,M_N}\otimes \nu_{\widehat \theta,\L_N\setminus \L_{N,M_N}}^N.\end{equation}
Let $H : [0,T] \times \Gamma\to \R$ be a bounded function with compact support contained 
in $\{-1,1\}\times[-K,K]^{d-1}\subset\Gamma$ for some $K>0$; let $\delta >0$, and $i\in\{1,2,3\}$. 
To shorten the notation, we denote 
\begin{eqnarray*}\label{block-W}
B_i^{H_t} (\zeta_t,\chi_t) 
&=& \dfrac{1}{N^{d-1}} \sum\limits_{x \in \L_{N,M_N}\cap\Gamma_N} H_t(x/N)  
\Big( {\widetilde \eta}_{i,t}(x) - b_i(x/N) \Big)\\
 &=& \dfrac{1}{N^{d-1}} \sum\limits_{x \in \L_{N,NK}\cap\Gamma_N} H_t(x/N)  
\Big( {\widetilde \eta}_{i,t}(x) - b_i(x/N)\Big)\, . 
\end{eqnarray*}
Because  $H$ has compact support  
in $\Gamma$, we have (cf. \eqref{def:mu-comp}),
\begin{equation*}
{\widetilde {\mathbb P}}_{\mu_N}^{M_N,\widehat b} \Big( \Big| \int_0^T 
B_i^{H_t} (\zeta_t,\chi_t)
 dt \Big|>\delta \Big) 
 =  {\widetilde {\mathbb P}}_{\widetilde \mu^{N,M_N}_{\widehat \theta}}^{M_N,\widehat b}\Big( \Big| \int_0^T 
B_i^{H_t} (\zeta_t,\chi_t)
 dt \Big|>\delta \Big)\, .
\end{equation*}
To prove the proposition it is enough to show that
\begin{equation*}
\limsup\limits_{N \rightarrow \infty} 
\frac{1}{|\L_{N,M_N}|}\log 
{\widetilde {\mathbb P}}_{\widetilde \mu^{N,M_N}_{\widehat \theta}}^{M_N,\widehat b}
\Big( \Big| \int_0^T B_i^{H_t} (\zeta_t,\chi_t)
 dt \Big| > \delta \Big) = -\infty\, .
\end{equation*}
Since the Radon-Nikodym derivative $\displaystyle{\frac{d\mu_{N,M_N}}{d\nu_{\widehat \theta,M_N}^N}}$ 
is bounded by $\exp(N M_N^{d-1}K_1)$ for some positive constant $K_1$, by \eqref{def:mu-comp}
it is enough to show that
\begin{equation*}
\limsup\limits_{N \rightarrow \infty} 
\frac{1}{|\L_{N,M_N}|}\log 
{\widetilde {\mathbb P}}_{\nu_{\widehat \theta}^N}^{M_N,\widehat b}
\Big( \Big| \int_0^T B_i^{H_t} (\zeta_t,\chi_t)
 dt \Big| > \delta \Big) = -\infty\, .
\end{equation*}
By exponential Chebyshev's inequality, the expression in the last limit is bounded above by
$$
-a\delta + \frac{1}{|\L_{N,M_N}|}\log 
{\widetilde {\mathbb E}}_{\nu_{\widehat \theta}^N}^{M_N,\widehat b}
\exp \Big(a|\L_{N,M_N}|\Big| \int_0^T B_i^{H_t} (\zeta_t,\chi_t)
 dt \Big|\Big)
$$
for any $a>0$. Using that $e^{|\alpha|} \leq e^\alpha + e^{-\alpha}$ and 
 that, for two generic sequences $(a_L)_{L\in\N},(b_L)_{L\in\N}$ we have 
$$\limsup\limits_L L^{-1} \log \big(a_L + b_L \big) 
\leq \max \Big( \limsup\limits_L L^{-1} \log a_L , \limsup\limits_L L^{-1} b_L \Big),$$
one can pull off the absolute value even if it means replacing $H$ by $-H$.
 Therefore, to prove the proposition, we have to show that,
for any bounded function $H$, there exists a positive constant $C>0$, such that for any $a>0$,
$$
\limsup_{N\to +\infty}\frac{1}{|\L_{N,M_N}|}\log 
{\widetilde {\mathbb E}}_{\nu_{\widehat \theta}^N}^{M_N,\widehat b}
\exp \Big(a|\L_{N,M_N}| \int_0^T B_i^{H_t} (\zeta_t,\chi_t)
 dt \Big)\le C
$$
and then to let $a\uparrow +\infty$.

By Feynman-Kac formula (cf. Appendix 1, Section 7 in \cite{kl}),
\begin{equation}
\begin{aligned}
& 
\frac{1}{|\L_{N,M_N}|}\log 
{\widetilde {\mathbb E}}_{\nu_{\widehat \theta}^N}^{M_N,\widehat b}\exp \Big(a|\L_{N,M_N}| 
\int_0^T B_i^{H_t} (\zeta_t,\chi_t) dt \Big)
 \\ \label{block-W_eq1}
& \ \leq \int_0^T \sup\limits_{f} 
\Bigg\{ \int a B_i^{H_t} (\zeta,\chi) f(\zeta,\chi) d\nu_{\widehat \theta,M_N}^N (\zeta, \chi) 
+ |\L_{N;M_N}|^{-1} \langle \mathfrak L_{N,M_N} \sqrt f, \sqrt f \rangle \Bigg\} dt \, ,
\end{aligned}
\end{equation}
where the supremum is carried over all densities $f$ with respect to $\nu_{\widehat \theta,M_N}^N$. 
By Lemma \ref{estimates_bdd}, 
\begin{equation}\label{by-5.1}
\begin{aligned}
\langle \mathfrak L_{N,M_N} \sqrt f, \sqrt f \rangle &
\leq - N^2 D_{M_N} ( f) + A_0 |\L_{N,M_N}| \leq - N^2 D_{M_N}^{\widehat b} ( f) + A_0 |\L_{N,M_N}|
\end{aligned}
\end{equation}
for some positive constant $A_0$, where $D_{M_N}$ and $D_{M_N}^{\widehat b}$ are defined  in
\eqref{def:Dn}--\eqref{def:dir-form-reaction}. \\

To evaluate $\displaystyle \int a B_i^{H_t} (\zeta,\chi) f(\zeta,\chi) d\nu_{\widehat \theta,M_N}^N (\zeta, \chi) $, observe first that for all  $x\in \L_{N,M_N}$, $1\le i\le 3$, 
\begin{align}\label{replacement_eq2}
  \widetilde\eta_i(x)-b_i(x/N) 
& =  \widetilde\eta_i(x)(1-b_i(x/N)) -b_i(x/N)(1-\widetilde\eta_i(x))  \nonumber \\
&  = \sum\limits_{0\le j \neq i\le 3} \Big( \widetilde\eta_i(x)  b_j(x/N) - 
b_i(x/N)  \widetilde\eta_j(x/N) \Big) .
\end{align} 
We detail for instance the case $i=1$, the others follow the same way.
\begin{align}
& \int a B_1^{H_t} (\zeta,\chi) f(\zeta,\chi) d\nu_{\widehat \theta,M_N}^N (\zeta, \chi) 
=\frac{a}{N^{d-1}}\sum_{x\in \L_{N,NK}\cap\Gamma_N}
\int H_{t}(x/N) \Bigg( \Big( \widetilde\eta_1(x) b_0(x/N)  - b_1(x/N) \widetilde\eta_0(x)\Big)\nonumber \\ \nonumber
& \qquad  + \Big( \widetilde\eta_1(x) b_3(x/N)  - b_1(x/N) \widetilde\eta_3(x)\Big)
+ \Big( \widetilde\eta_1(x) b_2(x/N)  - b_1(x/N) \widetilde\eta_2(x)\Big) \Bigg)
f(\zeta,\chi) d  \nu_{\widehat \theta,M_N}^N (\zeta,\chi) \\ \nonumber
& = \frac{a}{N^{d-1}}\sum_{x\in \L_{N,NK}\cap\Gamma_N} 
\int H_{t}(x/N) \Bigg( b_1 (x/N)\widetilde\eta_0(x)\Big( f(\sigma^x \zeta,\chi) - f(\zeta,\chi) \Big)\\ \nonumber
& \qquad  
 + b_1(x/N) \widetilde\eta_3(x)\Big(f(\zeta,\sigma^x \chi) - f(\zeta,\chi) \Big) 
+b_1 (x/N)\widetilde\eta_2(x)\Big(f(\sigma^x\zeta,\sigma^x \chi) - f(\zeta,\chi) \Big) \Bigg)
 d  \nu_{\widehat \theta,M_N}^N (\zeta,\chi)\\ \nonumber
&  \leq \frac{N^2}{2|\L_{N,M_N}|}\sum_{x\in \L_{N,NK}\cap\Gamma_N}\int  b_1 (x/N)
\Bigg(\widetilde\eta_0(x) \Big( \sqrt{f(\sigma^x \zeta,\chi)} - \sqrt{f(\zeta,\chi) }\Big)^2 \\\nonumber
&\qquad\quad + \widetilde\eta_3(x)\Big(\sqrt{f(\zeta,\sigma^x \chi)} - \sqrt{f(\zeta,\chi)} \Big)^2  
+ \widetilde\eta_2(x)\Big(\sqrt{f(\sigma^x\zeta,\sigma^x \chi)} - \sqrt{f(\zeta,\chi)} \Big)^2 \Bigg)
d\nu_{\widehat \theta,M_N}^N (\zeta,\chi) \\ \label{up-for-B1}
& \qquad + \dfrac{C'_1 a^2}{2}\frac{|\L_{N,M_N}|}{N^{d+1}} || H^2 ||_\infty\, ,
\end{align}
where $C_1'$ is some positive constant. 
We used \eqref{replacement_eq2},
 changes of variables given by \eqref{change_jump1}--\eqref{change_jump3} 
 with \eqref{b-for-rev} for the  two equalities. 
Next, we used \eqref{ineq2} with $A=|\L_{N,M_N}|aN^{-d}$,
and that $f$ is a probability density, 
with a computation similar to the one done for Consequences \ref{csq:chg-var}\textit{(i)}.\\ 
Note that the last sum above is a part of the Dirichlet
form for the boundaries $D_{M_N}^{\widehat b} ( f)$ (see \eqref{def:Dnb}, \eqref{def:Dnbx}).\\ 

Collecting all previous estimates, that is \eqref{block-W_eq1},
\eqref{by-5.1}, \eqref{up-for-B1} and the similar ones
for $B_i,i\in\{0,2,3\}$, 
we proved that
\begin{equation}\label{eq:last-3.3}
\begin{aligned}
& \frac{1}{|\L_{N,M_N}|}\log 
{\widetilde {\mathbb E}}_{\mu_N}^{M_N,\widehat b}\exp \Big(a|\L_{N,M_N}| \Big| \int_0^T B_i^{H_t} (\zeta_t,\chi_t)
 dt \Big|\Big)
\leq  T\frac{C'_1 a^2}{2}\frac{|\L_{N,M_N}|}{N^{d+1}} || H^2 ||_\infty + TA_0\, .
\end{aligned}
\end{equation}
Because $M_N=N^{1+\frac1d}$, the r.h.s. of \eqref{eq:last-3.3} goes to $TA_0$
when 
$N\uparrow +\infty$, which concludes the proof.
\end{proof}

\section{Hydrodynamic limit in infinite volume}\label{sec:hydro_infini}
 \subsection{The coupled process}\label{subsec:coupl-proc}
We couple $(\zeta_t,\chi_t)_{t\in [0,T]}$ to our original dynamics in infinite volume
$(\xi_t,\omega_t)_{t\in [0,T]}$.  We denote with a ``bar'' everything related to this
coupling. As explained in the introduction,
our model allows the use of basic coupling for each of the involved dynamics; namely, 
the coupled particles try to behave similarly as much as possible. 
The generator ${\overline {\mf L}}_{N}$ of the coupled process  is given by 
$$ {\overline {\mf L}}_{N}= N^2{\overline {\mc L}}_{N} + {\overline {\mathbb L}}_{N} 
+  N^2{\overline { L}}_{\widehat b, N}\, . $$
For this coupling, we shall come back to the initial equivalent formulation of configurations:
$\eta\in\{0,1,2,3\}^{\Lambda_N}$ corresponds to $(\xi,\omega)$, and  
$\widetilde\eta\in\{0,1,2,3\}^{\Lambda_N}$ to $(\zeta,\chi)$  via the application \eqref{eq:omxieta},
that is $\eta=\eta((\xi,\omega)),\widetilde\eta=\widetilde\eta((\zeta,\chi))$. 
For $x \in \Lambda_N$, $k\in\{1,...,d\}$, define $\eta^{x,x+ e_k}$ to be
the configuration obtained from $\eta$ by exchanging the values of $\eta$
at $x$ and $x+ e_k$. Notice that via \eqref{eq:2.0},  
$\eta^{x,x+e_k}$ is equivalent to $(\xi^{x,x+e_k},\omega^{x,x+e_k})$. \\ \\
Define the coupled generator for the exchange part by 
\begin{equation}\label{def:Lcoupl-exch}
{\overline {\mathcal L}}_{N}={\overline {\mathcal L}}_{N}^1
+{\overline {\mathcal L}}_{N}^2\, ,
\end{equation} 
where, for a cylinder function $f$ on $(\{0,1,2,3\}^{\Lambda_N})^2$, 
and $(\xi,\omega)\in{\widehat \Sigma}_N,(\zeta,\chi)\in{\widehat \Sigma}_N$, 
abbreviating  $\widetilde \eta$ for $\widetilde \eta((\zeta,\chi))$, 
and $\eta$ for $\eta((\xi,\omega))$ (which will be done in all this section), we have 
\begin{eqnarray}
{\overline {\mathcal L}}_{N}^1 f(\widetilde \eta,\eta) &
=&  \sum_{k=1}^d\sum_{x,x+ e_k \in \Lambda_{N,M_N}} 
{\overline {\mathcal L}}^{1,(x,x+ e_k)}f(\widetilde \eta,\eta)
\qquad\mbox{with}\nonumber\\
{\overline {\mathcal L}}^{1,(x,x+ e_k)} f(\widetilde \eta,\eta)&=& 
 f(\widetilde \eta^{x,x+e_k},\eta^{x,x+ e_k})
- f(\widetilde \eta,\eta) \, ,\label{Lcoupl-exch3}\\
{\overline {\mathcal L}}_{N}^2 f(\widetilde \eta,\eta)   & =&
\sum_{k=2}^d \sum_{x\in {\mathcal A}(k)}
{\overline {\mathcal L}}^{2,(x,x+ e_k)}f(\widetilde \eta,\eta)
\qquad\mbox{with}\nonumber\\
{\overline {\mathcal L}}^{2,(x,x+ e_k)} f(\widetilde \eta,\eta)&=&
 f(\widetilde \eta,\eta^{x,x+ e_k}) - f(\widetilde \eta, \eta)\, ,\label{Lcoupl-exch4}
\end{eqnarray}
where
${\mathcal A}(k)=\{x\in \Lambda_N: (x,x+e_k) \notin \Lambda_{N,M_N}\times\Lambda_{N,M_N}\}$.
This means that
either $(x,x+e_k) \in (\Lambda_{N,M_N}\times\Lambda_{N,M_N}^c)\cup(\Lambda_{N,M_N}^c\times\Lambda_{N,M_N})$, or $(x \notin \Lambda_{N,M_N}$ and $x+e_k\notin\Lambda_{N,M_N})$  
(recall that $\Lambda_{N,M_N}^c$ was defined in \eqref{def:LNn-ext}). \\ \\
For $1\le j\le 3$ and $x\in \L_N$, denote by $\eta_x^j$ 
the configuration obtained from $\eta$ by flipping the state of $x$ to $j$.
A (basic) coupling for the reaction part is given by  
${\overline{\mathbb L}_N}$ (see \cite[Proposition 4.2]{k}  or \cite{phd}
for details and justifications for this coupling): 
\begin{equation}\label{eq:gen-coupl-cprs}
 {\overline {\mathbb L}_N}= 
{\overline {\mathbb L}}^{1,a}_N+{\overline {\mathbb L}}^{1,b}_N + {\overline {\mathbb L}}^2_N ,
\end{equation}
where ${\overline {\mathbb L}}^{1,a}_N$ stands for the generator with coupled flips, 
${\overline {\mathbb L}}^{1,b}_N $ for uncoupled flips within $\Lambda_{N,M_N}$ and 
${\overline {\mathbb L}}^2_N$ for uncoupled flips occurring 
at the boundary of $\Lambda_{N,M_N}$ or outside $\Lambda_{N,M_N}$, that is,
\begin{small}
\begin{align}\nonumber
& {\overline {\mathbb L}}^{1,a}_N f(\widetilde \eta ,\eta) = 
\sum\limits_{x \in \Lambda_{N,M_N}} 
{\overline {\mathbb L}}^{1,a,x}_N f(\widetilde \eta ,\eta)\quad\mbox{with}\\\nonumber
&{\overline {\mathbb L}}^{1,a,x}_N f(\widetilde \eta ,\eta)=
\Big\{ \big( \widetilde \eta_0(x) \eta_0(x)  \overline \beta_{M_N}(x,\widetilde \eta, \eta)
  + \eta_3(x) \widetilde \eta_3(x) \big) \Big\}
\big[ f(\widetilde \eta_x^1, \eta_x^1) - f(\widetilde \eta,\eta) \big]\\\nonumber
&  + \Big\{  \widetilde \eta_1(x) \eta_1(x)  +\widetilde \eta_2(x)  \eta_2(x) \Big\} 
\big[ f(\widetilde \eta_x^0, \eta_x^0) - f(\widetilde \eta,\eta) \big]
  + \Big\{  r \widetilde \eta_0(x) \eta_0(x)  + \widetilde \eta_3(x) \eta_3(x) \Big\}
\big[ f(\widetilde \eta_x^2, \eta_x^2) - f(\widetilde \eta,\eta) \big]\\\nonumber
&  + \Big\{\widetilde \eta_2(x)  \eta_2(x) \overline\beta_{M_N}(x,\widetilde\eta, \eta)
  + r\widetilde \eta_1(x)  \eta_1(x) \Big\}
\big[f(\widetilde \eta_x^3, \eta_x^3) - f(\widetilde \eta,\eta) \big]\\\nonumber
 & + \Big\{ \widetilde \eta_2 (x) \eta_0(x) 
\overline\beta_{M_N}(x,\widetilde \eta,\eta)  \Big\}
\big[f(\widetilde \eta_x^3, \eta_x^1) - f(\widetilde \eta, \eta)\big]
+ \Big\{\widetilde \eta_0 (x) \eta_2(x)\overline\beta_{M_N}(x,\widetilde \eta,\eta)  \Big\}
\big[f(\widetilde \eta_x^1, \eta_x^3) -f(\widetilde \eta, \eta)\big]\\\nonumber
&  + \Big\{ \widetilde \eta_2 (x) \eta_3(x) \Big\}
\big[f(\widetilde \eta_x^0, \eta_x^1) -f(\widetilde \eta, \eta)\big]  
+ \Big\{ \widetilde \eta_3 (x) \eta_2(x)  \Big\} 
\big[f(\widetilde \eta_x^1, \eta_x^0) -f(\widetilde \eta, \eta)\big]\\\nonumber
&  + \Big\{ \widetilde \eta_3 (x) \eta_1(x)  \Big\}
\big[f(\widetilde \eta_x^2, \eta_x^0) - f(\widetilde \eta, \eta)\big]
+ \Big\{ \widetilde \eta_1 (x) \eta_3(x)  \Big\} 
\big[f(\widetilde \eta_x^0, \eta_x^2)-f(\widetilde \eta, \eta)\big] \\\label{eq:gen-coupl-cprs-1a}
&  + \Big\{ \widetilde \eta_1 (x) \eta_0(x)  r \Big\} 
\big[f(\widetilde \eta_x^3, \eta_x^2) -f(\widetilde \eta, \eta)\big] 
+ \Big\{ \widetilde \eta_0 (x) \eta_1(x) r \Big\}
\big[f(\widetilde \eta_x^2, \eta_x^3) - f(\widetilde \eta, \eta)\big] \, .
\end{align}
\end{small}
Let $\beta_{M_N} (\cdot,\cdot)$ be the growth rate 
on $\Lambda_{N,M_N}$ defined via   \eqref{eq:omxieta}, \eqref{omega} and \eqref{def:ratesCP-DRE} by 
$$\beta_{M_N} (x,\eta) =  \beta_{M_N} (x,\eta((\xi,\omega))) =\beta_{\Lambda_{N,M_N}} (x,\xi,\omega) = 
\sum\limits_{y \in \Lambda_{N,M_N} \atop \|y-x\|=1} 
\Big\{ \lambda_1  \eta_1(y)+ \lambda_2 \eta_3(y) \Big\},$$
 and 
\begin{eqnarray*}
\overline\beta_{M_N}(x,\widetilde \eta,\eta) 
&=&  \overline\beta_{M_N}(x,\widetilde \eta((\zeta,\chi)),\eta((\xi,\omega))) 
=\beta_{M_N}(x,\widetilde \eta)\wedge\beta_{M_N}(x,\eta),\\
\overline \beta_{M_N}^{(1)}(x,\widetilde \eta,\eta)
&=& \beta_{M_N}(x,\widetilde \eta) - \overline\beta_{M_N}(x,\widetilde\eta, \eta), 
\quad \overline \beta_{M_N}^{(2)}(x,\widetilde \eta,\eta)
 = \beta_{M_N}(x, \eta) - \overline\beta_{M_N}(x,\widetilde\eta, \eta) .
 \end{eqnarray*}
 Then
\begin{small}
\begin{align}\nonumber
& {\overline {\mathbb L}}^{1,b}_N f(\widetilde \eta ,\eta) = 
\sum\limits_{x \in \Lambda_{N,M_N}} 
{\overline {\mathbb L}}^{1,b,x}_N f(\widetilde \eta ,\eta)\quad\mbox{with}\\\nonumber
&{\overline {\mathbb L}}^{1,b,x}_N f(\widetilde \eta ,\eta)= 
  \Big\{ \widetilde \eta_0(x) \Big( 
  \eta_0(x) \overline\beta_{M_N}^{(1)}(x,\widetilde \eta,\eta)+ \big(\eta_1(x) +\eta_3(x)
   + \eta_2(x)\big) \beta_{M_N} (x,\widetilde \eta)\Big)  \\\nonumber
&  \qquad\qquad\quad + \widetilde \eta_3 (x) \Big(\eta_0(x)+\eta_1(x) \Big) \Big\}
\big[f(\widetilde \eta_x^1, \eta) - f(\widetilde \eta,\eta) \big]\\\nonumber
& +  \Big\{ \eta_0(x)\Big(\big(\widetilde \eta_0(x) + \widetilde \eta_2 (x)  \big) 
\overline\beta_{M_N}^{(2)}(x,\widetilde \eta,\eta) 
+ \big( \widetilde \eta_1 (x) +\widetilde \eta_3 (x)\big) \beta_{M_N} (x,\eta)\Big)
 + \widetilde \eta_0(x)\eta_3(x)\Big\}\\\nonumber
& \quad\times 
\big[f(\widetilde \eta, \eta_x^1) - f(\widetilde \eta,\eta)\big]\\\nonumber
&  + \Big\{ \widetilde \eta_2(x)  \Big( \big(\eta_2(x)+ \eta_0(x) \big) 
\overline\beta_{M_N}^{(1)}(x,\widetilde \eta,\eta) + \big(\eta_1(x)+ \eta_3(x)\big)
 \beta_{M_N}(x,\widetilde \eta)  \Big)+ \widetilde \eta_1 (x) \Big( \eta_2(x)
 +  \eta_3(x) \Big) r  \Big\}\\\nonumber
& \quad\times 
\big[f(\widetilde \eta_x^3, \eta) - f(\widetilde \eta,\eta) \big]\\\nonumber
&  + \Big\{ \eta_2(x)\Big( \big(\widetilde \eta_0 (x) + \widetilde \eta_2(x) \big)
\overline\beta_{M_N}^{(2)}(x,\widetilde \eta,\eta) + 
\big( \widetilde \eta_1 (x) +  \widetilde \eta_3 (x) \big)
\beta_{M_N}(x,\eta) \Big)
+ \Big(\widetilde \eta_2(x)+ \widetilde \eta_3 (x)\Big) \eta_1(x) r\Big\}\\\nonumber
& \quad\times 
\big[f(\widetilde \eta, \eta_x^3) - f(\widetilde \eta,\eta) \big]\\\nonumber
& + \Big\{ \widetilde \eta_2 (x) \Big( \eta_0(x)  +\eta_1(x)\Big) 
 + \widetilde \eta_1 (x)\Big(\eta_0 (x)+ \eta_2(x) \Big) \Big\}
\big[f(\widetilde \eta_x^0, \eta) - f(\widetilde \eta, \eta)\big]\\\nonumber
&  + \Big\{ \Big( \widetilde \eta_2 (x) + \widetilde \eta_3 (x) \Big) \eta_0(x) r
+  \Big( \widetilde \eta_0 (x) + \widetilde \eta_2 (x) \Big) \eta_3(x)  \Big\}
\big[f(\widetilde \eta, \eta_x^2) - f(\widetilde \eta, \eta)\big]\\\nonumber
&  + \Big\{ \Big( \widetilde \eta_0 (x) + \widetilde \eta_2 (x) \Big) \eta_1(x)
+ \Big( \widetilde \eta_0 (x) + \widetilde \eta_1 (x) \Big) \eta_2(x)  \Big\}
 \big[f(\widetilde \eta, \eta_x^0) - f(\widetilde \eta, \eta)\big]\\\label{eq:gen-coupl-cprs-1b}
&  + \Big\{ \big(\widetilde \eta_0 (x) \Big( \eta_3(x) + \eta_2(x)\Big) r  + \widetilde \eta_3 (x)  \Big(\eta_0(x) +  \eta_2(x) \Big)  \Big\}
\big[f(\widetilde \eta_x^2, \eta) - f(\widetilde \eta, \eta)\big] \, .
\end{align}
\end{small}
 Finally
\begin{eqnarray}\label{eq:gen-coupl-cprs-2}
&&\ \overline {\mathbb L}^2_N f(\widetilde \eta,\eta) 
=  \sum\limits_{x \in \Lambda_{N,M_N}}
\overline {\mathbb L}^{2,a,x}_N f(\widetilde \eta,\eta) 
 + \sum\limits_{x \notin \Lambda_{N,M_N}} \overline {\mathbb L}^{2,b,x}_N f(\widetilde \eta,\eta),
\quad\mbox{with}\\\nonumber
&&\overline {\mathbb L}^{2,a,x}_N f(\widetilde \eta,\eta)=
 \eta_0(x) \beta_{N,M_N}^{out}(x,\eta) 
\big[ f(\widetilde \eta, \eta_x^1) - f(\widetilde \eta,\eta) \big] 
 + \eta_0(x) \beta_{N,M_N}^{out}(x,\eta) 
 \big[f(\widetilde \eta, \eta_x^3) - f(\widetilde \eta,\eta) \big] \, ,
\end{eqnarray}
where, for $x \in \Lambda_{N,M_N}$, 
$\beta_{N,M_N}^{out}(x,\eta) = 
\sum\limits_{y\in \Lambda_{N,M_N}^c\atop \|y-x\|=1} 
\Big\{ \lambda_1 \eta_1(y)+ \lambda_2  \eta_3(y) \Big\}$;  and, 
for $x \notin \Lambda_{N,M_N}$, the transitions produce changes 
only on the second configuration $\eta$, 
in a similar way as
in ${\overline {\mathbb L}}^{1,a}_N+{\overline {\mathbb L}}^{1,b}_N$, 
but with $\beta_{M_N}$ replaced by $\beta_{\Lambda_N}$. 
Since we shall not use this second part of the
generator $\overline {\mathbb L}^2_N$ in our computations, we do not detail it.\par
\smallskip
Note that the generator $L_{\widehat b,N}$ (see \eqref{def:LbN}) can be rewritten as,
 for a cylinder function $g$ on $\{0,1,2,3\}^{\Lambda_N}$,  and $\eta=\eta((\xi,\omega))$, 
\begin{eqnarray}\label{Lb-alternative}
L_{\widehat b,N} g(\eta) &=& \sum\limits_{x \in \Gamma_N} \sum\limits_{j=0}^3 
b_j(x/N) \big(1-\eta_j(x)\big) \big( g (\eta_x^j) - g(\eta)\big)\nonumber\\
&=&  \sum\limits_{x \in \Gamma_N} \sum\limits_{j=0}^3 
b_j(x/N)  \big( g (\eta_x^j) - g(\eta)\big).\label{Lb-alternative}
\end{eqnarray}
We construct now a coupled dynamics in which, on each site $x$,
the state in both $\eta$ and $\widetilde \eta$ changes to the same state $j$ at rate $b_j(x/N)$
when possible, that is within $\Lambda_{N,M_N}$:
\begin{eqnarray}\label{Lb-coupl-alternative}
&&{\overline { L}}_{\widehat b, N}
={\overline { L}}_{\widehat b, N}^1+{\overline { L}}_{\widehat b, N}^2\qquad\mbox{where}\\
 &&{\overline { L}}_{\widehat b, N}^1 
=
\sum\limits_{x \in \Gamma_N\cap\Lambda_{N,M_N}} 
{\overline { L}}_{\widehat b, N}^{1,x} 
\qquad\mbox{with}\qquad
{\overline { L}}_{\widehat b, N}^{1,x} f(\widetilde \eta ,\eta)
=\sum\limits_{i=0}^3 b_i(x/N) 
\big(f(\widetilde \eta_x^i, \eta_x^i) 
- f(\widetilde \eta,\eta)\big) \label{Lb-coupl-alternative_1}  \\
&&{\overline { L}}_{\widehat b, N}^2 
=\sum\limits_{x: x \in \Gamma_N,\atop x \notin\Lambda_{N,M_N}} 
{\overline { L}}_{\widehat b, N}^{2,x} 
\qquad\mbox{with}\qquad
{\overline { L}}_{\widehat b, N}^{2,x} f(\widetilde \eta ,\eta)=
\sum\limits_{i=0}^3 b_i(x/N)
\big(f(\widetilde \eta, \eta_x^i) - f(\widetilde \eta,\eta)\big) \label{Lb-coupl-alternative_2} .
\end{eqnarray}
We now have all the material to investigate the specific entropy 
and Dirichlet form for the coupled process.
Recall from Section \ref{sec:specific} that 
$\nu^N_{\widehat\theta}:=\nu^N_{\widehat\theta(\cdot)} $ is a product probability measure on
$\widehat \Sigma _N$, 
where $\widehat\theta=(\theta_1,\theta_2,\theta_3):\Lambda \to (0,1)^3$ 
is  a smooth function such that 
$\widehat\theta (\cdot){\big\vert_{\Gamma}} =\; \widehat b (\cdot)$. Let $\mu$ be a probability measure on
$\widehat \Sigma _N$ and denote by
$\overline{\mu}= \mu\otimes \mu$, $\overline\nu^N_{\widehat\theta} 
=\nu^N_{\widehat\theta}\otimes \nu^N_{\widehat\theta}$
the product measures on $\widehat \Sigma _N\times \widehat \Sigma _N$. 
As in \eqref{def:rel-ent} and \eqref{def:rel-ent_exp}, for a positive integer $n>1$, we define the
entropy of $\overline{\mu}$ with respect to $\overline\nu^N_{\widehat\theta}$ by
\begin{equation}\label{def:rel-ent_exp-couple}
\overline s_n(\overline \mu_n | \overline \nu_{\widehat\theta, n}^N)= \int 
 \log \left( \overline f_n((\zeta,\chi),(\xi,\omega)) \right) 
 d\overline \mu_n ((\zeta,\chi),(\xi,\omega)), 
\end{equation}
where $\overline f_n$ is the probability density of 
$\overline \mu_n$ with respect to $\overline \nu_{\widehat\theta, n}^N$, and
$\overline \mu_n$ (resp. $\overline \nu_{\widehat\theta, n}^N$) stands for the marginal of $\overline\mu$ 
(resp. $\overline \nu_{\widehat\theta}^N$) on ${\widehat \Sigma}_{N,n}\times {\widehat \Sigma}_{N,n}$
(see \eqref{eq:marginal}).

Let $\overline {\mf L}_{N,n}$ 
denote the restriction of the generator $\overline {\mf L}_{N}$
to the box $\Lambda_{N,n}$:  
\begin{eqnarray}\label{def:gen-Nn_couple}
&&\overline {\mf L}_{N,n}=N^2\overline {\mathcal L}_{N,n}\,
+\, \overline {\mathbb L}_{N,n}\, +\, N^2 {\overline L}_{\widehat b,N,n}  \qquad\mbox{with} \\
\label{def:gen-x-couple-exch}
&&\overline {\mathcal L}_{N,n} =\sum\limits_{k=1}^d \sum_{x,x+e_k\in\Lambda_{N,n}}
\big\{{\bf 1}_{\{x,x+ e_k \in \Lambda_{N,M_N}\}} 
{\overline {\mathcal L}}^{1,(x,x+ e_k)}
+{\bf 1}_{\{x\in {\mathcal A}(k)\}}
{\overline {\mathcal L}}^{2,(x,x+ e_k)}\big\}\\\label{def:gen-x-couple-react}
&&\overline {\mathbb L}_{N,n} = \sum_{x\in \Lambda_{N,n}}
\big\{{\bf 1}_{\{x \in \Lambda_{N,M_N}\}} 
({\overline {\mathbb L}}^{1,a,x}_N
+{\overline {\mathbb L}}^{1,b,x}_N+\overline {\mathbb L}^{2,a,x}_N)
+{\bf 1}_{\{x \notin \Lambda_{N,M_N}\}} \overline {\mathbb L}^{2,b,x}_N\big\}\\\label{def:gen-x-couple-bound}
&&\overline {L}_{\widehat b,N,n} = \sum_{x\in \Lambda_{N,n}\cap\Gamma_N} 
\big\{{\bf 1}_{\{x \in \Gamma_N\cap\Lambda_{N,M_N}\}} 
{\overline { L}}_{\widehat b, N}^{1,x}
 +{\bf 1}_{\{ x \in \Gamma_N,x \notin\Lambda_{N,M_N}\}} 
{\overline { L}}_{\widehat b, N}^{2,x}\big\}\, ,
 \end{eqnarray} 
Define the  Dirichlet forms
\begin{equation}\label{def:Dn-couple}
 \overline D_n(\overline{\mu}_n | \overline{\nu}_{\widehat\theta, n}^N ) 
=\overline{ \mathcal D}_n^0 (\overline{\mu}_n | \overline{\nu}_{\widehat \theta ,n}^N)\,
+\, \overline D^{\widehat b}_n  (\overline{\mu}_n | \overline{\nu}_{\widehat \theta ,n}^N) \, ,
\end{equation}
with each one defined similarly to \eqref{def:Dn0}--\eqref{def:dir-form-reaction-x},
but relatively to \eqref{def:gen-x-couple-exch}--\eqref{def:gen-x-couple-bound}.\\
Define  the specific entropy $\overline {\mathcal S} (\overline\mu|\overline{\nu}_{\widehat\theta}^N)$ 
and the Dirichlet form $\overline{\mathfrak D}(\overline\mu|\overline\nu_{\widehat\theta}^N)$ of  $\overline\mu$ with respect to $\overline\nu_{\widehat\theta}^N$ as
\begin{align}\label{def:spec-ent-couple}
\overline{\mathcal S} (\overline\mu|\overline\nu_{\widehat\theta}^N) 
& =N^{-1} \sum_{n\geq 1}\overline s_n(\overline\mu_n | \overline\nu_{\widehat\theta,n}^N)e^{- n/N},\\
\overline{\mf D}(\overline\mu|\overline\nu_{\widehat\theta}^N)  & = 
N^{-1}\sum_{n\geq 1} \overline D_n(\overline\mu_n | \overline\nu_{\widehat\theta,n}^N)e^{-n/N}.
\end{align}
Since the product measure $\overline\nu_{\widehat\theta}^N$ 
is reversible for the boundary generator
$\overline {L}_{\widehat b,N}$, next lemma has a  
 proof  similar to the one of Theorem \ref{dirichlet} (which is
therefore  omitted).
\begin{lemma}\label{dirichlet-couple}
For any time $t\ge 0$, there exists a positive finite constant 
$\overline C_1\equiv \overline C(t,\widehat \theta, \lambda_1, \lambda_2, r)$, so  that 
$$
\int_0^t \overline{\mf D}(\overline\mu (s)|\overline\nu_{\widehat\theta}^N)\, ds\, \le\, 
\overline C_1 N^{d-2}\, .
$$
\end{lemma}
%
\subsection{Proof of 
Proposition \ref{bddinf3} }\label{subsec:proof-th-hydro}
For $i\in\{0,1,2,3\}$,
 let $h_{M_N,i}$ be the function on $\{0,1,2,3\}^{\Lambda_N}\times\{0,1,2,3\}^{\Lambda_N}$ given by,
 for $((\zeta,\chi),(\xi,\omega))\in\widehat\Sigma_N\times \widehat\Sigma_N$ and $\eta=\eta((\xi,\omega)),\widetilde\eta=\widetilde\eta((\zeta,\chi))$, using \eqref{eq:omxieta}, \eqref{omega}, 
\begin{equation}\label{def:hMNi}
h_{M_N,i}(\widetilde \eta,\eta)\;  =\; N^{-d-1} \sum_{n=1}^{M_N-1} e^{-n/N}
H_{i,n} (\widetilde \eta,\eta),\quad\mbox{with }\quad
H_{i,n} (\widetilde \eta,\eta) \, =\, \sum_{x\in \Lambda_{N,n}}
\big|\widetilde\eta_i(x)-\eta_i(x)\big|\, .
\end{equation} 
Let $K$ be such that 
the (compact) support of $\widehat G$ is a subset of $[-1,1]\times [-K,K]^{d-1}$.
For any 
$i\in\{0,1,2,3\}$ and $t\ge 0$, we have
\begin{equation*}
 \begin{aligned}
 \dfrac{1}{N^{d}} \sum\limits_{x \in \Lambda_N}  
\big|G_{i,t}(x/N)\big| \big| {\widetilde \eta}_{i,t}(x) -\eta_{i,t}(x) \big|&
 \le  A_0 \dfrac{1}{N^{d}} \sum_{n=1}^{N}  e^{-(n+KN)/N}\, \frac{1}{N}\sum\limits_{x \in \Lambda_{N,n+KN}}  
 \big| {\widetilde \eta}_{i,t}(x) -\eta_{i,t}(x) \big|\\
 \ & \le A_0 h_{M_N,i}(\widetilde \eta_t,\eta_t)\, ,
 \end{aligned}
\end{equation*}
for some positive constant $A_0=A_0(\widehat G,K)$.
Therefore, in order to prove the proposition it is enough 
to show that,
$$
\lim_{N\rightarrow +\infty}  
{\overline {\mathbb E}}_{\overline\mu_N}^{M_N,\widehat b}
\left[ \sum_{i=0}^3 h_{M_N,i}(\widetilde \eta_{t},\eta_{t}) \right]\; =\; 0\; .
$$
We start by splitting the quantity $h_{M_N,i}(\widetilde \eta_{t},\eta_{t})$ into two parts: 
The sum over all sites $n$, such that $M_N-N \le n\le M_N-1$
and the sum over sites $n<M_N-N$. By Hille-Yosida Theorem
\begin{equation}\label{bof0}
 \begin{aligned}
\partial_t {\overline {\mathbb E}}_{\overline\mu_N}^{M_N,\widehat b}
\left[ \sum_{i=0}^3  h_{M_N,i}(\widetilde \eta_{t},\eta_{t}) \right]
&\, =\, N^{-1-d} \sum^{M_N-1}_{n= M_N-N }e^{-n/N} 
{\overline {\mathbb E}}_{\overline\mu_N}^{M_N,\widehat b} 
\left[  \sum_{i=0}^3\overline {\mf L}_{N} H_{i,n} (\widetilde \eta_{t},\eta_{t})\right]\\
\ &\ \, +\, N^{-1-d} \sum_{n=1}^{M_N-N-1} e^{-n/N}
{\overline {\mathbb E}}_{\overline\mu_N}^{M_N,\widehat b} 
\left[  \sum_{i=0}^3\overline {\mf L}_{N} H_{i,n} (\widetilde \eta_{t},\eta_{t}) \right]\, .
 \end{aligned}
\end{equation}
The first part is bounded by a quantity vanishing when $N\uparrow \infty$. 
Indeed, since for each $x\in \L_N$, 
$\overline {\mf L}_{N} |\widetilde\eta_{i,t}(x)-\eta_{i,t}(x)| \le CN^2 $ 
for some positive constant $C$, we have
\begin{equation}\label{bof1}
N^{-1-d} \sum^{M_N-1}_{n= M_N-N } e^{-n/N} 
{\overline {\mathbb E}}_{\overline\mu_N}^{M_N,\widehat b} 
\left[  \sum_{i=0}^3\overline {\mf L}_{N} H_{i,n} (\widetilde \eta_{t},\eta_{t})\right]
\le K_1 N^{3-d}(M_N)^{d-1}e^{-N^{1/d}},
\end{equation}
for some positive constant $K_1$.

\noindent
We now split the second term of the r.h.s. of \eqref{bof0} 
according to the decomposition of the generator in \eqref{def:gen-Nn_couple}.
Recalling \eqref{def:Lcoupl-exch}-\eqref{Lcoupl-exch4}, 
 we have
\begin{equation*}
\begin{aligned}
&{\overline {\mathbb E}}_{\overline\mu_N}^{M_N,\widehat b} 
\left[ {\overline {\mathcal L}}_{N} H_{i,n}(\widetilde \eta_{t},\eta_{t}) \right]\\
& = 
 \sum\limits_{k=2}^d \sum\limits_{(x,y) \in \Lambda_{N,n}\times\Lambda_{N,n}^c\atop
y\in\{x+e_k,x-e_k\}}  
 \int  \Big[ |\widetilde \eta_{i}(y)-\eta_{i}(y) |
  -|\widetilde \eta_{i}(x)-\eta_{i}(x) |\Big] 
  \overline f_{n+1}^t((\zeta,\chi),(\xi,\omega))
   d \overline \nu_{\widehat\theta}^N((\zeta,\chi),(\xi,\omega))\\
 \ & = 
 \frac{1}{N}\sum_{\ell=1}^N\sum\limits_{k=2}^d 
 \sum\limits_{(x,y) \in \Lambda_{N,n}\times\Lambda_{N,n}^c\atop
y\in\{x+e_k,x-e_k\}}  
 \int  \Big[ |\widetilde \eta_{i}(y)-\eta_{i}(y) | 
 -|\widetilde \eta_{i}(x)-\eta_{i}(x) |\Big] 
 \overline f_{n+\ell}^t((\zeta,\chi),(\xi,\omega)) 
 d \overline \nu_{\widehat\theta}^N((\zeta,\chi),(\xi,\omega))
  \, , 
\end{aligned}
\end{equation*}
where we introduced a new sum in the spirit of \eqref{new-sum}.
Using the equality, for $j\in\{1,2,3\}$ and $y\in \Lambda_N$, 
\begin{equation}\label{eq:absval}
\eta_j(y) (1- \widetilde\eta_j(y))+(1-\eta_j(y)) \widetilde\eta_j(y)
= 
|\eta_j(y)- \widetilde\eta_j(y)|\, ,
\end{equation}
as well as the change of variables \eqref{change_stir} for $\overline f_{n+\ell}^t$,
and analogous computations as for 
Consequences \ref{csq:chg-var}\textit{(i)}, 
using \eqref{ineq2} with $A=N/a$ for some $a>0$ to be chosen later, 
\eqref{change_stir} and \eqref{Rij-Taylor}
(we do not give details as this is similar to what was done in the proof of Lemma \ref{ent_bound}),  
 we obtain, for  $n\le M_N-N-1$,
\begin{eqnarray}
{\overline {\mathbb E}}_{\overline\mu_N}^{M_N,\widehat b} 
\left[ {\overline {\mathcal L}}_{N}
\sum_{i=0}^3 H_{i,n}(\widetilde \eta_{t},\eta_{t}) \right]
 &\le & \frac{a}{N}\sum_{\ell=1}^N\sum\limits_{k=2}^d 
\sum\limits_{(x,y) \in \Lambda_{N,n}\times\Lambda_{N,n}^c\atop
y\in\{x+e_k,x-e_k\}}  
({\overline{\mathcal D}}^0_{n+\ell})^{x,y}
(\overline{\mu}_{n+\ell}(t) | \overline{\nu}_{\widehat \theta ,n+\ell}^N)
\nonumber\\&& \qquad + 
\frac1a C_2  n^{d-2}\, ,\label{bof2}
\end{eqnarray}
for some positive constant $C_2$. 

Now, gathering \eqref{bof1} with \eqref{bof2}, multiplying by $N^{-d-1} \exp(-n/N)$, summing 
over $1 \leq n \leq M_N-1$ and using \eqref{dirformalter},  we get (as in the transition from
\eqref{eq:last-for-s} to \eqref{ent_bound:eq})
\begin{equation}\label{bof3}
{\overline {\mathbb E}}_{\overline\mu_N}^{M_N,\widehat b} 
\left[ N^2{\overline {\mathcal L}}_{N} 
\sum_{i=0}^3 h_{M_N,i}(\widetilde \eta_{t},\eta_{t}) \right]
\le  K_1 N^{3-d}(M_N)^{d-1}e^{-N^{1/d}}
+ aN^{1-d} 
\overline{\mf D}(\overline\mu(t)|\overline\nu_{\widehat\theta}^N) +\frac{C'_2}a 
\, 
\end{equation}
for some positive constant $C'_2$.
Then, recalling \eqref{eq:gen-coupl-cprs}--\eqref{eq:gen-coupl-cprs-2}, we have
\begin{align}
\overline {\mathbb L}_N  |\eta_i(x) - \widetilde \eta_i(x)| 
& \leq  - 2C''_2 |\eta_i(x) - \widetilde \eta_i(x)| 
 + 34 C''_2 \sum\limits_{j =0}^3 (1-\eta_j(x))\widetilde \eta_j(x)\nonumber \\
 & \leq  - 2C''_2|\eta_i(x) - \widetilde \eta_i(x)| 
 + 34 C''_2 \sum\limits_{j =0}^3 |\eta_j(x) - \widetilde \eta_j(x)|\label{bof4}
\end{align}
with  $C''_2 = \max \big( r , 2d\lambda_1 \big) $, where we used \eqref{eq:absval} for the second
inequality.\\ 
Finally, recalling \eqref{Lb-coupl-alternative_1}--\eqref{Lb-coupl-alternative_2},
we have
\begin{align}
& {\overline { L}}_{\widehat b, N}^1|\widetilde \eta_{i}(x)-\eta_{i}(x) |
= -\Big(\sum\limits_{j=0}^3 b_j(x/N)\Big) |\widetilde \eta_i(x) - \eta_i(x)|\leq 0. \label{bof5}
\end{align}
%
Collecting all the above estimates,  we obtain,
for some positive constant $A_1$,
\begin{equation*}
 \begin{aligned}
\partial_t {\overline {\mathbb E}}_{{\overline \mu}_N}^{M_N,\widehat b}
\left[ \sum_{i=0}^3  h_{M_N,i}(\widetilde \eta_{t},\eta_{t}) \right]&
\leq A_1 {\overline {\mathbb E}}_{\mu_N}^{M_N,\widehat b}
\left[ \sum_{i=0}^3  h_{M_N,i}(\widetilde \eta_{t},\eta_{t}) \right] \\ 
\ & \quad +K_1 N^{3-d}(M_N)^{d-1}e^{-N^{1/d}}
+ aN^{1-d} 
\overline{\mf D}(\overline\mu(t)|\overline\nu_{\widehat\theta}^N) + \frac{C'_2}a 
\; .
 \end{aligned}
\end{equation*}
Integrating in time this inequality and using Lemma \ref{dirichlet-couple} gives,
for some positive constant $A_2$, 
\begin{equation*}
 {\overline {\mathbb E}}_{{\overline \mu}_N}^{M_N,\widehat b}
\left[ \sum_{i=0}^3  h_{M_N,i}(\widetilde \eta_{t},\eta_{t}) \right]
\ \le  A_1\int_0^t {\overline {\mathbb E}}_{{\overline \mu}_N}^{M_N,\widehat b}
\left[ \sum_{i=0}^3  h_{M_N,i}(\widetilde \eta_{s},\eta_{s}) \right]ds
 +A_2\left[ N^{3-d}(M_N)^{d-1}e^{-N^{1/d}} +\frac{a}{N}+\frac{1}{a}\right]
\; .
\end{equation*} 
Applying now Gronwall lemma, and choosing $a=\sqrt{N}$, we obtain
$$
{\overline {\mathbb E}}_{{\overline \mu}_N}^{M_N,\widehat b} 
\left[ \sum_{i=0}^3  h_{M_N,i}(\widetilde \eta_{t},\eta_{t}) \right]
\; \le\; A_2 \left[ N^{3-d}(M_N)^{d-1}e^{-N^{1/d}} +\frac{2}{\sqrt{N}}\right]
  e^{C t} \, t \; .
$$
To conclude the proof of the proposition, we just have to recall that 
 by our choice of $M_N$ in \eqref{def:MN},
the sequence $M_N$ is such that  $N^{3-d}(M_N)^{d-1}e^{-N^{1/d}}$ 
decreases to $0$ as $N\uparrow\infty$. 
\section{Uniqueness of weak solutions of equation \eqref{f00}}\label{sec:uniqueness}
This section is devoted to the proof of uniqueness of weak solutions of equation \eqref{f00} in infinite volume with 
 reservoirs. 

We need to introduce the following notation and tools.
Denote by $L^2((-1,1))$ the Hilbert space on the one-dimensional 
bounded interval $(-1,1)$ equipped with the inner product,
\begin{equation*}
\< \varphi,\psi \>_2 =\int_{-1}^1 \varphi(u_1) \, \overline {\psi(u_1)} \, du_1\; ,
\end{equation*}
where, for $z\in\mathbb C$, $\bar z$ is the complex conjugate of $z$ and
$|z|^2 =z{\bar z}$. The norm in $L^2((-1,1))$ is denoted by $\|
\cdot \|_2$.

Let $H^1((-1,1))$ be the Sobolev space of functions $\varphi$ with
generalised derivatives $\partial_{u_1} \varphi$
in $L^2((-1,1))$. Endowed with the scalar product
$\<\cdot, \cdot\>_{1,2}$, defined by
\begin{equation*}
\<\varphi,\psi\>_{1,2} = \< \varphi, \psi \>_2 +\<\partial_{u_1} \varphi \, , \, \partial_{u_1} \psi \>_2\;,
\end{equation*}
$H^1((-1,1))$ is a Hilbert space. The corresponding norm is denoted by
$\|\cdot\|_{1,2}$.
Denote by $H_0^1((-1,1))$
the closure of ${\mathcal C}^\infty_c((-1,1);\R)$ in $H_1((-1,1))$.

Consider the following classical
boundary-eigenvalue problem for the Laplacian:
\begin{equation}
\label{eq:laplace-eq}
 \left\{ \begin{array}{lll}
 -\Delta \varphi =\alpha  \varphi \,,  & \  \\
  \varphi \in  H^1_0((-1,1)) \, . & \ \end{array}
     \right. 
\end{equation}
{}From the Sturm--Liouville theorem (cf. \cite{zui}), one can construct for the problem (\ref{eq:laplace-eq})
a countable system $\{\varphi_n,\alpha_n :n\ge 1\}$ of eigensolutions which contains all possible eigenvalues.
The set $\{\varphi_n : n\ge 1\}$ of eigenfunctions forms a complete
orthonormal system in the Hilbert space $L^2((-1,1))$. Moreover each $\varphi_n$
belongs to $H^1_0 ((-1,1))$ and the set $\{\varphi_n/\alpha_n^{1/2} : n\ge 1\}$ is a complete orthonormal
system in the Hilbert space $H^1_0 ((-1,1))$.  Hence, a function $\psi$
belongs to $L^2((-1,1))$ if and only if
\begin{equation*}
\psi \;=\; \lim_{n\to \infty} \sum_{k=1}^n \<\psi,\varphi_k\>_2 \, \varphi_k 
\quad\mbox{in}\quad L^2((-1,1)).
\end{equation*}
%
In this case, for all $\psi_1,\psi_2\in L^2((-1,1))$
\begin{equation*}
\<\psi_1,\psi_2\>_2 \;=\;  \sum_{k=1}^\infty  \<\psi_1,\varphi_k\>_2\,
\overline{ \<\psi_2,\varphi_k\>_2 }\, .
\end{equation*}
Furthermore, a function $\psi$ belongs to
$H^1_0((-1,1))$ if and only if
\begin{equation*}
\psi \;=\; \lim_{n\to \infty} \sum_{k=1}^n \<\psi,\varphi_k\>_2 \, \varphi_k
\quad\mbox{in}\quad H^1_0((-1,1)),\quad\mbox{and} \\
\end{equation*}
\begin{equation}
\label{f04}
\<\psi_1,\psi_2\>_{1,2} \;=\;  \sum_{k=1}^\infty  \alpha_k \<\psi_1,\varphi_k\>_2\,
\overline{ \<\psi_2,\varphi_k\>_2 }
\quad\mbox{for all}\quad \psi_1,\psi_2\in H^1_0((-1,1)).
\end{equation}
One can check that since we work with $(-1,1)$, 
$\alpha_n=n^2\pi^2$ and $\varphi_n (u_1)= \sin(n\pi u_1)$, $n\in\mathbb N$.

\begin{proposition}\label{unique:bdd-infty}
For any $T>0$, the system of equations \eqref{f00} has a unique weak solution in the class  $\big(L^\infty \big([0,T] \times \Lambda \big)\big)^3$.
\end{proposition}
\begin{proof}[Proof of Proposition \ref{unique:bdd-infty}]
We follow the arguments in \cite{MM00} adapted to the our case.

Fix $T>0$, define the heat kernel on the time interval $(0,T]$ by the following expression
$$
p_1(t,u_1,v_1)\, =\, \sum_{n\ge 1}e^{-\alpha_n t }\varphi_n(u_1)\overline{\varphi_n(v_1)}\, ,\; t\in [0,T]\, ,\; u_1,v_1\in [-1,1]\, .
$$
Let $g\in \mathcal C_c^0 ((-1,1);\mathbb R)$ and denote by $\delta_{\cdot}$ the Dirac function. 
The heat kernel $p_1$ is such that 
$p_1(0,u_1,v_1)=\delta_{u_1-v_1}$, $p_1\in \mathcal C^{\infty}((0,T]\times (-1,1)\times (-1,1);\mathbb R)$
and the function defined via the convolution operator:
$$
\varphi_1(t,u_1):=(p_1\star g)(t,u_1)=\int_{-1}^1 p_1(t,u_1,v_1)g(v_1) dv_1 
$$ solves the following boundary value problem
\begin{equation}
\label{eq:chal}
\left\{ 
\begin{array}{l}
\partial_t \varphi = \partial_{u_1}^2 \varphi\, , \\ 
\varphi (0 ,\cdot) =\; g (\cdot) \, ,   \\
\varphi (t, \cdot)\in H_0^1((-1,1))\; \; \text {for } 0< t\le T \;.
\end{array}
\right. 
\end{equation}
Let $\check p$ be the heat kernel for 
$(t, \check u,\check v)\in(0,T)\times \mathbb R^{d-1}\times \mathbb R^{d-1}$ 
$$
\check p(t, \check u,\check v)= \big(4\pi t \big)^{-(d-1)/2}
\exp\Bigg\{ -{1 \over 4 t } \sum_{k=2}^d (u_k-v_k)^2 \Bigg\}.
$$
For each function $\check f\in{\mathcal C}_c({\mathbb R}^{d-1};\mathbb R)$, it is known that 
$$
\check {h}_{t}^{\check f}(t,\check u):=(\check p\star \check f)(t,\check u)
 =\int_{\mathbb R^{d-1}}\check p(t, \check u,\check v)\check f( \check v) d \check v\, .
$$
 solves the equation 
$\partial_t \check\rho =\Delta \check\rho$, $\check\rho_0 =f$,
on $(0,t]\times {\mathbb R}^{d-1}$. Moreover 
$\check{h}_{t}^f\in \mathcal C^\infty ((0,T] \times \mathbb R^{d-1};\mathbb R)$.

For 
$t\in (0,T]$, 
$\widehat f=(f_1,f_2,f_3)\in{\mathcal C}_c(\overline {\Lambda}; \mathbb R^3)$ and $\varepsilon >0$ 
small enough, let 
${\mathcal H}_{t,\varepsilon}^{\widehat f} :[0,t]\times \Lambda \longrightarrow \mathbb R$ be defined by
$$
{\mathcal H}_{t,\varepsilon}^{\widehat f}(s,u):= \sum_{i=1}^3
{\mathcal H}_{t,\varepsilon}^{f_i}(s,u):=\sum_{i=1}^3\big( p * f_i \big)(t+\epsilon-s, u),
$$
where $p$ is the heat kernel on $(0,T]\times \Lambda \times \Lambda$ given by 
$$
p(t,u,v)=p_1(t,u_1,v_1) \check p(t, \check u,\check v).
$$
Then ${\mathcal H}_{t,\varepsilon}^f$ solves the equation $\partial_t \rho =\Delta \rho$
on $(0,t]\times {\mathbb R}^d$, $\rho_0=f$.

Consider $\widehat \rho^{(1)}=(\rho^{(1)}_1,\rho^{(1)}_2,\rho^{(1)}_3)$ 
and $\widehat \rho^{(2)}=(\rho^{(2)}_1,\rho^{(2)}_2,\rho^{(2)}_3)$ two weak solutions of 
\eqref{f00} associated to an initial profile 
$\widehat  \gamma =(\gamma_1,\gamma_2,\gamma_3): \Lambda \rightarrow [0,1]^3$. 
Set $\overline{m}_i= \rho^{(1)}_i -\rho^{(2)}_i$, $1\le i\le 3$.
We shall prove below that for any function 
$m(\cdot,\cdot)\in L^\infty ([0,T]\times \Lambda)$ and each $i\le i\le d$,
\begin{equation}\label{ineq-uniq}
\int_0^t ds \left| \int_{\Lambda}   m(s,u)
{\mathcal H}_{t,\varepsilon}^{f_i} (s,v)dv\right| ds
\, \le\,  C_1 t\, \|m\|_\infty\| f_i\|_1,
\end{equation}
for some positive constant $C_1$, where for a trajectory 
$m :[0,t]\times \Lambda \to \mathbb R$, $\|m\|_\infty = \|m\|_{L^\infty([0,t]\times \Lambda)}$
stands for the infinite norm in $L^\infty([0,t]\times \Lambda)$.

On the other hand, from the fact that 
$\rho_i^{(1)}, \rho_i^{(2)}$, $1\le i\le 3$ are in $L^\infty ([0,T]\times \Lambda)$, it follows that
there exists a positive constant $C_2$ such that, 
for almost every $(s,u)\in [0,t]\times \Lambda$, for every $1\le i\le 3$,
\begin{equation*}
\big| F_i(\rho_i^{(1)}(s,u))- F_i(\rho_i^{(2)}(s,u))\big|\le C_2  \sum_{i=1}^3 \|\rho_i^{(1)}-\rho_i^{(2)} \|_\infty\, .
\end{equation*}
Since $\widehat \rho^{(1)}$ and $\widehat \rho^{(2)}$ are two weak solutions of \eqref{f00}, we obtain by
\eqref{ineq-uniq} that for all $0\le \tau\le t$, $1\le i,k\le 3$
\begin{equation*}
 \begin{aligned}
\Big| \big< \overline{m}_i (\tau,.), {\mathcal H}_{\tau,\varepsilon}^{f_k} (\tau,.)\big>\Big|&
=\sum_{i=1}^3 \Big| \int_0^\tau  \big<F_i(\widehat \rho^{(1)}) 
- F_i(\widehat \rho^{(2)}), {\mathcal H}_{\tau,\varepsilon}^{f_k} (\tau,.)\big>\Big|\\
\ & \le C_1' t \, \Big(\sum_{i=1}^3 \|\rho_i^{(1)}-\rho_i^{(2)} \|_\infty\Big)\,  \| f_k\|_1\, ,
 \end{aligned}
\end{equation*}
for $C_1'=C_1C_2$. 

By observing that $p(\varepsilon,\cdot,\cdot)$ is an approximation of the identity in $\varepsilon$,
 we obtain by letting $\varepsilon \downarrow 0$,
\begin{equation}\label{uniq3}
\Big| \big< \overline{m}_i (\tau,.), f_k \big>\Big|
\,\le \, C_1' t\, \Big(\sum_{i=1}^3 \|\rho_i^{(1)}-\rho_i^{(2)} \|_\infty\Big)\, \| f_k\|_1\, .
\end{equation}
We claim that $\overline{m}_i\in L^\infty ([0,t]\times \Lambda)$ and
\begin{equation}\label{uniq30}
 \|\overline{m}_i \|_\infty \, \le \, C_1' \, t\, \Big(\sum_{i=1}^3 \|\rho_i^{(1)}
 -\rho_i^{(2)}\|_\infty \Big)\, .
\end{equation}
Indeed (cf. \cite{Oel85}, \cite{MM00}), denote  by 
$R(t)=\sum_{i=1}^3 \|\rho_i^{(1)}-\rho_i^{(2)} \|_\infty\,$, 
by \eqref{uniq3}, for any open set $U$ of $\Lambda$ with 
finite Lebesgue measure $\lambda(U)$, we have for all $0\le \tau\le t$,
\begin{equation}\label{uniq30a}
\int_{U} \overline{m}_i(\tau,u)du\le C_1'\, t\, R(t) \lambda(U).
\end{equation}
Fix $0< \delta <1$. For any open set $U$ of $\Lambda$ with finite Lebesgue 
measure and for $0\le \tau\le t$ let
$$
B_{\delta,\tau}^U =\Big\{u\in U\; :\;\; \overline{m}_i(\tau,u)>C_1'\, t\, R(t)
(1+\delta)  \Big\}.
$$
Suppose that $\lambda(B_{\delta,\tau}^U)>0$, there exists an open set $V$, such 
that, $B_{\delta,\tau}^U \subset V$ and 
$\lambda\big(V\setminus  B_{\delta,\tau}^U \big)\le \lambda(V){\delta\over 2}$ 
and we have
\begin{equation*}
\begin{aligned}
\lambda(V)\big(C_1' \, t\, R(t) \big) & <
\lambda(V)\big(C_1'\, t\, R(t) \big) (1+\delta)(1-\delta/2) = \big(C_1'\, t\, R(t) \big) (1+\delta)
\big( \lambda(V)-\lambda(V)\delta/2\big)\\
&\le \big(C_1't\, R(t) \big) (1+\delta)\big( \lambda(V)
  - \lambda\big(V\setminus  B_{\delta,\tau}^U \big)\big) =\big(C_1'\sqrt{t}R(t) \big) (1+\delta)  
   \lambda\big( B_{\delta,\tau}^U \big)\\
&< \int_{B_{\delta,\tau}^U} \overline{m}_i(\tau,x)dx\, .
\end{aligned}
\end{equation*}
Thus, from \eqref{uniq30a} and since $B_{\delta,\tau}^U \subset V$, we get
\begin{equation*}
\begin{aligned}
\lambda(V)\big(C_1' \, t\, R(t) \big)  
&< \int_{V} \overline{m}_i(\tau,x)dx
&\le \big(C_1'\, t\, R(t) \big)\lambda(V)\, ,
\end{aligned}
\end{equation*}
which leads to a contradiction.

By the arbitrariness of $0< \delta <1$ we obtain that if $U$ 
is any open set of $\Lambda$ with $\lambda(U)<\infty$,
$$
\lambda \Big(\Big\{u\in U\; :\;\; \overline{m}_i(\tau,u)>C_1'\, t\, R(t) \Big\}\Big)=0.
$$
This implies
$$
\overline{m}_i(\tau,x)\le C_1'\, t\, R(t) \qquad {\rm a.e.}\ {\rm in}\ \Lambda
$$
and concludes the proof of \eqref{uniq30} by the arbitrariness of $\tau\in [0,t]$.

We now turn to the proof of the uniqueness, from \eqref{uniq30},
\begin{equation*}
 \|\overline{m}_i \|_\infty \, \le \, C_1' \, t\, \Big(\sum_{j=1}^3 \|\overline{m_j}\|_\infty \Big)\, ,
\end{equation*}
and then
$$
R(t)\le 3 C_1' \, t\, R(t)\, .
$$
Choosing $t=t_0$ such that $3C_1'\, t_0\, <1$, this gives uniqueness in 
$[0,t_0]\times \Lambda$. To conclude the proof we just have to repeat the same arguments in $[t_0,2t_0]$,
and in each interval $[kt_0,(k+1)t_0]$, $k\in \mathbb N$, $k>1$.
It remains to prove inequality \eqref{ineq-uniq}. {}From Fubini's Theorem, we have
\begin{equation*}
 \begin{aligned}
&\int_0^t \left| \int_{\Lambda} 
m(s,u){\mathcal H}_{t,\varepsilon}^{f_i} (s,u)du\right| ds \\
&\le \int_0^t ds \int_{\mathbb R^{d-1}}d \check v\int_{\mathbb R^{d-1}}d \check u
\Bigg|\sum_{n\ge 1}e^{-n^2\pi^2 (t+\varepsilon-s)} 
\int_{-1}^1 d v_1 \Big\{\sin (n\pi v_1) f_i(v_1,\check v)\Big\} \\
&\qquad\qquad\qquad\qquad\quad \times
\int_{-1}^1 d u_1 \Big\{\sin (n\pi u_1)\check p(t+\varepsilon- s,\check u,\check v) m(s,u_1,\check u)\Big\}
\Bigg|\\
&\le 
\int_0^t ds \int_{\mathbb R^{d-1}}d \check v\int_{\mathbb R^{d-1}}d \check u\; 
\check p(t+\varepsilon- s,\check u,\check v)
\Bigg|\sum_{n\ge 1} \Big<\varphi_n, m(s,(\cdot,\check u))\Big> 
\times \Big<\varphi_n, f_i(\cdot,\check v) \Big>
\Bigg|\\
&\le
\int_0^t ds\int_{\Lambda} du \int_{\Lambda} dv 
\Big\{ |m(s,u)| \, |f_i(v)|\, \check p(t+\varepsilon- s,\check u,\check v) \Big\} \\
& \le\, 4\, t\;  \|m\|_\infty \|f_i\|_1  \, ,
 \end{aligned}
\end{equation*}
where we used the fact that $\check p(s,\cdot,\cdot)$ 
is a probability kernel in $\mathbb R^{d-1}$ for all $s>0$.
\end{proof}
\section{Empirical currents}\label{sec:currents}
In this section, we derive the law of large numbers for the empirical currents 
stated in Proposition \ref{res:lln_currents}. Recall that for $x \in \Lambda_N $, $1\le i\le 3$, 
and $j=1,\ldots,d$, $W_t^{x,x+e_j}(\eta_i)$ stands for the conservative current 
of particles of type $i$ across the edge $\{x,x+e_j\}$,   and $Q_t^x(\eta_i)$  for
the total number of particles of type $i$ created 
minus the total number of particles of type $i$ annihilated at site $x$ before time $t$.  
We have the following families of jump martingales: 
%
\begin{multline}
\widetilde W_t^{x,x+e_j} (\eta_i)= W_t^{x,x+e_j}(\eta_i) 
 -   N^2\int_0^t\Big( \eta_{i,s}(x)(1-\eta_{i,s}(x+e_j))
 - (1-\eta_{i,s}(x))\eta_{i,s}(x+e_j)\Big) ds\label{mart-cur-cons-ij}
 \end{multline}
 with quadratic variation (because $J^{x,x+e_j}_t(\eta_i) $ 
 and $J^{x+e_j,x}_t(\eta_i)$ have no common jump)   
\begin{eqnarray}\label{quadvar-mart-cur-cons-ij}
\langle \widetilde W^{x,x+e_j}(\eta_i) \rangle_t  
&=& N^2 \int_0^t\Big( \eta_{i,s}(x)(1-\eta_{i,s}(x+e_j))
 + (1-\eta_{i,s}(x))\eta_{i,s}(x+e_j)\Big) ds
 \end{eqnarray} 
and, for  $\widehat f = (f_1,f_2,f_3) : \ws_N \rightarrow \mathbb R^3$ 
defined in \eqref{eq:fmicro}, 
\begin{equation}\label{mart-cur-non-cons-i}
\widetilde Q_t^x (\eta_i) = Q_t^x (\eta_i) - \int_0^t  \tau_x   f_i (\xi_s,\omega_s)  ds 
\end{equation}
with quadratic variations
\begin{equation}\label{quadvar-mart-cur-non-cons-i}
\begin{cases}
\langle \widetilde Q^x (\eta_1) \rangle_t  &= \displaystyle{\int_0^t
\tau_x\Big( \beta_N(0,\xi_s,\omega_s)  \eta_{0,s}(0) + \eta_{3,s}(0) + (r+1)\eta_{1,s}(0)\Big) ds},\\  
\langle \widetilde Q^x (\eta_2) \rangle_t  &= \displaystyle{\int_0^t
\tau_x\Big( r \eta_{0,s}(0) + \eta_{3,s}(0) 
+\beta_N(0,\xi_s,\omega_s) \eta_{2,s}(0) + \eta_{2,s}(0)\Big) ds},\\
\langle \widetilde Q^x (\eta_3) \rangle_t  &= \displaystyle{\int_0^t
\tau_x\Big( \beta_N(0,\xi_s,\omega_s) \eta_{2,s}(0) + r \eta_{1,s}(0) + 2 \eta_{3,s} (0)\Big) ds}.
\end{cases}
\end{equation}
\begin{proof}[Proof of Proposition \ref{res:lln_currents}]
Given a smooth continuous vector field $\mathbf G =(G_1,..., G_d) 
\in \mathcal C_c^\infty(\Lambda,\mathbb R^d)$, after definition \eqref{empcurr}, 
sum the   martingale \eqref{mart-cur-cons-ij} 
 over $\{x,x+e_j \in \Lambda_N\}$  to get the  martingale $\widetilde {\mathbf M}_t^G$, given by
\begin{eqnarray*}
\widetilde {\mathbf M}_t^G (\eta_i) & = & \sum\limits_{j=1}^d 
\Bigg( \langle W_{j,t}^N(\eta_i), G_j \rangle  - \dfrac{N^2}{N^{d+1}} \sum\limits_{x,x+e_j \in \Lambda_N} 
\int_0^t G_j(x/N) \Big( \eta_{i,s} (x) - \eta_{i,s}(x+e_j) \Big) ds \Bigg) 
\\
  & = &  \langle \mathbb W_t^N(\eta_i), \mathbf G \rangle 
- \dfrac{1}{N^{d}} \sum\limits_{j=1}^d \sum\limits_{x \in \Lambda_N} 
\int_0^t \partial_{x_j} G_j(x/N) \eta_{i,s} (x) ds   +O(N^{-1})\\
& = &   \langle \mathbb W_t^N(\eta_i), \mathbf G \rangle 
- \sum\limits_{j=1}^d \langle \pi_s^{N,i}, \partial_{x_j} G_j 
 \rangle  +O(N^{-1}),
\end{eqnarray*}
where we did a Taylor expansion.
Relying on \eqref{quadvar-mart-cur-cons-ij},
 the expectation of $\langle \widetilde{\mathbf M}^G(\eta_i) \rangle_t$
vanishes when $N \rightarrow \infty$, so that by Doob's martingale inequality, for any $\delta >0$,
$$\lim_{N\to\infty} \mathbb P_{\mu_N}^{N,\widehat b} \Big[\sup\limits_{0 \leq t \leq T} 
\big| 
\sum_{i=1}^3\widetilde {\mathbf M}_t^G(\eta_i)\big| > \delta \Big ] \;=\; 0.$$
Using that the empirical density  $\widehat \pi^N$  converges 
towards the solution of  \eqref{f00}, this concludes  the law of large numbers 
 \eqref{res:lln_currents_1} for the current $\mathbb W_T^N$. \\ \\
Fix a smooth vector field $\widehat H =(H_1,H_2,H_3) \in \mathcal C_c^\infty (\Lambda,\mathbb R^3)$.  
Sum   \eqref{mart-cur-non-cons-i} over $x \in \Lambda_N$ to get the martingale 
$$\widetilde {\mathbf N}_t^H (\eta_i)= \langle Q_t^N (\eta_i), H_i \rangle 
- \dfrac{1}{N^d} \sum\limits_{x\in\Lambda_N} \int_0^t H_i(x/N) \tau_x  f_i (\xi_s,\omega_s) ds$$
Relying on \eqref{quadvar-mart-cur-non-cons-i}, 
the expectation of its quadratic variation vanishes as $N \rightarrow \infty$ as well. 
Use the replacement lemma \ref{replacement} to express $\widetilde {\mathbf N}_t^H (\eta_i)$ 
with functionals of the density fields and conclude \eqref{res:lln_currents_2} 
by Doob's martingale inequality: For any $\delta >0$,
$$\lim_{N\to\infty} \mathbb P_{\mu_N}^{N,\widehat b} \Big[\sup\limits_{0 \leq t \leq T} \big| 
\sum_{i=1}^3 \widetilde {\mathbf N}_t^H(\eta_i)
\big| > \delta \Big ] \;=\; 0.$$
\end{proof}
%
 \section*{Acknowledgements} \noindent
We thank the anonymous referees for their careful reading of
the manuscript and their comments. 
%
%
 
%

\begin{thebibliography}{99} 
%
%
\bibitem{BDGJL06} Bertini L., De Sole A., Gabrielli D., Jona-Lasinio G., Landim C.: 
{Large deviations of the empirical current in interacting particle systems}.
\emph{Teor. Veroyatn. Primen.}, {\bf 51}, no. 1 (2006), 144--170.
%
\bibitem{blm09} Bertini L., Landim  C., Mourragui  M.: {Dynamical large deviations 
for the boundary driven weakly asymmetric exclusion process.}
\emph{Ann. Probab.},  {\bf 37}, no. 6 (2009), 2357--2403.
%
\bibitem{bd} Bodineau T., Derrida B.: Current large deviations for
  asymmetric exclusion processes with open boundaries.  
 \emph{J. Stat. Phys.}, {\bf 123}, no. 2, (2006), 277--300. 
%
\bibitem{BL12} Bodineau T., Lagouge M.: {Large deviations of the empirical currents 
for a boundary driven reaction diffusion model}.
\emph{Ann. Appl. Probab.}, {\bf 22}, no. 6 (2012), 2282–-2319.
%
\bibitem{bo}  Borrello, D.: Stochastic order and attractiveness for
 particle systems with multiple births, deaths and jumps.
\emph{Electron. J. Probab.}.  {\bf 16}, no. 4 (2011), 106--151.
%
\bibitem{dfl1} De Masi A., Ferrari P., Lebowitz J.:
Rigorous Derivation of Reaction-Diffusion Equations with
Fluctuations,\textit{ Phys. Rev. Lett.}, \textbf{55}, no. 3, (1985), 1947--1949.
%
\bibitem{dfl2} De Masi A., Ferrari P., Lebowitz J.:
 Reaction-Diffusion Equations for Interacting
Particle Systems, \textit{J. Stat. Phys.}, \textbf{44} (1986), 589--644.
%
\bibitem{DHB05}  Dyck, V.A., Hendrichs, J., Robinson, A.S.:  \textit{
Sterile insect technique: Principles and Practice in area-wide integrated pest management.}
The Wadsworth \& Brooks/Cole Statistics/Probability Series, Springer, 2005.
%
\bibitem{els1} Eyink G., Lebowitz J. L., Spohn H.: Hydrodynamics of
  stationary nonequilibrium states for some lattice gas models. 
  {\it Commun. Math. Phys.}, {\bf 132}, no. 1, (1990),  252--283.
 %
\bibitem{els2} Eyink G., Lebowitz J. L., Spohn H.: Lattice Gas
  Models in Contact with Stochastic Reservoirs: Local Equilibrium and
  Relaxation to the Steady State. {\it Comm. Math. Phys.},  {\bf
    140} (1991), 119--131.
    %
\bibitem{fpv} Ferrari P. A., Presutti E., Vares M. E.: Nonequilibrium
fluctuations for a zero range process. 
{\it Ann. Inst. Henri. Poincar\'e, Probab. Stat.} {\bf 24}, no. 2,  (1988), 237--268.
%
\bibitem{fr1a} Fritz, J.: On the hydrodynamic limit of a one dimensional Ginzburg-Landau lattice model:
the a priori bounds, {\it J. Stat. Phys.}, {\bf 47}, no. 3, (1987), 551--572.
%
\bibitem{fr1b} Fritz, J.: On the hydrodynamic limit of a Ginzburg-Landau lattice model, 
{\it Prob. Th. Rel. Fields}, {\bf 81}, no. 2, (1989), 291--318.
%
\bibitem{fr1}
Fritz, J.: {On the diffusive nature of the entropy flow in infinite systems:
remarks to a paper by Guo-Papanicolau-Varadhan.} {\it Comm. Math. Phys}.
{\bf 133}, no. 2, (1990), 331--352. 
%
\bibitem{gs} Gobron  T., Saada E.: 
Couplings, attractiveness and hydrodynamics for conser-
vative particle systems. 
\emph{Ann. Inst. Henri Poincar\'e Probab. Stat.}.  {\bf 46}, no. 4 (2010), 1132--1177.
%
\bibitem{gpv} Guo M.Z., Papanicolaou G., Varadhan S.R.S.: 
Nonlinear diffusion limit for a system with nearest neighbor interactions. 
\emph{ Comm. Math. Phys}.  {\bf 118}, no. 1, (1988), 31--59.
%
\bibitem{jlv} Jona-Lasinio G., Landim C., Vares M. E.: Large deviations 
for a reaction diffusion model, \textit{Probab.
Theory Related Fields}, \textbf{97}, no. 3 (1993), 339--361.
%
\bibitem{kov} Kipnis C., Olla S., Varadhan S.R.S.: Hydrodynamics
 and large deviations for simple exclusion processes.  
\emph{Commun.  Pure Appl. Math.}, {\bf 42}  (1989), 115--137. 
%
\bibitem{kl} Kipnis C., Landim C.: \textit{Scaling limits of interacting
particle systems.} Springer-Verlag, Berlin, New York, 1999.
%
\bibitem{Kni55} Knipling, E.: Possibilities of insect control or eradication 
through the use of sexually sterile males. \emph{J. Econ. Entomol.},  {\bf 48}, no. 4, (1955),  459--462.
%
\bibitem{phd} Kuoch, K.: {Ph.{D}. manuscript.} {\url{www.theses.fr/2014PA05S020}}, (2014).
%
\bibitem{k} Kuoch, K.: {Phase transition for a contact process with random slowdowns.} 
\emph{Markov Proc. Related Fields}, {\bf 22}, no. 1, (2016), 53--82.
%
\bibitem{lm} Landim, C., Mourragui, M.:
Hydrodynamic limit of mean zero asymmetric zero range processes in infinite volume.
\emph{Ann. Inst. Henri. Poincar\'e, Probab. Stat.} \textbf{ 33}, no. 1, (1997), 361--385.
%
 \bibitem{L} Liggett T.M.: \textit{Interacting Particle Systems.} 
Classics in Mathematics (Reprint of first edition),  
Springer-Verlag, New York, 2005. 
%
\bibitem{MM00} Marra R., Mourragui M.: Phase segregation dynamics for 
the Blume-Capel model with Kac interaction.
\emph{Stoch. Proc. Appl.}, {\bf 88}, no. 1, (2000), 79--124.
%
\bibitem{Oel85} Oelschlager, K.: 
A law of large numbers for moderately interacting diffusion processes.
\emph{Z. Wahrsch. Verw. Gebiete}, {\bf 69}, no.2, (1985), 279--322.
%
\bibitem{sp1} Spohn, H.:
\textit{Large Scale Dynamics of Interacting particles.}
Springer-Verlag, Berlin, 1991.
%
\bibitem{y1} Yau, H. T.: {Metastability of Ginzburg-Landau model with
a conservation law,} {\it J. Stat. Phys}. {\bf 74}, nos. 3/4, (1994), 705--742.
%
\bibitem{zui}  Zuily C.: \textit{ \'El\'ements de distributions et 
d'\'equations aux d\'eriv\'ees partielles. Cours et probl\`emes r\'esolus.}
Collection Sciences Sup, Editions Dunod, Paris, 2002.
%
\end{thebibliography}
\end{document}